\documentclass[11pt]{amsart}

\usepackage{amsmath}
\usepackage{amssymb}
\usepackage{bbm}
\usepackage{pdfsync}
\usepackage{esint} 
\usepackage[breaklinks=true]{hyperref}
\usepackage[utf8]{inputenc}
\usepackage[utf8]{inputenc}
\usepackage[T1]{fontenc}

\newcommand{\R}{{\mathbb R}}       
\newcommand{\N}{{\mathbb N}}

\newcommand{\Z}{{\mathbb Z}}       
\newcommand{\DD}{{\mathcal D}}

\newcommand{\HH}{{\mathcal H}}
\newcommand{\LL}{{\mathcal L}}
\newcommand{\PP}{{\mathcal P}}
\newcommand{\QQ}{{\mathcal Q}}

\newcommand{\TT}{{\mathcal T}}

\newcommand{\RR}{{\mathcal R}}
\newcommand{\CH}{{\mathcal Ch}}
\newcommand{\EE}{{\mathcal E}}

\newcommand{\cI}{{\mathcal I}}
\newcommand{\cJ}{{\mathcal J}}

\newcommand{\diam}{{\rm diam}}
\newcommand{\dist}{{\rm dist}}

\newcommand{\rf}[1]{{(\ref{#1})}}

\newcommand{\supp}{\operatorname{supp}}

\newcommand{\vphi}{{\varphi}}
\newcommand{\ve}{{\varepsilon}}
\newcommand{\vv}{{\vspace{2mm}}}
\newcommand{\vvv}{{\vspace{3mm}}}
\newcommand{\wt}[1]{{\widetilde{#1}}}
\newcommand{\wh}[1]{{\widehat{#1}}}

\newcommand{\noi}{\noindent}
\newcommand{\rest}{{\lfloor}}

\newcommand{\sss}{{\mathsf {Stop}}}

\newcommand{\DB}{{\mathsf {DB}}}
\newcommand{\Gen}{{\mathsf {Gen}}}
\newcommand{\tree}{{\rm Tree}}

\newcommand{\pv}{\operatorname{pv}}

\newcommand{\HE}{{\mathsf {HE}}}
\newcommand{\bad}{{\mathsf{Bad}}}

\newcommand{\HD}{{\mathsf{HD}}}
\newcommand{\hd}{{\mathsf{hd}}}
\newcommand{\GH}{{\mathsf{GH}}}
\newcommand{\Ot}{{\mathsf{Ot}}}

\newcommand{\LD}{{\mathsf{LD}}}

\newcommand{\MDW}{{\mathsf{MDW}}}

\newcommand{\Trc}{{\mathsf{Trc}}}
\newcommand{\sL}{{\mathsf{L}}}
\newcommand{\sD}{{\mathsf{D}}}

\newcommand{\sF}{{\mathsf{F}}}

\newcommand{\sM}{{\mathsf{M}}}

\newcommand{\Neg}{{\mathsf{Neg}}}

\newcommand{\End}{{\mathsf{End}}}
\newcommand{\Reg}{{\mathsf{Reg}}}
\newcommand{\NDB}{{\mathsf{NDB}}}
\newcommand{\Ty}{{\mathsf{Ty}}}

\newcommand{\GDF}{{\mathsf{GDF}}}

\def\XXint#1#2#3{{\setbox0=\hbox{$#1{#2#3}{\int}$ }
\vcenter{\hbox{$#2#3$ }}\kern-.58\wd0}}

\usepackage{pgf,tikz} 

\definecolor{ffffff}{rgb}{1.0,1.0,1.0}
\definecolor{qqqqff}{rgb}{0.0,0.0,1.0}
\definecolor{ffqqqq}{rgb}{1.0,0.0,0.0}
\definecolor{zzzzqq}{rgb}{0.6,0.6,0.0}
\definecolor{marronet}{rgb}{0.6,0.2,0}
\definecolor{negre}{rgb}{0,0,0}
\definecolor{vermell}{rgb}{0.8,0.05,0.05}
\definecolor{blau}{rgb}{0.3,0.2,1.}
\definecolor{blauclar}{rgb}{0.,0.,1.}
\definecolor{grisfosc}{rgb}{0.25098039215686274,0.25098039215686274,0.25098039215686274}
\definecolor{verd}{rgb}{0.1,0.6,0.1}
\definecolor{taronja}{rgb}{0.9,0.6,0.05}
\definecolor{vermellclar}{rgb}{1.,0.,0.}
\definecolor{verdet}{rgb}{0,0.8,0.1}
\definecolor{blauverd}{rgb}{0,0.4,0.2}
\definecolor{grisclar}{rgb}{0.6274509803921569,0.6274509803921569,0.6274509803921569}

\newtheorem{theorem}{Theorem}[section]
\newtheorem{lemma}[theorem]{Lemma}

\newtheorem{coro}[theorem]{Corollary}

\newtheorem{mpropo}[theorem]{Main Proposition}
\newtheorem{claim}[theorem]{Claim}
\newtheorem*{claim*}{Claim}

\newtheorem*{theorem*}{Theorem}

\theoremstyle{definition}

\theoremstyle{remark}
\newtheorem{rem}[theorem]{\bf Remark}

\numberwithin{equation}{section}

\newcommand{\brem}{\begin{rem}}
\newcommand{\erem}{\end{rem}}

\begin{document}


\title[The Riesz transforms and the Painlev\'e problem]{The measures with $L^2$-bounded Riesz transform and the Painlev\'e problem for Lipschitz harmonic functions}

\author{Xavier Tolsa}

\address{Xavier Tolsa
\\
ICREA, Passeig Llu\'{\i}s Companys 23 08010 Barcelona, Catalonia;\\
Universitat Aut\`onoma de Barcelona 
\\
08193 Bellaterra, Catalonia; and Centre de Recerca Matem\`atica,
08193 Bellaterra, Catalonia.
}

\address{Xavier Tolsa \\
ICREA, Passeig Llu\'{\i}s Companys 23 08010 Barcelona, Catalonia;  
Universitat Aut\`onoma de Barcelona and Centre de Recerca Matem\`atica,
08193 Bellaterra, Catalonia.
}

\thanks{Supported by 2017-SGR-0395 (Catalonia), MTM-2016-77635-P (MINECO, Spain), and
by
 the European Research Council (ERC) under the European Union's Horizon 2020 research and innovation programme (grant agreement 101018680).
}

\subjclass{42B20, 28A75, 49Q15}

\begin{abstract}
This work provides a geometric characterization of the measures $\mu$ in $\R^{n+1}$ with polynomial upper growth of degree $n$ such that the $n$-dimensional Riesz transform $\RR\mu (x) = \int \frac{x-y}{|x-y|^{n+1}}\,d\mu(y)$ belongs to $L^2(\mu)$. More precisely, it is shown that
$$\|\RR\mu\|_{L^2(\mu)}^2 + \|\mu\|\approx \int\!\!\int_0^\infty \beta_{2,\mu}(x,r)^2\,\frac{\mu(B(x,r))}{r^n}\,\frac{dr}r\,d\mu(x) +
\|\mu\|,$$
where $\beta_{\mu,2}(x,r)^2 = \inf_L \frac1{r^n}\int_{B(x,r)} \left(\frac{\dist(y,L)}r\right)^2\,d\mu(y),$
with the infimum taken over all affine $n$-planes $L\subset\R^{n+1}$. 
As a corollary, one obtains a characterization of the removable sets for Lipschitz
harmonic functions in terms of a metric-geometric potential and one deduces that the class of removable sets for Lipschitz harmonic functions is invariant by bilipschitz mappings.

\end{abstract}

\maketitle

\tableofcontents

\section{Introduction}

Given a Radon measure $\mu$ in $\R^{n+1}$, its 
($n$-dimensional) Riesz transform at $x\in\R^{n+1}$ is defined by
$$\RR\mu (x) = \int \frac{x-y}{|x-y|^{n+1}}\,d\mu(y),$$
whenever the integral makes sense. For $f\in L^1_{loc}(\mu)$,
one writes $\RR_\mu f(x) = \RR(f\mu)(x)$.
Given $\ve>0$, the $\ve$-truncated Riesz transform of $\mu$ equals
$$\RR_\ve\mu (x) = \int_{|x-y|>\ve} \frac{x-y}{|x-y|^{n+1}}\,d\mu(y),$$
and the operator $\RR_{\mu,\ve}$ is defined by
 $\RR_{\mu,\ve}f(x) = \RR_\ve(f\mu)(x)$. 
 
 We say that $\RR_\mu$ is bounded in $L^2(\mu)$ if
the operators $\RR_{\mu,\ve}$ are bounded uniformly in $L^2(\mu)$ uniformly on $\ve$, and then we denote
$$\|\RR_\mu\|_{L^2(\mu)\to L^2(\mu)} = \sup_{\ve>0} \|\RR_{\mu,\ve}\|_{L^2(\mu)\to L^2(\mu)}.$$
We also write
$$\RR_*\mu(x) = \sup_{\ve>0} |\RR_{\ve}\mu(x)|, \qquad  \pv\RR\mu(x) = \lim_{\ve>0} \RR_{\ve}\mu(x),$$
in case that the latter limit exists. Remark that, sometimes, abusing notation we will
write $\RR\mu$ instead of $\pv\RR\mu$.

This paper provides a full geometric description of the measures $\mu$ with no point masses such that $\RR_\mu$ is bounded in
$L^2(\mu)$. 
In the case $n=1$, such description has already been obtained (see \cite{MV}, \cite{Leger}, \cite{Tolsa-duke}), relying on the connection between Menger curvature and the Cauchy kernel
found by Melnikov \cite{Melnikov}. In higher dimensions, a similar connection is missing,
and thus the obtention of analogous results presents major difficulties. 
In the case when the measure $\mu$ is AD-regular (i.e., Ahlfors-David regular) that geometric description is equivalent to the 
codimension $1$ David-Semmes problem, solved by Nazarov, the author of the current paper, and Volberg
in \cite{NToV1}. Recall that a measure $\mu$ is AD-regular (or $n$-AD-regular) if there exists a constant $C>0$ such that
$$C^{-1}\,r^n\leq \mu(B(x,r))\leq C\,r^n\quad \mbox{ for all $x\in\supp\mu$ and $0<r\leq \diam(\supp\mu)$.}$$
One of the main motivations for the description of the measures $\mu$ such that $\RR_\mu$ is bounded in $L^2(\mu)$
is the characterization of the removable singularities for Lipschitz harmonic functions. Also, one may expect other
applications regarding the study of harmonic and elliptic measures. Indeed, in some of the recent advances on this topic, the connection between harmonic measure, the Riesz transform, and rectifiability has played an essential role (see \cite{AHM3TV}, \cite{AMT}, and \cite{AMTV}, for example).

Next we need to introduce additional notation. For a ball $B\subset \R^{n+1}$, we consider its $n$-dimensional density (with respect to $\mu$):
$$\theta_\mu(B)= \frac{\mu(B)}{r(B)^n},$$
and its $\beta_{2,\mu}$ coefficient:
$$\beta_{2,\mu}(B) = \inf_L \left(\frac1{r(B)^n}\int_B \left(\frac{\dist(x,L)}{r(B)}\right)^2\,d\mu(x)\right)^{1/2},$$
where the infimum is taken over all $n$-planes $L\subset\R^{n+1}$ and $r(B)$ stands for the radius of $B$. For
$B=B(x,r)$ we may also write $\theta_\mu(x,r)$ and $\beta_{2,\mu}(x,r)$ instead of $\theta_\mu(B)$ and 
$\beta_{2,\mu}(B)$. The coefficients $\beta_{2,\mu}$ were introduced by David and Semmes in their fundamental works \cite{DS1}, \cite{DS2} on uniform rectifiability. They can be considered as $L^2$ variants of 
some analogous coefficients considered previously by Peter Jones in his celebrated travelling salesman theorem \cite{Jones}.

The first main result of this paper is the following.

\begin{theorem}\label{teomain1}
Let $\mu$ be a Radon measure in $\R^{n+1}$ satisfying the polynomial growth condition
\begin{equation}\label{eqgrow00}
\mu(B(x,r))\leq \theta_0\,r^n\quad \mbox{ for all $x\in\supp\mu$ and all $r>0$}
\end{equation}
and such that $\RR_*\mu(x)<\infty$ $\mu$-a.e.
Then
\begin{equation}\label{eqbetawolff}
\int\!\!\int_0^\infty \beta_{2,\mu}(x,r)^2\,\theta_\mu(x,r)\,\frac{dr}r\,d\mu(x)\leq C\,(\big\|\pv\RR\mu\|_{L^2(\mu)}^2
+\theta_0^2\,\|\mu\|\big),
\end{equation}
where $C$ is an absolute constant.
\end{theorem}

Let us remark that the growth condition \rf{eqgrow00} and the assumption that $\RR_*\mu(x)<\infty$ $\mu$-a.e.\ imply the existence of principal values $\pv\RR\mu(x)$ $\mu$-a.e., by \cite{NToV2}, and so
$\big\|\pv\RR\mu\|_{L^2(\mu)}$ is well defined.

 A converse to the estimate \rf{eqbetawolff} also holds: if $\mu$ satisfies the growth condition
\rf{eqgrow00}, then
\begin{equation}\label{eqbetawolff'}
\|\pv\RR\mu\|_{L^2(\mu)}^2\leq C\,\int\!\!\int_0^\infty \beta_{2,\mu}(x,r)^2\,\theta_\mu(x,r)\,\frac{dr}r\,d\mu(x) 
+C\,\theta_0^2\,\|\mu\|,
\end{equation}
where $C$ is an absolute constant. This was shown in \cite{Azzam-Tolsa} in the case $n=1$, and in \cite{Girela} in full generality.

From \rf{eqbetawolff'}, Theorem \ref{teomain1}, and a direct application of the $T1$ theorem for non-doubling measures (\cite{NTrV1}, \cite{NTrV2}) one deduces the following.

\begin{theorem}\label{teomain2}
Let $\mu$ be a Radon measure in $\R^{n+1}$ with no point masses. Then $\RR_\mu$ is bounded in $L^2(\mu)$ if and
only if it satisfies the polynomial growth condition
\begin{equation}\label{eqgrow01}
\mu(B(x,r))\leq C\,r^n\quad \mbox{ for all $x\in\supp\mu$ and all $r>0$}
\end{equation}
and
\begin{equation}\label{eqbetawolff2}
\int_B\int_0^{r(B)} \beta_{2,\mu}(x,r)^2\,\theta_\mu(x,r)\,\frac{dr}r\,d\mu(x)\leq C^2\,\mu(B)\quad\mbox{ for any ball
$B\subset\R^{n+1}$.}
\end{equation}
Further, the optimal constant $C$ is comparable to $\|\RR_\mu\|_{L^2(\mu)\to L^2(\mu)}$.
\end{theorem}

In the case $n=1$ the preceding results are already known. They were proven in \cite{Azzam-Tolsa}, relying on the
corona decomposition involving the curvature of $\mu$ from \cite{Tolsa-bilip}. 
Further, for arbitrary $n>1$, 
Theorems \ref{teomain1} and \ref{teomain2} were proven recently by D\k{a}browski and
the author \cite{DT} for a special class of measures $\mu$ with an appropriate Wolff type energy satisfying some scale invariant 
estimates. In fact, the proof of Theorems \ref{teomain1} and \ref{teomain2} in full generality in the current paper relies heavily on the results from \cite{DT}. 

It is also known that if one assumes that there exists a sufficiently large class of $L^2(\mu)$ bounded singular integral operators with an odd Calder\'on-Zygmund kernel, then the condition \rf{eqbetawolff2}
 holds, in any codimension (i.e., assuming that, instead of $\R^{n+1}$, the ambient space is $\R^d$, with $d\geq n$). This was proved in \cite{JNT}.

Observe that, when $\mu$
is $n$-AD-regular, $\theta_\mu(x,r)\approx1$ for all $x\in\supp\mu$ and $0<r\leq \diam(\supp\mu)$, and so
the condition \rf{eqbetawolff2} is equivalent to the uniform $n$-rectifiability of $\mu$, by \cite{DS1}.
 So one deduces that the $L^2(\mu)$ boundedness
of $\RR_\mu$ implies the uniform $n$-rectifiability of $\mu$ and then one recovers the solution of the David-Semmes
problem from \cite{NToV1}. Recall that a measure $\mu$ in $\R^d$ is called {\em uniformly $n$-rectifiable} (UR) if it is $n$-AD-regular and
there exist constants $\kappa, M >0$ such that for all $x \in \supp\mu$ and all $0<r\leq \diam(\supp\mu)$ 
there is a Lipschitz mapping $g$ from the ball $B_n(0,r)\subset\R^{n}$ to $\R^d$ with $\text{Lip}(g) \leq M$ such that
$$
\mu(B(x,r)\cap g(B_{n}(0,r)))\geq \kappa\, r^{n}.$$

It is worth comparing Theorem \ref{teomain2} with a related result obtained in \cite{JNRT} in connection
with the fractional Riesz transform $\RR^s$ associated with the kernel
$x/|x|^{s+1}$ for $s\in (n,n+1)$. The precise result, which involves the $s$-dimensional density $\theta_\mu^s(x,r) = \mu(B(x,r))r^{-s}$, is the following.

\begin{theorem*}[\cite{JNRT}]
Let $\mu$ be a Radon measure in $\R^{n+1}$ with no point masses and let $s\in(n,n+1)$. Then $\RR_\mu^s$ is bounded in $L^2(\mu)$ if and
only if \begin{equation}\label{eqbetawolff3}
\int_B\int_0^{r(B)} \theta^s_\mu(x,r)^2\,\frac{dr}r\,d\mu(x)\leq C^2\,\mu(B)\quad\mbox{ for any ball
$B\subset\R^{n+1}$.}
\end{equation}
Further, the optimal constant $C$ is comparable to $\|\RR_\mu^s\|_{L^2(\mu)\to L^2(\mu)}$.
\end{theorem*}

In the preceding theorem one does not need to ask any growth condition analogous to \rf{eqgrow01}
(with $n$ interchanged with $s$) because this condition is already implied by \rf{eqbetawolff3}. Observe that in
\rf{eqbetawolff3} the density $\theta^s_\mu(x,r)$ replaces $\beta_{2,\mu}(x,r)^2$ in \rf{eqbetawolff2}, which scales similarly to
$\theta^s(x,r)$ when $n=s$. On the other hand, the proof of the last theorem in \cite{JNRT} makes an
extensive use of blowup techniques, which essentially rely on the fact that any measure $\mu$ satisfying the growth condition $\mu(B(x,r))\leq r^s$ for all $x\in\R^{n+1}$, $r>0$, and such that 
$\RR^s\mu=0$ in a suitable $BMO(\mu)$ sense, must be the zero measure (see \cite{JN1}, \cite{JN2}). In 
the codimension $1$ case, one might expect that if $\mu$ satisfies \rf{eqgrow01}
 and $\RR\mu=0$ in the $BMO(\mu)$ sense, then $\mu=c\HH^n|_L$ for some $n$-plane $L$. However, this is 
 still an open problem. If this were known to be true, probably in the present paper (and in \cite{DT}) we
 could use blowup arguments analogous to the ones in \cite{JNRT}.

As shown in \cite{Azzam-Tolsa}, the finiteness of the double integral on the left hand side of \rf{eqbetawolff}
is equivalent to the existence of a suitable corona decomposition for $\mu$ satisfying an appropriate packing
condition. This condition is stable by bilipschitz maps (see also \cite{Girela} for more details). So we get the
following corollary.

\begin{coro}\label{coro1}
Let $\mu$ be a Radon measure in $\R^{n+1}$ with no point masses. Let $\vphi:\R^{n+1}\to\R^{n+1}$ be a bilipschitz
map. Let $\sigma=\vphi\#\mu$ be the image measure of $\mu$ by $\vphi$.
If $\RR_\mu$ is bounded in $L^2(\mu)$, then $\RR_\sigma$ is bounded in $L^2(\sigma)$. Further,
$$\|\RR_\sigma\|_{L^2(\sigma)\to L^2(\sigma)}\leq C\,\|\RR_\mu\|_{L^2(\mu)\to L^2(\mu)},$$
where $C$ depends only on the bilipschitz constant of $\vphi$.
\end{coro}

Remark that, up to now, the preceding result was not  known even for the case of invertible affine maps such as the one defined by
$$\vphi(x_1,\ldots,x_{n+1}) = (2x_1,x_2,\ldots,x_{n+1}).$$

As shown in \cite{Girela}, the conditions \rf{eqgrow01} and \rf{eqbetawolff2} imply the $L^2(\mu)$ boundedness of any singular integral operator of the form
$$T_\mu f(x) =\int K(x-y)\,f(y)\,d\mu(y),$$
where $K$ is an odd kernel such that
\begin{equation}
\label{eqKernel}
|\nabla^j K(x)|\lesssim \frac1{|x|^{n+j}}\quad \mbox{ for all $x\neq0$ and $0\leq j\leq 2$}
\end{equation}
(remark that $T_\mu$ is said to be bounded in $L^2(\mu)$ is the truncated operators $T_{\mu,\ve}$, defined analogously to $\RR_{\mu,\ve}$, are bounded in $L^2(\mu)$ uniformly on $\ve>0$).
Then we deduce the following.

\begin{coro}\label{coro1.5}
Let $\mu$ be a Radon measure in $\R^{n+1}$ with no point masses. Let $T_\mu$ be a singular integral 
operator associated with an odd kernel $K$ satisfying \rf{eqKernel}.
If $\RR_\mu$ is bounded in $L^2(\mu)$, then $T_\mu$ is also bounded in $L^2(\mu)$. Further,
$$\|T_\mu\|_{L^2(\mu)\to L^2(\mu)}\leq C\,\|\RR_\mu\|_{L^2(\mu)\to L^2(\mu)},$$
where $C$ depends just on $n$ and the implicit constants in \rf{eqKernel}.
\end{coro}

Lipschitz
We turn now to the applications of the results above to Lipschitz harmonic functions and Lipschitz harmonic
capacity. Given a compact set $E\subset \R^{n+1}$, one says that $E$ is removable for Lipschitz harmonic functions 
if for any open set $\Omega\supset E$, any function $f:\Omega\to\R$ which is Lipschitz in $\Omega$ and harmonic in
$\Omega\setminus E$
can be extended in a harmonic way to the whole $\Omega$. To study this problem and some related questions on approximation by Lipschitz harmonic functions it is useful to introduce the Lipschitz harmonic capacity $\kappa$
(see \cite{Paramonov} and \cite{Mattila-Paramonov}). This is defined by
$$\kappa(E) = \sup |\langle \Delta f,1\rangle|,$$
where the supremum is taken over all Lipschitz functions $f:\R^{n+1}\to\R$ which are harmonic in $\R^{n+1}\setminus E$ and satisfy $\|\nabla f\|_\infty\leq1$, with $\Delta f$ understood in the sense of distributions. It turns
out that $E$ is removable for Lipschitz harmonic functions if and only if $\kappa(E)=0$.

Extending previous results for analytic capacity from \cite{Tolsa-sem},  Volberg showed in \cite{Volberg} that
$$\kappa(E) \approx \sup \mu(E),$$
where the supremum is taken over all measures $\mu$ satisfying the polynomial growth condition 
\rf{eqgrow01} with constant $C=1$ and such that $\|\RR_\mu\|_{L^2(\mu)\to L^2(\mu)}\leq 1$.
Combining this result with Theorem \ref{teomain2}, we obtain:

\begin{coro}\label{coro2}
Let $E\subset\R^{n+1}$ be compact. Then
$$\kappa(E)\approx \mu(E),$$
where the supremum is taken over all Radon measures $\mu$ such that
$$
\mu(B(x,r))\leq r^n\quad \mbox{ for all $x\in\supp\mu$ and all $r>0$}
$$
and
$$
\int\!\!\int_0^\infty \beta_{2,\mu}(x,r)^2\,\theta_\mu(x,r)\,\frac{dr}r\,d\mu(x)\leq \mu(E).
$$
\end{coro}

To derive this corollary, remark that if $\mu$ satisfies the conditions above, by Chebyshev one deduces that there is a big piece
$F\subset E\cap\supp\mu$, with $\mu(F)\approx\mu(E)$, such that the measure $\wt\mu = \mu|_F$ satisfies \rf{eqbetawolff2}, and so
$\RR_{\wt \mu}$ is bounded in $L^2(\wt \mu)$. Hence $\kappa(E)\gtrsim \mu(F)\approx\mu(E)$. The converse
direction of the corollary is a straightforward consequence of the aforementioned theorem of Volberg and 
Theorem \ref{teomain2}.

As explained above, the conditions on the measure $\mu$ in Corollary \ref{coro2} are stable by bilipschitz maps. So we deduce that if $\vphi:\R^{n+1}
\to\R^{n+1}$ is bilipschitz, then
$$\kappa(E)\approx \kappa(\vphi(E))\quad\mbox{ for any compact set $E\subset\R^{n+1}$,}$$
with the implicit constant just depending on the bilipschitz constant of $\vphi$ and the ambient dimension.

Another suggestive characterization of the capacity $\kappa(E)$ can be given in terms of the following potential, 
which we call the Jones-Wolff potential of $\mu$:
$$U_\mu(x) = \sup_{r>0} \theta_\mu(x,r) + \left(\int_0^\infty \beta_{2,\mu}(x,r)^2\,\theta_\mu(x,r)\,\frac{dr}r\right)^{1/2}.$$

\begin{coro}\label{coro3}
Let $E\subset\R^{n+1}$ be compact. Then
$$\kappa(E) \approx \sup\{\mu(E):\,U_\mu(x)\leq 1\,\forall x\in E\}.$$
\end{coro}

An immediate consequence of this result is that $E$ is non-removable for Lipschitz harmonic functions if and only if
it supports a non-zero measure such that $U_\mu(x)\leq1$ for all $x\in E$. 

The characterization of the capacity $\kappa$ and of removable sets for Lipschitz harmonic functions 
in terms of a metric-geometric potential such as $U_\mu$ should be considered as an analogue of the characterization
of analytic capacity and of removable sets for bounded analytic functions
in terms of curvature of measures \cite{Tolsa-sem}. So one can think of the results stated in Corollaries
\ref{coro2} and \ref{coro3} as possible solutions of the Painlevé problem for Lipschitz harmonic functions.

\vv
Next we describe the main ideas involved in the proof of Theorem \ref{teomain1}. As explained above, the proof
relies on the results obtained in \cite{DT}. More precisely, for a Radon measure $\sigma$ we consider the Wolff type
energy
$$\mathbb E(\sigma) = \int\!\! \int_0^\infty \left(\frac{\sigma(B(x,r))}{r^{n-\frac38}}\right)^2\,\frac{dr}r\,d\sigma(x)=\int\!\! \int_0^\infty r^{\frac34}\,\theta_\sigma(B(x,r)^2\,\frac{dr}{r}\,d\sigma(x).$$
As in \cite{DT}, given the Radon measure $\mu$, we consider a suitably modified version of the David-Mattila lattice $\DD_\mu$ associated 
with $\mu$. Then for a given $\PP$-doubling cube $Q\in\DD_\mu$, we let $\EE(4Q)$ be an appropriate discrete version of $\mathbb E(\mu|_{4Q})\,\ell(Q)^{-3/4}$.
We say that a $\PP$-doubling cube $Q$ has high energy, and we write $Q\in\HE$, if
$$\EE(4Q) \geq M_0^2\,\Theta(Q)^2\,\mu(Q),$$
where $M_0\gg1$ is some fixed constant, $\Theta(Q)$ is another discrete version of $\theta_\mu(2B_Q)$, and $B_Q$ is a ball concentric with $Q$, containing $Q$, with radius
comparable to $\ell(Q)$ (for the precise definitions of $\PP$-doubling cubes, $\Theta(Q)$, and $\EE(4Q)$ see Section \ref{sec4}). In \cite[Main Theorem 3.4]{DT} it is shown that
$$\int\!\!\int_0^\infty \beta_{2,\mu}(x,r)^2\,\theta_\mu(x,r)\,\frac{dr}r\,d\mu(x)\leq C\,\Big(\|\RR\mu\|_{L^2(\mu)}^2
+\theta_0^2\,\|\mu\| + \sum_{Q\in\HE} \EE(4Q)\Big),
$$
So to prove Theorem \ref{teomain1} it suffices to show that
\begin{equation}\label{eqwwwpp}
\sum_{Q\in\HE} \EE(4Q)\leq C\,\big(\|\RR\mu\|_{L^2(\mu)}^2 +\theta_0^2\,\|\mu\| \big).
\end{equation}
This is the task we perform in this paper (see Main Proposition \ref{propomain}). 

To prove \rf{eqwwwpp} we introduce
in Section \ref{sec4} 
a related energy $\EE_\infty$ which is more appropriate for the stopping time conditions and the bootstrapping 
argument involved in the proof of \rf{eqwwwpp}. The objective is then to show that the family $\DB$ of $\PP$-doubling cubes $Q$ 
such that
$$\EE_\infty(9Q) \geq M_0\,\Theta(Q)^2\,\mu(Q)$$
satisfies a Carleson type estimate analogous to \rf{eqwwwpp}. The first step is the construction of a
family (called $\GDF$) of cubes $Q$ which, in a sense, contain many stopping cubes whose density is much larger than
the density $\Theta(Q)$. The selection of this family, in Section \ref{sec5}, is one of the key steps for the proof proof of Theorem \ref{teomain1}. In Section \ref{sec6} we associate a family of ``tractable trees'' of cubes with each cube from
$\GDF$, by arguments somewhat similar to others appearing in \cite{DT}, although here we need to consider different and additional stopping conditions. In the tractable trees the density of $\mu$ oscillates in such a way that one can
bound from below the Haar coefficients of $\RR\mu$ for the cubes which belong to that tree or are close to that tree. For each tractable tree $\TT$, this is shown by a variational argument applied to a measure $\eta$ that approximates the measure $\mu$ at the level of some regularized stopping cubes of $\TT$.
This argument provides a lower estimate for $\|\RR\eta\|_{L^p(\eta)}$,  
 essentially in the same way as in \cite{DT}, and so we refer to the appropriate lemma from \cite{DT}
when this is required in Section \ref{sec8}. 

To complete the estimates from below for the Haar coefficients of 
$\RR\mu$ near the tree $\TT$, in Section \ref{sec9} we transfer the lower estimates obtained for $\|\RR\eta\|_{L^p(\eta)}$ to 
$\RR\mu$. This is another of the delicate key points of the proof of Theorem \ref{teomain2}. It requires much more precise estimates
than other related arguments appearing in \cite{DT} or \cite{JNRT}, mainly because the presence of cubes from $\DB$ (or $\HE$) in the tree $\TT$ originates ``error terms'' in the transference of those estimates which are difficult to control. Most important, we can only quantify the presence of cubes from $\DB$ in most of the trees $\TT$ by a bootstrapping argument which gives rather weak bounds.

\vv

In the whole paper we denote by $C$ or $c$ some constants that may depend on the dimension and perhaps other fixed parameters. Their values may change at different occurrences. On the contrary, constants with subscripts, like $C_0$, retain their values.
For $a,b\geq 0$, we write $a\lesssim b$ if there is $C>0$ such that $a\leq Cb$. We write $a\approx b$ to mean $a\lesssim b\lesssim a$.

\vv




\section{The modified dyadic lattice of David and Mattila and the dyadic martingale decomposition}\label{sec:DMlatt}\label{sec3}

\subsection{The David-Mattila lattice}
We recall now the properties of the dyadic lattice of cubes
with small boundaries of David-Mattila associated with a Radon measure $\mu$. 
This lattice has been constructed in \cite[Theorem 3.2]{David-Mattila}. 
Later on we will state some additional useful properties of $\DD_\mu$ that are obtained by modifying its construction as in \cite{DT}.

\begin{lemma}[David, Mattila]
\label{lemcubs}
Let $\mu$ be a compactly supported Radon measure in $\R^{d}$.
Consider two constants $C_0>1$ and $A_0>5000\,C_0$ and denote $E=\supp\mu$. 
Then there exists a sequence of partitions of $E$ into
Borel subsets $Q$, $Q\in \DD_{\mu,k}$, with the following properties:
\begin{itemize}
\item For each integer $k\geq0$, $E$ is the disjoint union of the ``cubes'' $Q$, $Q\in\DD_{\mu,k}$, and
if $k<l$, $Q\in\DD_{\mu,l}$, and $R\in\DD_{\mu,k}$, then either $Q\cap R=\varnothing$ or else $Q\subset R$.
\vv

\item The general position of the cubes $Q$ can be described as follows. For each $k\geq0$ and each cube $Q\in\DD_{\mu,k}$, there is a ball $B(Q)=B(x_Q,r(Q))$ such that
$$x_Q\in E, \qquad A_0^{-k}\leq r(Q)\leq C_0\,A_0^{-k},$$
$$E\cap B(Q)\subset Q\subset E\cap 28\,B(Q)=E \cap B(x_Q,28r(Q)),$$
and
$$\mbox{the balls\, $5B(Q)$, $Q\in\DD_{\mu,k}$, are disjoint.}$$

\vv
\item The cubes $Q\in\DD_{\mu,k}$ have small boundaries. That is, for each $Q\in\DD_{\mu,k}$ and each
integer $l\geq0$, set
$$N_l^{ext}(Q)= \{x\in E\setminus Q:\,\dist(x,Q)< A_0^{-k-l}\},$$
$$N_l^{int}(Q)= \{x\in Q:\,\dist(x,E\setminus Q)< A_0^{-k-l}\},$$
and
$$N_l(Q)= N_l^{ext}(Q) \cup N_l^{int}(Q).$$
Then
\begin{equation}\label{eqsmb2}
\mu(N_l(Q))\leq (C^{-1}C_0^{-3d-1}A_0)^{-l}\,\mu(90B(Q)).
\end{equation}
\vv

\item Denote by $\DD_{\mu,k}^{db}$ the family of cubes $Q\in\DD_{\mu,k}$ for which
\begin{equation}\label{eqdob22}
\mu(100B(Q))\leq C_0\,\mu(B(Q)).
\end{equation}
We have that $r(Q)=A_0^{-k}$ when $Q\in\DD_{\mu,k}\setminus \DD_{\mu,k}^{db}$
and
\begin{equation}\label{eqdob23}
\mu(100B(Q))\leq C_0^{-l}\,\mu(100^{l+1}B(Q))\quad
\mbox{for all $l\geq1$ with $100^l\leq C_0$ and $Q\in\DD_{\mu,k}\setminus \DD_{\mu,k}^{db}$.}
\end{equation}
\end{itemize}
\end{lemma}

\vv

\begin{rem}\label{rema00}
The constants $C_0$ and $A_0$ are chosen so that
$$A_0 = C_0^{C(d)},$$
where $C(d)$ depends
 just on $d$ and $C_0$ is big enough.
\end{rem}

We use the notation $\DD_\mu=\bigcup_{k\geq0}\DD_{\mu,k}$. Observe that the families $\DD_{\mu,k}$ are only defined for $k\geq0$. So the diameter of the cubes from $\DD_\mu$ are uniformly
bounded from above.
For $Q\in\DD_{\mu,k}$, we set
$\ell(Q)= 56\,C_0\,A_0^{-k}$ and we call it the side length of $Q$. Notice that 
$$C_0^{-1}\ell(Q)\leq \diam(28B(Q))\leq\ell(Q).$$
Observe that $r(Q)\approx\diam(Q)\approx\ell(Q)$.
Also we call $x_Q$ the center of $Q$, and the cube $Q'\in \DD_{\mu,k-1}$ such that $Q'\supset Q$ the parent of $Q$.
We denote the family of cubes from $\DD_{\mu,k+1}$ which are contained in $Q$ by $\CH(Q)$, and we call their elements children or sons of $Q$.
 We set
$B_Q=28 B(Q)=B(x_Q,28\,r(Q))$, so that 
$$E\cap \tfrac1{28}B_Q\subset Q\subset B_Q\subset B(x_Q,\ell(Q)/2).$$

For a given $\gamma\in(0,1)$, let $A_0$ be big enough so that the constant $C^{-1}C_0^{-3d-1}A_0$ in 
\rf{eqsmb2} satisfies 
$$C^{-1}C_0^{-3d-1}A_0>A_0^{\gamma}>10.$$
Then we deduce that, for all $0<\lambda\leq1$,
\begin{align}\label{eqfk490}
\mu\bigl(\{x\in Q:\dist(x,E\setminus Q)\leq \lambda\,\ell(Q)\}\bigr) + 
\mu\bigl(\bigl\{x\in 3.5B_Q\setminus Q:\dist&(x,Q)\leq \lambda\,\ell(Q)\}\bigr)\\
&\leq_\gamma
c\,\lambda^{\gamma}\,\mu(3.5B_Q).\nonumber
\end{align}

We denote
$\DD_\mu^{db}=\bigcup_{k\geq0}\DD_{\mu,k}^{db}$.
Note that, in particular, from \rf{eqdob22} it follows that
\begin{equation}\label{eqdob*}
\mu(3B_{Q})\leq \mu(100B(Q))\leq C_0\,\mu(Q)\qquad\mbox{if $Q\in\DD_\mu^{db}.$}
\end{equation}
For this reason we will call the cubes from $\DD_\mu^{db}$ doubling. 
Given $Q\in\DD_\mu$, we denote by $\DD_\mu(Q)$
the family of cubes from $\DD_\mu$ which are contained in $Q$. Analogously,
we write $\DD_\mu^{db}(Q) = \DD^{db}_\mu\cap\DD(Q)$. 

\vv

As shown in \cite[Lemma 5.28]{David-Mattila}, every cube $R\in\DD_\mu$ can be covered $\mu$-a.e.\
by a family of doubling cubes:

\begin{lemma}\label{lemcobdob}
Let $R\in\DD_\mu$. Suppose that the constants $A_0$ and $C_0$ in Lemma \ref{lemcubs} are
chosen suitably. Then there exists a family of
doubling cubes $\{Q_i\}_{i\in I}\subset \DD_\mu^{db}$, with
$Q_i\subset R$ for all $i$, such that their union covers $\mu$-almost all $R$.
\end{lemma}

The following result is proved in \cite[Lemma 5.31]{David-Mattila}.
\vv

\begin{lemma}\label{lemcad22}
Let $R\in\DD_\mu$ and let $Q\subset R$ be a cube such that all the intermediate cubes $S$,
$Q\subsetneq S\subsetneq R$ are non-doubling (i.e.\ belong to $\DD_\mu\setminus \DD_\mu^{db}$).
Suppose that the constants $A_0$ and $C_0$ in Lemma \ref{lemcubs} are
chosen suitably. 
Then
\begin{equation}\label{eqdk88}
\mu(100B(Q))\leq A_0^{-10d(J(Q)-J(R)-1)}\mu(100B(R)).
\end{equation}
\end{lemma}


Given a ball $B\subset \R^{n+1}$, we consider its $n$-dimensional density:
$$\theta_\mu(B)= \frac{\mu(B)}{r(B)^n}.$$
We will also write $\theta_\mu(x,r)$ instead of $\theta_\mu(B(x,r))$.

From the preceding lemma we deduce:

\vv
\begin{lemma}\label{lemcad23}
Let $Q,R\in\DD_\mu$ be as in Lemma \ref{lemcad22}.
Then
$$\theta_\mu(100B(Q))\leq (C_0A_0)^{n+1}\,A_0^{-9d(J(Q)-J(R)-1)}\,\theta_\mu(100B(R))$$
and
$$\sum_{S\in\DD_\mu:Q\subset S\subset R}\theta_\mu(100B(S))\leq c\,\theta_\mu(100B(R)),$$
with $c$ depending on $C_0$ and $A_0$.
\end{lemma}

For the easy proof, see
 \cite[Lemma 4.4]{Tolsa-memo}, for example.

\vv

\subsection{The dyadic martingale decomposition and the additional properties from \cite{DT}}

For $f\in L^2(\mu)$ and $Q\in\DD_\mu$ we denote 
\begin{equation}\label{eqdq1}
\Delta_Q f=\sum_{S\in\CH(Q)}m_{\mu,S}(f)\chi_S-m_{\mu,Q}(f)\chi_Q,
\end{equation}
where $m_{\mu,S}(f)$ stands for the average of $f$ on $S$ with respect to $\mu$.
Then we have the orthogonal expansion, for any cube $R\in\DD_\mu$,
$$\chi_{R} \bigl(f - m_{\mu,R}(f)\bigr) = \sum_{Q\in\DD_\mu(R)}\Delta_Q f,$$
in the $L^2(\mu)$-sense, so that
$$\|\chi_{R} \bigl(f - m_{\mu,R}(f)\|_{L^2(\mu)}^2 = \sum_{Q\in\DD_\mu(R)}\|\Delta_Q f\|_{L^2(\mu)}^2.$$

In several places in this paper we will have to estimate terms such as 
$\|\RR(\chi_{2B_Q\setminus Q}\mu)\|_{L^2(\mu|_Q)}$, and so we will need to deal with integrals such as 
$$\int_{2B_Q\setminus Q}\left(\int_Q \frac1{|x-y|^n}\,d\mu(y)\right)^2 d\mu(x).$$
We will describe now how this integral can be estimated in terms of the energy $\EE(9Q)$ when using
the enhanced David-Mattila lattice from \cite{DT}.

We need some additional notation.
Given $Q\in\DD_\mu$ and $\lambda>1$, we let $\lambda Q$ be the union of cubes $P$ from the same
generation as $Q$ such that $\dist(x_Q,P)\leq \lambda \,\ell(Q)$. 
 Notice that
\begin{equation}\label{eqlambq12}
\lambda Q\subset B(x_Q,(\lambda+\tfrac12)\ell(Q)).
\end{equation}
Also, we let
$$\DD_\mu(\lambda Q)=\{P\in\DD_\mu:P\subset \lambda Q,\,\ell(P)\leq \ell(Q)\},$$
and 
$$\DD_{\mu,k}(\lambda Q) =\{P\in\DD_\mu:P\subset \lambda Q,\,\ell(P)=A_0^{-k} \ell(Q)\},\qquad
\DD_\mu^k(\lambda Q) = \bigcup_{j\geq k} \DD_{\mu,j}(\lambda Q).
$$

Next we set
\begin{equation}\label{eqdmuint}
\wt \DD_\mu^{int}(Q) = \big\{P\in\DD_\mu(Q):2B_P\cap (\supp\mu\setminus Q)\neq \varnothing\big\}
\end{equation}
and
\begin{equation}\label{eqdmuext}
\wt \DD_\mu^{ext}(Q) = \big\{P\in\DD_\mu:\ell(P)\leq \ell(Q),P\subset \R^{n+1}\setminus Q,\,2B_P\cap Q\neq \varnothing\big\}.
\end{equation}
Also,
\begin{equation}\label{eqdmutot}
\wt \DD_\mu(Q) = \wt \DD_\mu^{int}(Q) \cup \wt \DD_\mu^{ext}(Q),
\end{equation}
and
$$\wt \DD_{\mu,k}(Q) = 
\{P\in\wt \DD_\mu:P\subset \lambda Q,\,\ell(P)=  A_0^{-k}\ell(Q)\}.$$

We need the following auxiliary result, whose proof follows from a straightforward calculation (see \cite[Lemma 2.7]{DT}).

\begin{lemma}\label{lemDMimproved2}
Let $\mu$ be a compactly supported Radon measure in $\R^{d}$ and $Q\in\DD_\mu$. For any  $\alpha\in(0,1)$, we have
\begin{align}\label{eqdosint}
\int_{2B_Q\setminus Q}\left(\int_Q \frac1{|x-y|^n}\,d\mu(y)\right)^2d\mu(x) \,
+ &\int_{Q}\left(\int_{2B_Q\setminus Q} \frac1{|x-y|^n}\,d\mu(y)\right)^2d\mu(x) \\
& \lesssim_{\alpha,A_0}
\sum_{P\in\wt \DD_\mu(Q)} \left(\frac{\ell(Q)}{\ell(P)}\right)^\alpha\theta_\mu(2B_P)^2
\,\mu(P).\nonumber
\end{align}
\end{lemma}

\vv
The following result is proven in \cite[Lemma 2.9]{DT}.

\begin{lemma}\label{lemdmutot}
 Let $\mu$ be a compactly supported Radon measure in $\R^{d}$.
Assume that $\mu$ has polynomial growth of degree $n$ and let $\gamma\in(0,1)$. The lattice $\DD_\mu$ from Lemma \ref{lemcubs} can be constructed so that the following holds for all
all $Q\in\DD_{\mu}$:
\begin{equation}\label{eqfhq29}
\sum_{P\in\wt \DD_\mu(Q)} \left(\frac{\ell(Q)}{\ell(P)}\right)^{\frac{1-\gamma}2} \theta_\mu(2B_P)^2
\,\mu(P) \leq C(\gamma)\sum_{P\in\DD_\mu: P\subset 9Q} \left(\frac{\ell(P)}{\ell(Q)}\right)^\gamma\theta_\mu(2B_P)^2\mu(P).
\end{equation}
\end{lemma}

Remark that saying that $\mu$ has  polynomial growth of degree $n$ means that \rf{eqgrow00}
holds for some arbitrary constant $\theta_0$.
In the lemma this
assumption is just used to ensure that the sums above are finite. Further, the constant $C(\gamma)$ does not depend on the polynomial growth constant $\theta_0$.

By combining Lemmas \ref{lemDMimproved2} and \ref{lemdmutot}, we obtain

\begin{lemma}\label{lemDMimproved}
Let $\mu$ be a compactly supported Radon measure in $\R^{d}$.
Assume that $\mu$ has $n$-polynomial growth and let $\gamma\in(0,1)$. The lattice $\DD_\mu$ from Lemma
\ref{lemcubs} can be constructed so that the following holds for
all $Q\in\DD_{\mu}$:
\begin{align*}
\int_{2B_Q\setminus Q}\left(\int_Q \frac1{|x-y|^n}\,d\mu(y)\right)^2 d\mu(x) 
 + &\int_{Q}\left(\int_{2B_Q\setminus Q} \frac1{|x-y|^n}\,d\mu(y)\right)^2 d\mu(x)\\
&\leq C(\gamma)\sum_{P\in\DD_\mu: P\subset 2Q} \left(\frac{\ell(P)}{\ell(Q)}\right)^\gamma\theta_\mu(2B_P)^2\mu(P).
\end{align*}
\end{lemma}

\vv


\section{The family $\hd^k(Q)$ and the Main Proposition}\label{sec4}

\subsection{$\PP$-doubling cubes and the family $\hd^k(Q)$}
In the rest of the paper we assume that $\mu$ is a compactly supported Radon measure with polynomial growth of degree $n$
such that $\RR_*\mu(x)<\infty$ $\mu$-a.e.
We assume that $\DD_\mu$ is a David-Mattila dyadic lattice satisfying the properties 
described in the preceding section, in particular, the ones in Lemmas \ref{lemcubs}
and \ref{lemDMimproved}, with $\gamma=9/10$.
By rescaling, we assume that $\DD_{\mu,k}$ is defined for all $k\geq k_0$, with $A_0^{-k_0}\approx
\diam(\supp\mu)$, and we also suppose that there is a unique cube in $\DD_{\mu,k_0}$ which coincides
 with the whole $\supp\mu$.
Further, from now on, we allow all the constants $C$ and all implicit constants in the relationship
``$\lesssim$'' to depend on the parameters $C_0, A_0$ of the dyadic lattice of David-Mattila.

We denote
$$\PP(Q) = \sum_{R\in\DD_\mu:R\supset Q} \frac{\ell(Q)}{\ell(R)^{n+1}} \,\mu(2B_R).$$
We say that a cube $Q$ is $\PP$-doubling, and we write $Q\in\DD_\mu^\PP$, if
$$\PP(Q) \leq C_d\,\frac{\mu(2B_Q)}{\ell(Q)^n},$$
for  $C_d =4A_0^n$. 
Notice that
$$\PP(Q) \approx_{C_0} \sum_{R\in\DD_\mu:R\supset Q} \frac{\ell(Q)}{\ell(R)} \,\theta_\mu(2B_R).$$
and thus  $Q$ being $\PP$-doubling implies that 
$$\sum_{R\in\DD_\mu:R\supset Q} \frac{\ell(Q)}{\ell(R)} \,\theta_\mu(2B_R)\leq C_d'\,\theta_\mu(2B_Q)$$
for some constant $C_d'$ depending on $C_d$. Conversely, the latter condition implies that $Q$ is $\PP$-doubling with another constant $C_d$ depending on $C_d'$.

As shown in \cite[Lemma 3.1]{DT}, from the properties of the David-Mattila lattice, it follows:

\begin{lemma}
Suppose that $C_0$ and $A_0$ to be chosen suitably. If $Q$ is $\PP$-doubling, then $Q\in\DD_{\mu,db}$.
\end{lemma}

\vv
Notice that, by the preceding lemma, if $Q$ is $\PP$-doubling, then 
$$\sum_{R\in\DD_\mu:R\supset Q} \frac{\ell(Q)^{n+1}}{\ell(R)^{n+1}} \,\mu(2B_R) \lesssim_{C_d} \mu(Q).$$

For technical reasons that will be more evident below, it is appropriate to consider a discrete version of the density $\theta_\mu$. Given $\lambda\geq1$ and $Q\in\DD_\mu$, we let
$$\Theta(Q) = A_0^{kn} \quad \mbox{ if\, $\dfrac{\mu(2B_Q)}{\ell(Q)^n}\in [A_0^{kn},A_0^{(k+1)n})$}.$$
Clearly, $\Theta(Q)\approx \theta_\mu(2B_Q)$.
Notice also that if $\Theta(Q) = A_0^k$ and $P$ is a son of $Q$, then
$$
\frac{\mu(2 B_P)}{\ell(P)^n} \leq \frac{\mu(2 B_Q)}{\ell(P)^n} = A_0^n\,\frac{\mu(2 B_Q)}{\ell(Q)^n}.
$$
Thus,
\begin{equation}\label{eqson1}
\Theta(P)\leq A_0^n\,\Theta(Q)\quad \mbox{  for any son $P$ of $Q$.}
\end{equation}

Given $Q\in\DD_\mu$ and $k\geq1$, we denote by $\hd^k(Q)$ the family of maximal cubes $P\in\DD_\mu$ satisfying
\begin{equation}\label{a0tilde}
\ell(P)<\ell(Q), \qquad \Theta(P)\geq  A_0^{kn}\Theta(Q).
\end{equation}

The following result is proved in \cite[Lemma 3.3]{DT} (see also \cite[Lemma 2.1]{Reguera-Tolsa}).

\begin{lemma}\label{lemdobpp}
Let $Q_0,Q_1,\ldots,Q_m$ be a family of cubes from $\DD_\mu$ such that $Q_j$ is son of $Q_{j-1}$ for $1\leq j\leq 
m$. Suppose that $Q_j$ is not $\PP$-doubling for $1\leq j\leq m$.
Then
\begin{equation}\label{eqcad35}
\frac{\mu(B_{Q_m})}{\ell(Q_m)^n}\leq A_0^{-m/2}\,\PP(Q_0)
\end{equation}
and
\begin{equation}\label{eqcad35'}
\PP(Q_m)\leq 2A_0^{-m/2}\,\PP(Q_0).
\end{equation}
\end{lemma}

\vv

We also have:

\begin{lemma}\label{lempdoubling}
Let $Q\in\DD_\mu$ be $\PP$-doubling. Then, for $k\geq4$, every $P\in\hd^k(Q)\cap\DD_\mu(9Q)$ is also $\PP$-doubling
and moreover $\Theta(P)=A_0^{nk}\Theta(Q)$.
\end{lemma}

This is proven in \cite[Lemma 3.2]{DT}. From this result it follows immediately that
\begin{equation}\label{eqforakdf33}
\Theta(P)\approx A_0^{nk}\Theta(Q)\quad\mbox{ for all $P\in\hd^k(Q)\cap\DD_\mu(9Q)$
and all
 $k\geq1$.}
\end{equation}

\vv

\subsection{The energies $\EE$, $\EE^H$, $\EE_\infty$, and the Main Proposition}

For given $\lambda\geq1$ and $Q\in\DD_\mu$, we consider the energy
$$\EE(\lambda Q) = \sum_{P\in\DD_\mu(\lambda Q)} \left(\frac{\ell(P)}{\ell(Q)}\right)^{3/4}\Theta(P)^2\,\mu(P).$$
We also denote
$$\EE^H(\lambda Q) = \sum_{k\geq 0} 
\sum_{P\in\hd^k(Q)\cap\DD_\mu(\lambda Q)} \left(\frac{\ell(P)}{\ell(Q)}\right)^{3/4}\Theta(P)^2\,\mu(P)$$
and
$$\EE_\infty(\lambda Q) = \sup_{k\geq1}
\sum_{P\in\hd^k(Q)\cap\DD_\mu(\lambda Q)} \left(\frac{\ell(P)}{\ell(Q)}\right)^{\!1/2}\Theta(P)^2\,\mu(P).$$

\vv
\begin{lemma}\label{lemenergias}
For every $Q\in\DD_\mu$
we have
$$\EE(9Q) \lesssim \EE^H(9Q)\lesssim\EE_\infty(9Q).$$
\end{lemma}

\begin{proof}
For a given $R\in\hd^k(Q)$, we denote by $\tree_H(R)$ the family of cubes from $\DD_\mu$ that 
are contained in $R$ and are not contained in any cube from $\hd^{k+1}(Q)$.
Using that $\Theta(P)\lesssim \Theta(R)$ for all $P\in\tree_H(R)$ (remember that we do not keep track of the implicit constants depending on $A_0$), we get
\begin{align*}
\EE(9Q) & =
\sum_{k\geq0} 
\sum_{R\in\hd^k(Q)\cap\DD_\mu(9Q)}\, \sum_{P\in\tree_H(R)}
\left(\frac{\ell(P)}{\ell(Q)}\right)^{3/4}\Theta(P)^2\,\mu(P)\\
& \lesssim 
\sum_{k\geq0} 
\sum_{R\in\hd^k(Q)\cap\DD_\mu(9Q)} \Theta(R)^2\sum_{P\in\tree_H(R)}
\left(\frac{\ell(P)}{\ell(Q)}\right)^{3/4}\,\mu(P)\\
&\lesssim 
\sum_{k\geq0} 
\sum_{R\in\hd^k(Q)\cap\DD_\mu(9Q)} \left(\frac{\ell(R)}{\ell(Q)}\right)^{3/4}\Theta(R)^2\,\mu(R) = \EE^H(9Q).
\end{align*}

To show $\EE^H(9Q)\lesssim\EE_\infty(9Q)$, 
denote
$$m_k(Q) = \frac1{\ell(Q)}\max\{\ell(P):P\in\hd^k(Q)\cap\DD_\mu(9Q)\}.$$
Then we have
\begin{align*}
\EE^H(9Q)
& = 
\sum_{k\geq 0} 
\sum_{P\in\hd^k(Q)\cap\DD_\mu(9Q)} \left(\frac{\ell(P)}{\ell(Q)}\right)^{3/4}\Theta(P)^2\,\mu(P)\\
& \leq \sum_{k\geq 0} 
 m_k(Q)^{1/4}
\sum_{P\in\hd^k(Q)\cap\DD_\mu(9Q)} \left(\frac{\ell(P)}{\ell(Q)}\right)^{\!1/2}\Theta(P)^2\,\mu(P)
\end{align*}
To estimate $m_k(Q)$, observe that if $P\in\hd^k(Q)\cap\DD_\mu(9Q)$ and $k\geq4$, then
$$A_0^{kn}\,\theta_\mu(2B_Q)\approx\theta_\mu(2B_P) \lesssim \frac{\ell(Q)^n}{\ell(P)^n} \,\theta_\mu(2B_Q).$$
Hence,
$\ell(P)\lesssim A_0^{-k}\,\ell(Q)$, and thus, since this also holds in the case $1\leq k\leq 3$,
\begin{equation}\label{eqmkpet3}
m_k(Q) \lesssim  A_0^{-k}\qquad \mbox{ for all $k\geq1$.}
\end{equation}
Consequently, 
\begin{align*}
\EE^H(9Q)&\lesssim  \sum_{k\geq 0}A_0^{-k/4}
\sum_{P\in\hd^k(Q)\cap\DD_\mu(9Q)} \left(\frac{\ell(P)}{\ell(Q)}\right)^{\!1/2}\theta_\mu(2B_P)^2\,\mu(P)\\
& \lesssim  \sum_{k\geq 0}A_0^{-k/4} \EE_\infty(9Q)
\approx \EE_\infty(9Q).
\end{align*}
\end{proof}

\vv

\begin{rem}\label{remmk}
For the record, notice that, given $Q\in\DD_\mu^\PP$ and
\begin{equation}\label{eqmkpet2}
m_k(Q) = \frac1{\ell(Q)}\max\{\ell(P):P\in\hd^k(Q)\cap\DD_\mu(9Q)\},
\end{equation}
as shown in \rf{eqmkpet3}, it turns out that 
\begin{equation}\label{eqmkpet4}
m_k(Q) \leq  C_1 A_0^{-k}.
\end{equation}
\end{rem}

\vv


Given $M\gg1$ (we will choose $M>  A_0^{2n}\gg 1$), we say that $Q\in\DD_\mu$ is $M$-dominated from below
if
there exists some $k\geq1$ such that
\begin{equation}\label{eqDB}
\sum_{P\in\hd^k(Q)\cap\DD_\mu(9Q)} \left(\frac{\ell(P)}{\ell(Q)}\right)^{\!1/2}\Theta(P)^2\,\mu(P) > M^2\,\Theta(Q)^2\,\mu(9Q),
\end{equation}
or in other words,
\begin{equation}\label{eqDB'}
\EE_\infty(9Q)> M^2\,\Theta(Q)^2\,\mu(9Q),
\end{equation}
We denote by $\DB(M)$ the family of cubes from $\DD_\mu^\PP$ that are $M$-dominated from below. 
Notice that the cubes from $\DB(M)$ are assumed to be $\PP$-doubling.

\vv

Recall that in \cite[Theorem 3.4]{DT}, the following result has been obtained:

\begin{theorem}\label{propomain00}
Let $\mu$ be a Radon measure in $\R^{n+1}$ with polynomial growth of degree $n$ with constant $\theta_0$.
For any choice of $M>1$, let
$$\HE(M) = \{Q\in\DD_\mu^\PP:\EE(4Q) \geq M^2\,\Theta(Q)^2\,\mu(Q)\}.$$
Then we have
\begin{equation}\label{eqpropo*}
\sum_{Q\in\DD_\mu} \beta_{2,\mu}(2B_Q)^2\,\Theta(Q)\,\mu(Q)\leq C\,\big(\|\RR\mu\|_{L^2(\mu)}^2 + \theta_0^2\,\|\mu\|
+ \sum_{Q\in\HE(M)}\EE(4Q)\big),
\end{equation}
with $C$ depending on $M$.
\end{theorem}
\vv

In the theorem above we wrote $\RR\mu$ instead of $\pv\RR\mu$. From now on we will follow this
notation.
In the present paper we will obtain the next result.

\begin{mpropo}\label{propomain}
Suppose that $\mu$ has polynomial growth of degree $n$ with constant $\theta_0$ and $\RR_*\mu(x)<\infty$ $\mu$-a.e. Let $M_0=A_0^{k_0 n}$, where $k_0$ is some big enough absolute constant depending just on $n$.
Then
$$\sum_{Q\in\DB(M_0)} \EE_\infty(9Q)\leq C\, \big(\|\RR\mu\|_{L^2(\mu)}^2 + 
\theta_0^2\,\|\mu\|\big),$$
where $C$ depends just on $n$ and the parameters of the dyadic lattice $\DD_\mu$.
\end{mpropo}
\vv

Notice that if $Q\in\HE(M)$, by Lemma \ref{lemenergias},
then
$$M^2\,\Theta(Q)^2\,\mu(Q)\leq \EE(4Q) \lesssim 
\EE^H(4Q) \lesssim\leq  \EE_\infty(9Q).$$
Hence, $Q\in\DB(M')$ for some $M'$ depending on $M$. So, by combining Propositions \ref{propomain00}
and \ref{propomain} and choosing $M$ and $M_0$ appropriately, we deduce that
$$\sum_{Q\in\DD_\mu} \beta_{2,\mu}(2B_Q)^2\,\Theta(Q)\,\mu(Q)\leq C\,\big(\|\RR\mu\|_{L^2(\mu)}^2 + \theta_0^2\,\|\mu\|
\big),$$
which implies Theorem \ref{teomain1}.

The rest of the paper is devoted to the proof of Main Proposition \ref{propomain}.

\vv


\section{The good dominating family $\GDF$}\label{sec5}

Let
\begin{equation}\label{eqk00}
M\geq A_0^{k_0 n}=:M_0\quad \mbox{ for some $k_0\geq 4$.}
\end{equation}
For each $Q\in \DB(M)$ we choose the minimal $k(Q,M)\in\N$ such that \rf{eqDB} holds with $k=k(Q,M)$.

\begin{lemma}\label{lemkq}
Assume $k_0$ big enough in \rf{eqk00}. For each $Q\in \DB(M)$, we have 
$$k(Q,M) > \frac{8n- 1}{8n-2}\, k_0+4$$.
\end{lemma}

\begin{proof}
This follows from the fact that for $j\geq0$,
\begin{align*}
\sum_{P\in\hd^j(Q)\cap \DD_\mu(9Q)} 
\left(\frac{\ell(P)}{\ell(Q)}\right)^{\!1/2}\Theta(P)^2\mu(P) & \leq  CA_0^{2n j}\,\Theta(Q)^2\, m_j(Q)^{\!1/2}
\sum_{P\in\hd^j(Q)\cap \DD_\mu(9Q)} \mu(P) \\
& \leq C_1A_0^{2n j -\frac j2}\,\Theta(Q)^2\mu(9Q),
\end{align*}
where  we used \rf{eqforakdf33} and we applied \rf{eqmkpet4} to estimate $m_j(Q)$.
Then, for $0\leq j \leq  \frac{8n-1}{8n-2}\, k_0+4$, we have
\begin{align*}
\sum_{P\in\hd^j(Q)\cap \DD_\mu(9Q)} 
\left(\frac{\ell(P)}{\ell(Q)}\right)^{\!1/2}\Theta(P)^2\mu(P) & \leq C_1
A_0^{(2n-\frac12)\left(\frac{8n-1}{8n-2}\,k_0 +4\right)}\,\Theta(Q)^2\mu(9Q)\\
& = C_1
 M_0^{2}\,A_0^{-\frac{k_0}4+8n-2}\,
 \Theta(Q)^2\mu(9Q).
\end{align*}
So, for $k_0$ big enough, the right hand side above is smaller than $M^{2}\,\Theta(Q)^2\mu(9Q),$ 
which ensures that $k(Q,M)> \frac{8n- 1}{8n-2}\, k_0+4$.
\end{proof}
\vv

Let 
$$k_\Lambda = \frac{8n- 1}{8n-2}\, k_0,$$
so that $k(Q,M)>k_\Lambda$ for each $Q\in\DB(M)$, by the preceding lemma. 
Assuming $k_0$ to be a multiple of $8n-2$, it follows that $k_\Lambda$ is natural number.
 Notice also that $k_\Lambda$ is the mean of $k_0$ and $4nk_0/(4n-1)$, so that, if we let
$$\Lambda = A_0^{k_\Lambda n},$$
we have that $\Lambda$ is the geometric mean of $M_0
$ and $M_0^{\frac{4n}{4n-1}}$, that is
\begin{equation}\label{eqlammm}
\Lambda = M_0^{1/2}\,M_0^{\frac{2n}{4n-1}} = M_0^{\frac{8n-1}{8n-2}}>M_0.
\end{equation}
\vv

Observe that, for $Q\in \DB(M)$, taking into account that $k(Q,M)-k_\Lambda>4$,
\begin{align}\label{eqprel3}
 M^2\,\Theta(Q)^2\,\mu(9Q) & \leq
\sum_{P\in\hd^{k(Q,M)}(Q)\cap\DD_\mu(9Q)} \left(\frac{\ell(P)}{\ell(Q)}\right)^{\!1/2}\Theta(P)^2\,\mu(P)\\
& = 
\sum_{S\in\hd^{k(Q,M)-k_\Lambda}(Q)\cap\DD_\mu(9Q)} \,\sum_{P\in\hd^{k(Q,M)}(Q): P\subset S}
\left(\frac{\ell(P)}{\ell(Q)}\right)^{\!1/2}\Theta(P)^2\,\mu(P) \nonumber\\
 & \leq \Lambda^2\,
\sum_{S\in\hd^{k(Q,M)-k_\Lambda}(Q)\cap\DD_\mu(9Q)}\!\! \Theta(S)^2\!\!\sum_{P\in\hd^{k(Q,M)}(Q): P\subset S}
\left(\frac{\ell(P)}{\ell(Q)}\right)^{\!1/2}\mu(P).\nonumber
\end{align}

Given $Q\in\DB(M)$ and $S\in \hd^{k(Q,M)-k_\Lambda}(Q)\cap\DD_\mu(9Q)$, we write $S\in G(Q,M)$ if 
\begin{equation}\label{eqgq1}
\mu(S) \leq 2 \Lambda^2
\sum_{P\in\hd^{k(Q,M)}(Q): P\subset S}
\left(\frac{\ell(P)}{\ell(S)}\right)^{\!1/2}\mu(P).
\end{equation}
We also denote $B(Q,M) = \hd^{k(Q,M)-k_\Lambda}(Q)\cap\DD_\mu(9Q)\setminus G(Q,M)$. 

Given $\lambda>0$, for $Q\in\DB(M)$ and $S\in G(Q,M)$, we denote
$${\rm big}_\lambda(S) = \{P\in\hd^{k(Q,M)}:P\subset S,\,\ell(P)\geq \lambda \,\ell(S)\}.$$

\vv

\vv
\begin{lemma}\label{lem16}
If $\lambda>0$ satisfies 
\begin{equation}\label{eqlambda1}
\lambda\leq \frac{c_0}{\Lambda^{4}} =:\lambda_0(\Lambda)
\end{equation}
for some small absolute constant $c_0\in(0,1)$,
then, for each $Q\in\DB(M)$ we have
$$M^2\Theta(Q)^2\,\mu(Q) \lesssim \Lambda^{-\frac1{2n}}
\sum_{S\in G(Q,M)}\left(\frac{\ell(S)}{\ell(Q)}\right)^{\!1/2}\sum_{P\in {\rm big}_\lambda(S)}
\Theta(P)^2\mu(P).$$
Also, each $S\in G(Q,M)$ satisfies
$$\Theta(S)^2\,\mu(S) \leq 4 \Lambda^{-\frac1{2n}}\sum_{P\in{\rm big}_\lambda(S)}
\Theta_\mu(P)^2\,\mu(P).$$
\end{lemma}

\begin{proof} Arguing as in \rf{eqprel3}, by the definition of $B(Q,M)$ we get
\begin{align*}
M^2\,\Theta_\mu(2Q)^2\,\mu(9Q) & \leq 
\sum_{S\in G(Q,M)} \sum_{P\in\hd^{k(Q,M)}(Q): P\subset S}
\left(\frac{\ell(P)}{\ell(Q)}\right)^{\!1/2} \Theta(P)^2\,\mu(P)\\
&\quad + 
\Lambda^2\,
\sum_{S\in B(Q,M)}\Theta(S)^2\!\!\sum_{P\in\hd^{k(Q,M)}(Q): P\subset S}
\left(\frac{\ell(P)}{\ell(Q)}\right)^{\!1/2}\mu(P)\\
& \leq \sum_{S\in G(Q,M)} \sum_{P\in\hd^{k(Q,M)}(Q): P\subset S}
\left(\frac{\ell(P)}{\ell(Q)}\right)^{\!1/2} \Theta(P)^2\,\mu(P)\\
&\quad + \frac12 
\sum_{S\in B(Q,M)} \Theta(S)^2
\left(\frac{\ell(S)}{\ell(Q)}\right)^{\!1/2}\mu(S).
\end{align*}
Using that $B(Q,M)\subset \hd^{k(Q,M)-k_\Lambda}(Q)\cap\DD_\mu(9Q)$ and that, by the definition of $k(Q,M)$, 
$$
\sum_{S\in \hd^{k(Q,M)-k_\Lambda}(Q)\cap\DD_\mu(9Q)} 
\left(\frac{\ell(S)}{\ell(Q)}\right)^{\!1/2}\Theta(S)^2\mu(S) \leq M^2\,\Theta(Q)^2\,\mu(9Q),$$
we get
$$M^2\,\Theta(Q)^2\,\mu(9Q) \leq 2
\sum_{S\in G(Q,M)} \sum_{P\in\hd^{k(Q,M)}(Q): P\subset S}
\left(\frac{\ell(P)}{\ell(Q)}\right)^{\!1/2}\Theta(P)^2\,\mu(P).$$

Next, for $S\in G(Q,M)$, we split
$$
\sum_{P\in\hd^{k(Q,M)}(Q): P\subset S}
\left(\frac{\ell(P)}{\ell(Q)}\right)^{\!1/2}\Theta(P)^2\,\mu(P) =
\sum_{P\in {\rm big}_\lambda(S)} \cdots \; + 
\sum_{P\in\hd^{k(Q,M)}(Q)\setminus {\rm big}_\lambda(S): P\subset S}\!\!\! \cdots.$$
We estimate the last sum:
\begin{align*}
\sum_{P\in\hd^{k(Q,M)}(Q)\setminus {\rm big}_\lambda(S): P\subset S}
\left(\frac{\ell(P)}{\ell(Q)}\right)^{\!1/2}\Theta(P)^2\,\mu(P) & \leq \lambda^{1/2}\,\left(\frac{\ell(S)}{\ell(Q)}\right)^{\!1/2}\!\!\!
\sum_{P\in\hd^{k(Q,M)}(Q): P\subset S}\!\!\!
\Theta(P)^2\,\mu(P)\\
&\leq \lambda^{1/2}\,\left(\frac{\ell(S)}{\ell(Q)}\right)^{\!1/2}\,\Lambda^2\,\Theta(S)^2\,\mu(S)\\
& \leq  c_0^{1/2}\,\left(\frac{\ell(S)}{\ell(Q)}\right)^{\!1/2}\,\Theta(S)^2\,\mu(S),
\end{align*}
taking into account the choice of $\lambda$ for the last estimate.
By \rf{eqgq1}, since $S\in G(Q,M)$, we have
\begin{align}\label{eqgh620}
\left(\frac{\ell(S)}{\ell(Q)}\right)^{\!1/2}\,\Theta(S)^2\,\mu(S) &\leq  2 \Lambda^2
\left(\frac{\ell(S)}{\ell(Q)}\right)^{\!1/2}\,\Theta(S)^2\,\sum_{P\in\hd^{k(Q,M)}(Q): P\subset S}
\left(\frac{\ell(P)}{\ell(S)}\right)^{\!1/2}\mu(P)\\
&\leq2
\sum_{P\in\hd^{k(Q,M)}(Q): P\subset S}
\left(\frac{\ell(P)}{\ell(Q)}\right)^{\!1/2}\Theta(P)^2\,\mu(P)\nonumber
\end{align}
Hence, for $c_0$ small enough,
$$\sum_{P\in\hd^{k(Q,M)}(Q)\setminus {\rm big}_\lambda(S): P\subset S}
\left(\frac{\ell(P)}{\ell(Q)}\right)^{\!1/2}\Theta(P)^2\,\mu(P)  \leq \frac12
\sum_{P\in\hd^{k(Q,M)}(Q): P\subset S}
\left(\frac{\ell(P)}{\ell(Q)}\right)^{\!1/2}\Theta(P)^2\,\mu(P).$$
Consequently, for every $S\in G(Q,M)$,
\begin{equation}\label{eqgh621}
\sum_{P\in\hd^{k(Q,M)}(Q): P\subset S}
\left(\frac{\ell(P)}{\ell(Q)}\right)^{\!1/2}\Theta(P)^2\,\mu(P)\leq 2
\sum_{P\in{\rm big}_\lambda(S): P\subset S}
\left(\frac{\ell(P)}{\ell(Q)}\right)^{\!1/2}\Theta(P)^2\,\mu(P).
\end{equation}

From the conditions $P\subset S$ and $\Theta(P)= \Lambda\,\Theta(S)$ we also get
$$\Theta(P)\leq \Theta(S)\,\frac{\ell(S)^n}{\ell(P)^n} 
= \Lambda^{-1}\,\Theta(P)\,\frac{\ell(S)^n}{\ell(P)^n}.$$
Thus,
\begin{equation}\label{eq**2}
\ell(P)\leq \Lambda^{-1/n}\ell(S).
\end{equation}
Then we derive
\begin{align*}
M^2\,\Theta(Q)^2\,\mu(Q) & \leq 4 \Lambda^{-\frac 1{2n}} 
\sum_{S\in G(Q,M)}  \left(\frac{\ell(S)}{\ell(Q)}\right)^{\!1/2}\sum_{P\in{\rm big}_\lambda(S)}
\Theta(P)^2\,\mu(P),
\end{align*}
which proves the first statement of the lemma.

Concerning the second statement, notice that by \rf{eqgh620}, \rf{eqgh621}, and \rf{eq**2}, we have
\begin{align*}
\Theta(S)^2\,\mu(S) & \leq 2
\sum_{P\in\hd^{k(Q,M)}(Q): P\subset S}
\left(\frac{\ell(P)}{\ell(S)}\right)^{\!1/2}\Theta(P)^2\,\mu(P)\\
& \leq 4 \sum_{P\in{\rm big}_\lambda(S)}
\left(\frac{\ell(P)}{\ell(S)}\right)^{\!1/2}\Theta(P)^2\,\mu(P)\\
& \leq 4 \Lambda^{-\frac1{2n}}\sum_{P\in{\rm big}_\lambda(S)}
\Theta(P)^2\,\mu(P).
\end{align*}
\end{proof}

\vv

We denote 
$$\DB := \DB(M_0) =
\bigcup_{M\geq M_0} \big(\DB(M) \setminus \DB(2M)\big).$$
Remark that the last identity holds because of the polynomial growth of $\mu$.
For each $Q\in \DB$, choose $M(Q)$ such that $Q\in\DB(M(Q))\setminus \DB(2M(Q))$ 
We denote by $\GDF$ (which stands for ``good dominating family") the family of the cubes $S\in\DD_\mu^\PP$ belonging to $G(Q,M(Q))$ for some $Q\in\DB$.
In particular, by the preceding lemma, the cubes $S\in\GDF$ 
satisfy the property that there exists a family $\cI_S\subset \DD_\mu^\PP(S)$
such that
\begin{equation}\label{cond21}
\Theta(P)= \Lambda\,\Theta(S)\quad \text{ and } \quad\ell(P)\geq\lambda_0\,\ell(S)\quad
\mbox{ for all $P\in \cI_S$}
\end{equation}
(with $\lambda_0=\lambda_0(\Lambda)$ as in \rf{eqlambda1}),
and
\begin{equation}\label{cond22}
\mu(S) \leq 4 \Lambda^{2- \frac 1{2n}}
\sum_{P\in \cI_S}\mu(P).
\end{equation}
For a family $\cI\subset\DD_\mu$, we denote
$$\sigma(\cI) = \sum_{P\in \cI}\Theta(P)^2\,\mu(P),$$
and for $p\geq1$,
$$\sigma_p(\cI) = \sum_{P\in \cI}\Theta(P)^p\,\mu(P),$$
so that $\sigma(\cI)=\sigma_2(\cI)$. 
Observe that, for $S\in\GDF$,
\begin{equation}\label{eqiS}
\sigma(\cI_S) = \Lambda^2\Theta(S)^2\,\sum_{P\in \cI_S}\mu(P) \geq \frac14\, \Lambda^{\frac 1{2n}}
\sigma(S),
\end{equation}
by \rf{cond22}.
\vv

\begin{lemma}\label{lemdbnodb}
We have
$$\sum_{Q\in \DB} \EE_\infty(9Q) \lesssim \Lambda^{-\frac1{2n}} \sum_{S\in \GDF}\sigma(\cI_S).$$
\end{lemma}

\begin{proof}
For each $Q\in \DB$, choose $M=M(Q)$ such that $Q\in\DB(M)\setminus \DB(2M)$.
By the definition of $\EE_\infty(9Q)$, we have
$$\EE_\infty(9Q) \lesssim M(Q)^2\, \Theta(Q)^2\,\mu(Q).$$
Then, by Lemma \ref{lem16} we get
\begin{align*}
\sum_{Q\in \DB} \EE_\infty(9Q) & \lesssim  \Lambda^{-\frac1{2n}}
\sum_{Q\in \DB}
\sum_{S\in G(Q,M(Q))}\left(\frac{\ell(S)}{\ell(Q)}\right)^{\!1/2}\sum_{P\in {\rm big}_\lambda(S)}
\Theta(P)^2\mu(P)\\
& \lesssim \Lambda^{-\frac1{2n}}
\sum_{Q\in \DD_{\mu\PP}} \sum_{S\in \GDF:S\subset 9Q}\left(\frac{\ell(S)}{\ell(Q)}\right)^{\!1/2}\,\sigma(\cI_S)\\
&= \Lambda^{-\frac1{2n}} \sum_{S\in \GDF}\sigma(\cI_S)
\sum_{Q\in \DD_{\mu\PP}:9Q\supset S}
\left(\frac{\ell(S)}{\ell(Q)}\right)^{\!1/2}\\
&\lesssim \Lambda^{-\frac1{2n}} \sum_{S\in \GDF}\sigma(\cI_S).
\end{align*}
\end{proof}

\vv


\begin{lemma}
Let $\delta_0\in (0,1)$.
Let $S,P,P'\in\DD_{\mu}$ be such that $P\subset P'\subset S$. Suppose that 
$$\Theta(P)\geq \Lambda\,\Theta(S) \quad\mbox{ and }\quad\Theta(P')\leq \delta_0\,\Theta_\mu(S).$$
Then we have
$$c_1\ell(P) \leq (\delta_0\,\Lambda^{-1})^{1/n} \ell(P')\leq (\delta_0\,\Lambda^{-1})^{1/n} \ell(S) .$$
\end{lemma}

\begin{proof}
This is an immediate consequence of the following:
$$\Theta(P') \,\frac{\ell(P')^n}{\ell(P)^n}\geq \Theta(P)  \geq c\Lambda\,\Theta(S)\geq c\Lambda\delta_0^{-1}\Theta(P').$$
\end{proof}
\vv

\begin{rem}\label{rem*1}
Let 
\begin{equation}\label{eqlambda0}
\lambda=\lambda_0(\Lambda)= \frac{c_0}{\Lambda^{4}}.
\end{equation}
By the preceding lemma, if 
$(\delta_0 \Lambda^{-1})^{1/n} < c_1\lambda,$
or equivalently,
\begin{equation}\label{eqdeltaM}
\delta_0< c' \Lambda^{1-4n},
\end{equation}
then, for any $S\in\GDF$ and $P\in \cI_S$, 
 there does not exist any cube $P'\in\DD_\mu$ satisfying 
$$P\subset P'\subset S
\quad\mbox{ and }\quad \Theta(P)\leq \delta_0\,\Theta(S).$$
\end{rem}
\vv

Another easy (but important) property of the family $\GDF$ is stated in the next lemma.

\begin{lemma}\label{lementrecubs}
Let $S_1,S_2\in\GDF$ be such that $S_2\subsetneq S_1$ and $\Theta(S_1)=\Theta(S_2)$. Then
there exist $Q\in\DB$ and $Q'\in \DD_\mu$ such that $Q'\subset 9Q$, $\ell(Q)=\ell(Q')$ and
$S_1\supset Q'\supset S_2$, with $S_2\in G(Q,M(Q))$ for some $M(Q)\geq M_0$.
\end{lemma}

\begin{proof}
This is due to the fact that, by definition, there exists $Q\in\DB$ such that  $S_2\in G(Q,M)$ for 
some $M=M(Q)\geq M_0$. So $S_2\in\hd^{k(Q,M)-k_\Lambda}(Q)$ and $S_2\subset 9Q$. Then 
$\ell(Q)\leq\ell(S_1)$, because otherwise $S_2\not\in\hd^{k(Q,M)-k_\Lambda}(Q)$ since $S_2$ would not be  
maximal among the cubes $S$ contained in $9Q$ such that $\Theta(S)\geq A_0^{(k(Q,M)-k_\Lambda)n}\,\Theta(Q)$ (as $S_1$ is also contained in $9Q$ and $\Theta(S_1)= A_0^{(k(Q,M)-k_\Lambda)n}\,\Theta(Q)$). The fact that $\ell(S_1)\geq\ell(Q)\geq\ell(S_2)$ implies the existence of a cube $Q'$ such as the one in the lemma.
\end{proof}
\vv



\section{The cubes with moderate decrement of Wolff energy and the associated tractable trees}\label{sec6}

Given $R\in\DD_\mu^\PP$, we denote 
$$\HD(R) = \hd^{k_\Lambda}(R).$$
Also, we take $\delta_0 \leq c' \Lambda^{1-4n}$, so that \rf{eqdeltaM} holds.
We let $\LD(R)$ be the family of cubes $Q\in\DD_{\mu}$
which
are maximal and satisfy
$$\ell(Q)< \ell(R)\quad \mbox{ and }\quad\Theta(Q) \leq \delta_0\,\Theta(R).$$
We denote by $\NDB(R)$ (which stands for ``near $\DB$'') the family of cubes $Q$ which do not belong to $\LD(R)\cup\HD(R)$ and satisfy the following:
\begin{itemize}
\item $\ell(Q)<\lambda\,\ell(R)$, with $\lambda= c_0\,\Lambda^{-4}$ as in \rf{eqlambda0}, and
\item there exists another cube $Q'\in\DB$ of the same generation as $Q$ such that $Q'\subset 20Q$.
\end{itemize}
We let $\bad(R)$ be the family of maximal cubes from $\LD(R)\cup\HD(R)\cup\NDB(R)$ (not necessarily contained in $R$) and we denote
$$\sss(R)= \bad(R)\cap\DD_\mu(R).$$

\vv
\begin{lemma}\label{lemred*}
For any cube $R\in\GDF$, we have $\cI_R\subset\big(\HD(R)\setminus \NDB(R)\big)\cap\sss(R)$, and thus
$$\sigma(\HD(R)\cap\sss(R)\setminus \NDB(R))\geq \frac14\,\Lambda^{\frac1{2n}}\,\sigma(R)
.$$
\end{lemma}

Remark that $\cI_R$ is the family of cubes defined in \rf{cond21}.

\begin{proof}
Recall that, by \rf{eqiS},
$$\sigma(\cI_R)\geq \frac14\,\Lambda^{\frac1{2n}}\,\sigma(R)
.$$
By the choice of $\delta_0$ and Remark \ref{rem*1}, there does not exist any $Q\in\LD(R)$ which 
contains any cube from $\cI_R$. Also, the cubes from $Q\in \cI_R$ satisfy $\ell(Q)\geq \lambda\ell(R)$ and so there does exist any cube $Q'\in\NDB(R)$ such that $Q\subset Q'\subset R$. So $\cI_R\subset\HD(R)\cap\sss(R)\setminus \NDB(R)$. The last statement in the lemma follows from \rf{cond22}.
\end{proof}

\vv

We say that a cube $R\in\DD_\mu^\PP$ has moderate decrement of Wolff energy, and we write
$R\in\MDW$, if
$$\sigma(\HD(R)\cap\sss(R)\setminus \NDB(R))\geq B^{-1}\,\sigma(R),$$
where $$B= \Lambda^{\frac1{100n}}.$$  In particular, by Lemma \ref{lemred*}, $\GDF\subset\MDW$.
In this section we will show how to associate to each $R\in\MDW$ a suitable family of tractable trees (to be defined later).

Now we introduce the enlarged cubes $e_j(R)$. Given $j\geq0$ and $R\in\DD_{\mu,k}$, 
we let 
$$e_j(R) = R \cup \bigcup Q,$$
where the last union runs over the cubes $Q\in\DD_{\mu,k+1}$ such that
\begin{equation}\label{eqxrq83}
\dist(x_R,Q)< \frac{\ell(R)}2 + 2j\ell(Q).
\end{equation}
Notice that, since $\diam(Q)\leq \ell(Q)$,
\begin{equation}\label{eqxrq84}
\supp\mu\cap B\big(x_R,\tfrac12\ell(R) + 2j\ell(Q)\big)\subset e_j(R) \subset B\big(x_R,\tfrac12\ell(R) + (2j+1)\ell(Q)\big).
\end{equation}
Also,  we have
\begin{equation}\label{eqqj8d}
e_j(R)\subset 2R\quad \mbox{ for $0\leq j\leq \frac34 A_0$,}
\end{equation}
since, for any $Q\in\DD_{\mu,k+1}$ satisfying \rf{eqxrq83}, its parent satisfies $\wh Q$
$$\dist(x_R,\wh Q)< \frac{\ell(R)}2 + 2j A_0^{-1}\ell(\wh Q) \leq 2\ell(R).$$

For $R\in\MDW$, we let
$$\sss(e_j(R)) = \bad(R) \cap\DD_\mu(e_j(R)),$$
where $\DD_\mu(e_j(R))$ stands for the subfamily of the cubes from $\DD_\mu$ which are contained in
$e_j(R)$ and have side length ta most $\ell(R)$.

\vv
\begin{lemma}\label{lem5.2}
For any $R\in\MDW$ there exists some $j$, with $10\leq j\leq A_0/4$ such that
\begin{equation}\label{eqsigmaj}
\sigma(\HD(R)\cap\sss(e_{j}(R))\setminus \NDB(R)) \leq B^{1/4} \sigma(\HD(R)\cap\sss(e_{j-10}(R))\setminus \NDB(R)),
\end{equation}
 assuming $A_0$ big enough (possibly depending on $n$). 
\end{lemma}


\begin{proof}
To shorten notation, here we write $\wt\HD(R)=\HD(R)\setminus \NDB(R)$.
Given $R\in\MDW$, suppose that such $j$ does not exist. Let $j_0$ be the largest integer which is multiple of $10$ and smaller that $A_0/4$.
Then we get
\begin{align*}
\sigma(\wt\HD(R)\cap\sss(e_{j_0}(R))) & \geq B^{1/4}\sigma(\wt\HD(R)\cap\sss(e_{j_0-10}(R)))\\
&\geq
\ldots \geq \big(B^{\frac14}\big)^{\frac{j_0}{10}-1}\sigma(\wt\HD(R)\cap\sss(R)) \geq B^{\frac{j_0}{40}-\frac54}\sigma(R).
\end{align*}
By \rf{eqqj8d}, we have $e_{j_0}(R)\subset 2R$ and thus
$$\sigma(\wt\HD(R)\cap\sss(e_{j_0}(R)))\leq \sum_{Q\in\wt\HD(R)\cap \sss(e_{j_0}(R))} \Lambda^2\Theta(R)^2\mu(Q)\leq 
\Lambda^2\Theta(R)^2\mu(2R).$$
Since $R$ is $\PP$-doubling so from $\DD^{db}_\mu$, denoting by $\wh R$ the parent of $R$,  we derive 
\begin{equation}\label{eqdoub*11}
\mu(2R)\leq \mu(2B_{\wh R}) \leq \frac{\ell(\wh R)^{n+1}}{\ell(R)}\,\PP(R) \leq C_d\,A_0^{n+1}
\mu(2B_R)\leq C_0\,C_d\,A_0^{n+1}
\mu(R).
\end{equation}
So we deduce that
$$B^{\frac{j_0}{40}-\frac54}\sigma(R)\leq C_0\,C_d\,A_0^{n+1}\,\Lambda^2\sigma(R),$$
or equivalently, recalling the choice of $B$ and $C_d$,
$$\Lambda^{\frac{1}{100n}\left(\frac{j_0}{40}-\frac54\right) -2} \leq 4 C_0\,A_0^{2n+1}.$$
Since $\Lambda\geq A_0^n$ and $j_0\approx A_0$, it is clear that this inequality is violated if $A_0$ is big enough.
\end{proof}
\vv

Given $R\in\MDW$,  let $j\geq 10$ be minimal such that \rf{eqsigmaj} holds. We denote
$h(R)=j-10$. 
We write
$$e(R) = e_{h(R)},\qquad e'(R) = e_{h(R)+1}, \quad e''(R)=e_{h(R)+2}, \quad e^{(k)}(R) = e_{h(R)+k}(R),$$
for $k\geq 1$. 
We let
\begin{align*}
B(e(R)) &= B\big(x_R,(\tfrac12 + 2A_0^{-1}h(R))\ell(R)\big),\\
B(e'(R)) & = B\big(x_R,(\tfrac12 + 2A_0^{-1}(h(R)+1))\ell(R)\big),\\
B(e''(R)) & = B\big(x_R,(\tfrac12 + 2A_0^{-1}(h(R)+2))\ell(R)\big),\\
B(e^{(k)}(R)) & = B\big(x_R,(\tfrac12 + 2A_0^{-1}(h(R)+k))\ell(R)\big).
\end{align*}
By construction we have 
$$B(e'(R))\cap\supp\mu\subset e'(R),$$
and analogously replacing $e'(R)$ by $e(R)$ or $e''(R)$.
Remark also that
$$e(R)\subset B(e'(R))\quad \text{ and }\quad \dist(e(R),\partial B(e'(R))) \geq A_0^{-1}\ell(R),$$
and, analogously,
$$e'(R)\subset B(e''(R))\quad \text{ and }\quad \dist(e'(R),\partial B(e''(R))) \geq A_0^{-1}\ell(R).$$

\vv

\begin{lemma}\label{lem-calcf}
For each $R\in\MDW$  we have
$$B(e''(R)) \subset (1+8A_0^{-1})\,\,B(e(R)) \subset B(e^{(6)}(R)),$$
and more generally, for $k\geq 2$ such that $h(R)+k-2\leq A_0/2$,
$$B(e^{(k)}(R)) \subset (1+8A_0^{-1})\,\,B(e^{(k-2)}(R)) \subset B(e^{(k+4)}(R)).$$
Also,
$$B(e^{(10)}(R))\subset B\big(x_R,\tfrac32 \ell(R)\big).$$
\end{lemma}

This lemma is proven exactly in the same way as \cite[Lemma 4.4]{DT} and so we omit the proof.

\vv

Next we need to define some families of cubes that can be considered as ``generalized trees". First we introduce some additional notation regarding the stopping cubes. For $R\in\MDW$ we write $\sss(e'(R))=\sss(e_{h(R)+1}(R))$. Also, we denote 
$$\HD_{1}(R) = \sss(R)\cap \HD(R)\setminus \NDB(R),$$
$$\HD_1(e(R)) = \sss(e(R))\cap \HD(R)\setminus \NDB(R),$$
and
$$\HD_{1}(e'(R)) = \sss(e'(R))\cap \HD(R)\setminus \NDB(R).$$
Also, we set 
$$\HD_2(e'(R)) = \bigcup_{Q\in \HD_1(e'(R))} (\sss(Q)\cap \HD(Q)\setminus \NDB(Q))$$
and
\begin{equation}\label{eqstop2}
\sss_2(e'(R)) = \big(\sss(e'(R)) \setminus \HD_1(e'(R))\big) \cup \bigcup_{Q\in \HD_1(e'(R))} \sss(Q).
\end{equation}

Given $R\in\MDW$, we let $\TT_\sss(e'(R))$ be the family of cubes made up of $R$ and all the cubes of the next generations which are contained in $e'(R)$ but are not 
strictly contained in any cube from $\sss_2(e'(R))$.

We define now the family of negligible cubes. We say that a cube $Q\in\TT_\sss(e'(R))$ is negligible for $\TT_\sss(e'(R))$, and we write $Q\in\Neg(e'(R))$ if
 if
 there does not exist any cube from $\TT_\sss(e'(R))$ that contains $Q$ and is $\PP$-doubling.

\vv
\begin{lemma}\label{lemnegs}
Let $R\in\MDW$. If $Q\in\Neg(e'(R))$, then $Q\subset e'(R)\setminus R$, $Q$ is not contained in any cube from $\HD_1(e'(R))$, and
\begin{equation}\label{eqcostat}
\ell(Q) \gtrsim \delta_0^{2}\,\ell(R).
\end{equation}
\end{lemma}

This lemma is proven in the same way as \cite[Lemma 4.5]{DT} and so we omit the proof.

\vv

The cubes from $\sss_2(e'(R))$ need not be $\PP$-doubling. This may cause some problems for some of the estimates involving the Riesz transform localized around the
trees $\TT_\sss(e'(R))$ that will be required later. For this reason, we need to consider enlarged versions of them. For $R\in\MDW$, we let $\End(e'(R))$ be family made up of the following cubes:
\begin{itemize}
\item the cubes from $\sss_2(e'(R))\cap \Neg(e'(R))$,
\item the cubes that are contained in any cube from $\sss_2(e'(R))\setminus \Neg(e'(R))$, which are $\PP$-doubling, and moreover are maximal.
\end{itemize}
Notice that all the cubes from $\End(e'(R))$ are $\PP$-doubling, with the possible exception of the ones from $\Neg(e'(R))$.
We let $\TT(e'(R))$ be the family of cubes that are contained in $e'(R)$ and are not 
strictly contained in any cube from $\End(e'(R))$.

Given $R\in\MDW$, we say that $\TT(e'(R))$ is tractable (or that $R$ is tractable) if 
$$\sigma(\HD_2(e'(R)))\leq B\,\sigma(\HD_1(e(R))).$$
In this case we write $R\in\Trc$.

 \vv
 Our next objective consists in showing how we can associate a family of tractable trees to any $R\in\GDF$, so that, roughly speaking, we can reduce the estimate of $\sigma(\DB)$ to estimating the Haar coefficients of $\RR\mu$ from below on such family of tractable cubes.
First we need the following lemma.

\begin{lemma}\label{lemalg1}
Let $R\in\MDW$ be such that $\TT(e'(R))$ is not tractable. Then there exists a family $\GH(R)\subset
\HD_1(e'(R))\cap\MDW$ satisfying:
\begin{itemize}
\item[(a)] The balls $B(e''(Q))$, with $Q\in\GH(R)$ are pairwise disjoint.
\item[(b)] For every $Q\in\GH(R)$, $\sigma(\HD_1(e(Q)))\geq \sigma(\HD_1(Q))\geq B^{1/2}\sigma(Q)$.
\item[(c)] $$B^{1/4} \sum_{Q\in\GH(R)} \sigma(\HD_1(e(Q))) \gtrsim \sigma(\HD_2(e'(R))).$$
\end{itemize}
\end{lemma}

This lemma is proven in the same way as \cite[Lemma 4.6]{DT} and so we omit the proof again.

Remark that the property (c) and the fact that $R\not\in\Trc$ yield
\begin{equation}\label{eqiter582}
\sum_{Q\in\GH(R)} \sigma(\HD_1(e(Q))) \gtrsim B^{3/4}\,\sigma(\HD_1(e(R))),
\end{equation}
which will be suitable for iteration in the arguments below.

\vv

Given $R\in\GDF$, we will construct now a subfamily of cubes from $\DD_\mu^\PP$ generated by $R$,
which we will denote by $\Gen(R)$, by iterating the construction of Lemma \ref{lemalg1}.
We follow the next algorithm.
Given $R\in\GDF$, we denote 
$$\Gen_0(R) = \{R\}.$$
If $R\in\Trc$, we set $\Gen_1(R)=\varnothing$, and otherwise
$$\Gen_1(R) = \GH(R),$$
where $\GH(R)$ is defined in Lemma \ref{lemalg1}.
For $j\geq 2$, we set
$$\Gen_{j}(R) = \bigcup_{Q\in\Gen_{j-1}(R)\setminus \Trc} \GH(Q),$$
where $\GH(Q)$ is defined in Lemma \ref{lemalg1}. For $j\geq0$, we also set
$$\Trc_j(R) = \Gen_j(R)\cap\Trc,$$
and
$$\Gen(R) = \bigcup_{j\geq0}\Gen_j(R),\qquad\Trc(R) = \bigcup_{j\geq0}\Trc_j(R).$$
\vv

From the construction above, we get the following lemma, which is proven in the same way as \cite[Lemma 4.7]{DT} and so we omit the proof.

\begin{lemma}\label{eqtec74}
For $R\in\GDF$, we have
$$\bigcup_{Q\in\Gen(R)\cup \Trc(R)}Q \subset B(e''(R)).$$
Also,
\begin{equation}\label{eqiter*44}
\sigma(\HD_1(e(R)))\leq \sum_{j\geq0} B^{-j/2}\sum_{Q\in\Trc_j(R)}\sigma(\HD_1(e(Q))).
\end{equation}
\end{lemma}

\vv


\section{The layers $\sF_j^h$ and $\sL_j^h$ and the typical tractable trees}\label{sec7}

 We denote
$$\sF_j= \big\{R\in\GDF:\Theta(R)=A_0^{nj}\big\},$$
so that 
$$\GDF= \bigcup_{j\in\Z} \sF_j.$$
Next we split $\sF_j$ into layers $\sF_j^h$, $h\geq1$, which are defined as follows:
$\sF_j^1$ is the family of maximal cubes from $\sF_j$, and by induction
$\sF_j^h$ is the family of maximal cubes from $\sF_j\setminus \bigcup_{k=1}^{h-1} \sF_j^{h-1}$.
So we have the splitting
$$\GDF= \bigcup_{j\in\Z}\,\bigcup_{h\geq1} \sF_j^h.$$

Our next objective is to choose a suitable subfamily $\sL_j^h\subset \sF_j^h$, for each $j,h$.
By Theorem 9.31 from \cite{Tolsa-llibre}, there
is a family $J_0\subset \sF_j^h$ such that\footnote{Actually the property 1) is not stated in that theorem, however this can be obtained by preselecting a subfamily of maximal balls from $\sF_j^h$
with respect to inclusion and then applying the theorem to the maximal subfamily.}
\begin{itemize}
\item[1)] no ball $B(e^{(4)}(Q))$, with $Q\in J_0$, is contained in any other ball  
$B(e^{(4)}(Q'))$, with $Q'\in \sF_j^h$, $Q'\neq Q$,
\item[2)] the balls $B(e^{(4)}(Q))$, with $Q\in J_0$, have finite superposition, 
and
\item[3)] every ball $B(e^{(4)}(Q))$, with $Q\in\sF_j^h$ is contained in some ball 
$(1+8A_0^{-1})\,B(e^{(4)}(R))$, with $R\in J_0$. Consequently,
$$\bigcup_{Q\in\sF_j^h} B(e^{(4)}(Q)) \subset \bigcup_{R\in J_0} (1+8A_0^{-1})\,B(e^{(4)}(R)).$$
\end{itemize}

From the finite superposition property 2), by rather standard arguments which are analogous to the
ones in the proof of Besicovitch's covering theorem in \cite[Theorem 2.7]{Mattila-llibre}, say, 
one deduces that $J_0$ can be split into 
 $m_0$ subfamilies $J_1,\ldots, J_{m_0}$ such that, for each $k$,  the balls $\{B(e^{(4)}(Q)): Q\in J_k\}$  are pairwise disjoint, with $m_0\leq C(A_0)$.

By the condition 3),  
we get
\begin{align*}
\sum_{Q\in \sF_j^h} \sigma(\HD_1(Q)) & = \Lambda^2 A_0^{2nj}\sum_{Q\in \sF_j^h} 
\sum_{P\in\HD_1(Q)} \mu(P)\\
& \leq \Lambda^2 A_0^{2nj}\sum_{R\in J_0} \sum_{\substack{Q\in \sF_j^h:\\ B(e^{(4)}(Q))
\subset (1+8A_0^{-1})B(e^{(4)}(R))}} \sum_{P\in\HD_1(Q)} \mu(P).
\end{align*}
Observe that, for $R\in J_0$ and $Q\in\sF_j^h$ such that  $B(e^{(4)}(Q))
\subset (1+8A_0^{-1})\,B(e^{(4)}(R))$,
 any cube $P\in\HD_1(Q)$ is contained in some cube from $\HD_1(e^{(10)}(R))$ since
$\supp\mu\cap (1+8A_0^{-1})B(e^{(4)}(R))\subset e^{(10)}(R)$,  by Lemma \ref{lem-calcf}.
 Using also that the cubes from $\sF_j^h$ are disjoint and Lemma \ref{lem5.2}, we deduce that
\begin{align*}
\sum_{Q\in \sF_j^h} \sigma(\HD_1(Q)) & \leq \Lambda^2 A_0^{2nj}\sum_{R\in J_0}
\sum_{P\in\HD_1(e^{(10)}(R))} \mu(P) \\ &= \sum_{R\in J_0}\sigma(\HD_1(e^{(10)}(R))) \leq B^{1/4}\sum_{R\in J_0}\sigma(\HD_1(e(R))).
\end{align*}
Next we choose $\sL_j^h=J_k$ to be the family such that
$$\sum_{Q\in J_k}\sigma(\HD_1(e(Q)))$$
is maximal among $J_1,\ldots,J_{m_0}$, so that
\begin{align*}
\sum_{Q\in \sF_j^h} \sigma(\HD_1(Q))\leq m_0\,B^{1/4}\sum_{Q\in \sL_j^h}\sigma(\HD_1(e(Q))).\end{align*}
So we have:

\begin{lemma}\label{lemljh}
The family $\sL_j^h$ satisfies:
\begin{itemize}
\item[(i)] no ball $B(e^{(4)}(Q))$, with $Q\in \sL_j^h$, is contained in any other ball  
$B(e^{(4)}(Q'))$, with $Q'\in \sF_j^h$, $Q'\neq Q$,
\item[(ii)] the balls $B(e^{(4)}(Q))$, with $Q\in \sL_j^h$, are pairwise disjoint, 
and
\item[(iii)] 
 $$\sum_{Q\in \sF_j^h} \sigma(\HD_1(Q)) \leq m_0\, B^{1/4} 
\sum_{Q\in \sL_j^h}\sigma(\HD_1(e(Q))).$$
\end{itemize}
\end{lemma}

We denote 
$$\sL_j= \bigcup_{h\geq 1}\sL_j^h,\qquad \sL=\sL(\GDF)= \bigcup_{j\in\Z}\sL_j =
\bigcup_{j\in\Z}\,\bigcup_{h\geq 1}\sL_j^h.$$
By Lemma \ref{lemljh} (iii), we have
\begin{align}\label{eqover5}
\sum_{R\in \GDF}\sigma(\HD_1(R)) & = \sum_{j\in\Z, \,h\geq0}\,
\sum_{R\in\sF_j^h} \sigma(\HD_1(R)) \\ 
& \leq m_0\,B^{1/4} \!\!\sum_{j\in\Z, \,h\geq0}\,\sum_{R\in \sL_j^h}\sigma(\HD_1(e(R))) = m_0\,B^{1/4} \!\!\sum_{R\in \sL(\GDF)}\sigma(\HD_1(e(R))).\nonumber
\end{align}

\vv

Our next objective consists of proving the next lemma, which is the main technical achievement in this section.
Although the statement looks similar to \cite[Lemma 5.2]{DT}, the proof is very different. 
The cubes from the $\NDB(\,\cdot\,)$ play an important role in the arguments. In fact, the main reason for
the introduction of the stopping condition involving the family $\NDB(\,\cdot\,)$ in Section \ref{sec6} is that it allows to prove
the next lemma.

\begin{lemma}\label{lemimp9}
There exists some constant $C_2$ such that, for all $P\in\DD_\mu$ and all $k\geq0$,
$$\#\big\{R\in\sL(\GDF):\exists \,Q\in\Trc_k(R) \mbox{ such that } P\in\TT(e'(Q))\big\}\leq C_2\,(\log\Lambda)^2.$$
\end{lemma}

\begin{proof}
First notice that if $R\in\sL(\GDF)$ and $Q\in\Trc_k(R)$ are such that $P\in\TT(e'(Q))\setminus\Neg(e'(Q))$, then there exists some $\PP$-doubling cube that contains $P$ and belongs to $\TT_\sss(e'(Q))$, by the definition of
the family $\Neg(e'(Q))$. We denote by $\wt P$ be the smallest one. This satisfies
$$\delta_0\,\Theta(Q)\lesssim \Theta(\wt P)\leq \Lambda^2\,\Theta(Q),$$
or equivalently, $\Lambda^k\Theta(R)\in [\Lambda^{-2}\Theta(\wt P),C\delta_0^{-1}\Theta(\wt P)]$.
Hence, if $\Theta(R)=A_0^{nj}$, it follows that 
$$-C\log\Lambda\leq j + c\,k\log\Lambda - c'\log\Theta(\wt P)\leq C|\log\delta_0| = C'\log\Lambda.$$
Thus, $R$ belongs at most to $C''\log\Lambda$ families  $\sL_j$ such that
there exists $Q\in\Trc_k(R)$ such that $P\in\TT(e'(Q))\setminus\Neg(e'(Q))$

Suppose now that there exists $Q\in\Trc_k(R)$ such that $P\in\Neg(e'(Q))\subset\TT(e'(Q))$.
In this case, by Lemma \ref{lemnegs}, $\ell(P) \gtrsim \delta_0^{-2}\,\ell(Q)$. Hence, there
are at most $C\,|\log\delta_0|\approx \log\Lambda$ cubes $Q$ such that 
$P\in\TT(e'(Q))\cap\Neg(e'(Q))$, which in turn implies that again there are at most 
$C'''\log\Lambda$ families  $\sL_j$ such that
there exists $Q\in\Trc_k(R)$ satisfying $P\in\TT(e'(Q))\cap\Neg(e'(Q))$.

By the previous discussion, to prove
the lemma, it is enough to show that, for each $j\in\Z$, $P\in\DD_\mu$, $k\geq0$,
\begin{equation}\label{eqlj83}
\#\sL_j(P,k) \leq C_3\log\Lambda,
\end{equation}
where
$$\sL_j(P,k)= \big\{R\in\sL_j:\exists \,Q\in\Trc_k(R) \mbox{ such that } P\in\TT(e'(Q))\big\}.$$

To prove \rf{eqlj83}, let $R_0$ be a cube in $\sL_j(P,k)$ with maximal side length, and let $h_0$ be such that $R_0\in \sL_j^{h_0}(P,k)\equiv \sL_j(P,k)\cap \sL_j^{h_0}$. 

\begin{claim}
Let $R_1$ be another cube from $\sL_j(P,k)$, and let $h_1$ be such that $R_1\in \sL_j^{h_1}(P,k)$. 
Then $h_1\geq h_0$.
\end{claim}

\begin{proof}
Suppose that $h_1< h_0$.
Let $R_0^{h_1}$ be the cube that contains $R_0$ and belongs to
$\sF_j^{h_1}$. Observe that, by Lemma \ref{eqtec74},
\begin{align*}
P& \subset  B(e''(R_0)) \cap B(e''(R_1)) \subset B(x_{R_0},\tfrac32\ell(R_0)) \cap
B(x_{R_1},\tfrac32\ell(R_1)).
\end{align*}
Since $\ell(R_0)\geq\ell(R_1)$, we infer that
$$B(x_{R_1},\tfrac32\ell(R_1)) \subset B(x_{R_0},\tfrac92\ell(R_0)).$$
As $x_{R_0}\in B(x_{R_0}^{h_1}, \frac12\ell(R_0^{h_1}))$ and
$\ell(R_0)\leq A_0^{-1}\ell(R_0^{h_1})$, we deduce that
\begin{align*}
 B(x_{R_0},\tfrac92\ell(R_0))  & \subset B(x_{R_0}^{h_1}, \tfrac12\ell(R_0^{h_1}) +
\tfrac92\ell(R_0)) \subset B(x_{R_0}^{h_1}, \tfrac12\ell(R_0^{h_1})\! +
\tfrac92 A_0^{-1}\ell(R_0^{h_1})) \\
&\subset B(x_{R_0}^{h_1}, (\tfrac12 + 8 A_0^{-1})\ell(R_0^{h_1})) \subset B(e^{(4)}(R_0^{h_1})),
\end{align*}
where the lat inclusion follows from the definition of $B(e^{(4)}(R_0^{h_1}))$.
Then we deduce that
$$ B(e^{(4)}(R_1))\subset B(x_{R_1},\tfrac32\ell(R_1)) \subset
 B(e^{(4)}(R_0^{h_1})),$$
which contradicts the property (i) of the family $\sL_j^{h_1}$ in Lemma \ref{lemljh}, because $R_1\neq R_0^{h_1}$.
\end{proof}

\begin{claim}
Let $R_1$ be another cube from $\sL_j(P,k)$, and let $h_1$ be such that $R_1\in \sL_j^{h_1}(P,k)$. 
Then 
\begin{equation}\label{eqclaimh1}
h_1\leq h_0+C\,\log\Lambda.
\end{equation}
\end{claim}

\begin{proof}
Suppose that $h_1> h_0+1$. This implies that there are cubes
$\{R_1^h\}_{h_0+1\leq h \leq h_1-1}$ such that $R_1^h\in\sF_j^h$, with
$$R_1^{h_0+1}\supset R_1^{h_0+2}\supset\ldots \supset R_1^{h_1-1}\supsetneq R_1^{h_1}=R_1.$$
Observe now that $\ell(R_1^{h_0+1})\leq \ell(R_0)$. Otherwise, there exists some cube
$R_1^{h_0}\in\sF_j^{h_0}$ that contains $R_1^{h_0+1}$ with 
$$\ell(R_1^{h_0})\geq A_0\,\ell(R_1^{h_0+1})\geq A_0\,\ell(R_0),$$
Since $P \subset  B(e''(R_0)) \cap B(e''(R_1))$, arguing as in the previous claim, we deduce that
$B(e^{(4)}(R_0))\subset B(e^{(4)}(R_1^{h_0}))$, which contradicts again the property (i) of the family $\sL_j^{h_0}$, as above. So we have
$$\ell(R_1^h)\leq \ell(R_1^{h_0+1})\leq \ell(R_0)\quad \mbox{ for $h\geq h_0+1$.}$$

By the construction of $\Trc_k(R_0)$, there exists a sequence of cubes
$S_0=R_0, S_1, S_2, \ldots, S_k=Q$ such that 
$$ S_{i+1}\in \GH(S_i)\; \mbox{ for $i=0,\ldots,k-1$,}$$
and $P\in\TT(e'(S_{k}))$. In case that $P$ is contained in some $Q'\in\HD_1(e'(Q))=\HD_1(e'(S_k))$, we write $S_{k+1}=Q'$, and we let $\tilde k:=k+1$. Otherwise, we let $\tilde k:=k$.
In any case, obviously we have  $\ell(S_{i+1})<\ell(S_i)$ for all $i$.
So, for each $h$ with $h_0+1\leq h\leq h_1$ there is some $i=i(h)$ such that 
\begin{equation}\label{eq0asd}
\ell(S_i)>\ell(R_1^h)\geq \ell(S_{i+1}),
\end{equation}
with $0\leq i \leq \tilde k$, where we understand that $S_{\tilde k +1}= P$. 
We claim that either $i\lesssim 1$ or $i=\wt k$, with the implicit constant depending on $n$. 
Indeed, in the case $i<\wt k$, let $T\in\DD_\mu$ be such that $T\supset S_{i+1}$ and $\ell(T)=\ell(R_1^h)$.
Notice that, since $2R_1^h\cap 2T\neq \varnothing$ (because both $R_1^h$ and $2T$ contain $P$) and 
$\ell(R_1^h)=\ell(T)$, we have
\begin{equation}\label{eq1asd}
\PP(T) \approx \PP(R_1^h)\approx \Theta(R_1^h)=\Theta(R_0).
\end{equation}
On the other hand, 
\begin{equation}\label{eq2asd}
\Theta(T)\geq \delta_0\,\Theta(S_i)
\end{equation} 
because otherwise 
either $T\in\LD(S_i)$ or it is contained in some cube from $\LD(S_i)\cup\NDB(S_i)$. In any case, this would imply that $S_{i+1}$ does not belong to $\HD_1(e'(S_i))$.
Thus, from \rf{eq1asd} and \rf{eq2asd}
we derive that
$$\Theta(R_0)\gtrsim \delta_0\,\Theta(S_i) =  \delta_0\,\Lambda^i\,\Theta(R_0).$$
Hence  $\Lambda^i\lesssim\delta_0^{-1}$, which yields $i\lesssim_n 1$ if $i<\wt k$, as claimed.

The preceding discussion implies that, in order to prove \rf{eqclaimh1}, it suffices to show that, for each fixed 
$i=0,\ldots,\tilde k$, there are at most $C\log\Lambda$ cubes $R_1^h$ satisfying 
\rf{eq0asd}. To this end, suppose first that $i<\tilde k$. It is easy to check that there is at most one 
cube  $R_1^h$ satisfying 
\rf{eq0asd} such that
\begin{equation}\label{eqr1hh}
\ell(R_1^h)\leq \lambda\,\ell(S_i).
\end{equation}
Indeed, if otherwise $R_1^h$ and $R_1^{h'}$ are such that
\begin{equation}\label{eqigfj4}
\lambda\,\ell(S_i)\geq \ell(R_1^h) >\ell(R_1^{h'})> \ell(S_{i+1}),
\end{equation}
then, by Lemma \ref{lementrecubs},
there exist $T_a\in\DB$ and $T_b\in \DD_\mu$ such that $T_b\subset 9T_a$, $\ell(T_a)=\ell(T_b)$ and
$R_1^h\supsetneq T_b\supsetneq R_1^{h'}$.
Now, let $T_c\in\DD_\mu$ be such that $T_c\supset S_{i+1}$ and $\ell(T_c)=\ell(T_a)$.
 Since $T_a\in\DB$, we infer that $T_c\in\NDB(S_i)$, taking into account that 
 $2T_b\cap 2T_c\neq \varnothing$, $T_b\subset 9T_a$, and $\ell(T_c)>\lambda\,\ell(S_i)$.
 This is a contradiction, because this would imply that either $T_c\in\sss(e'(S_i))$ or 
  $T_c$ is contained in some cube from $\sss(e'(S_i))$, which ensures that
  $S_{i+1}\not\in\GH(S_i)$ (notice that we are using the fact that $i<\tilde k$).
So \rf{eqr1hh} holds. Clearly this implies that
there are at most $C|\log\lambda|\approx \log\Lambda$ cubes $R_1^h$ satisfying 
\rf{eq0asd}. 

In the case $i=\tilde k$, the same argument as above shows that if
$R_1^h$ and $R_1^{h'}$ satisfy \rf{eqigfj4},
then the cube $T_c$ in the preceding paragraph belongs to $\NDB(S_{\tilde k})$ again.
So again either $T_c\in\sss(e'(S_{\tilde k}))$ or $T_c$ is contained in some cube from 
$\sss(e'(S_{\tilde k}))$. As a consequence, by the definition of the family $\End(e'(R))$ and Lemma
 \ref{lemdobpp}, if we denote by $T_m$ the $m$-th descendant of $T_c$ which contains $P$, it follows that 
$$\Theta(T_m)\lesssim A_0^{-m/2}\,\PP(T_c)\approx A_0^{-m/2}\,\PP(T_b) \leq A_0^{-m/2}\,\PP(R_1^{h'})
\quad \mbox{for all $m\geq1$.}$$
The last inequality follows from the fact that we can assume that $R_1^{h'}\in G(T_a,M)$, for some $M\geq M_0$, by Lemma \ref{lementrecubs}. Since any cube $R_1^{h''}$ with $h''>h'$ is contained in a cube $2T_m$ with $\ell(T_m)\approx\ell(R_1^{h''})$ for some $m\geq h'-h''-1$, we deduce that
\begin{align*}
\Theta(R_1^{h''}) & \lesssim \Theta(2T_m)\lesssim \Theta(T_{m-1})\lesssim A_0^{-(m-1)/2}\,\PP(R_1^{h'})\\
&\lesssim A_0^{-(h'-h'')/2}\,\PP(R_1^{h'})\approx A_0^{-(h'-h'')/2}\,\Theta(R_1^{h'}).
\end{align*}
Since $\Theta(R_1^{h'}) = \Theta(R_1^{h''})$, we get $|h'-h''|\lesssim1$. Consequently, in the case $i=\tilde k$ there are again at most $C|\log\lambda|\approx \log\Lambda$ cubes $R_1^h$ satisfying 
\rf{eq0asd}. 
\end{proof}

To prove the lemma, notice that each family $\sL_j^h(P,k)$ consists of a single cube, at most.
Indeed, if $R\in\sL_j^h(P,k)$, then $P\subset B(e''(R))$. Thus if $R,R'\in\sL_j^h(P,k)$,
then $B(e''(R))\cap B(e''(R'))\neq\varnothing$, which cannot happen if $R\neq R'$.
From this fact and the preceding claims, we infer that $\#\sL_j(P,k)\leq C\log\Lambda$, so that \rf{eqlj83} holds.
\end{proof}

\vv
For $R\in\sL(\GDF)$, $Q\in\Trc_k(R)$, we write $P\sim \TT(e'(Q))$ if there exists some
$P'\in\TT(e'(Q))$ such that 
\begin{equation}\label{defsim0}
A_0^{-2}\ell(P)\leq \ell(P')\leq A_0^2\,\ell(P)\quad \text{ and }\quad 20P'\cap20P\neq\varnothing.
\end{equation}
We say that $\TT(e'(Q))$ is a typical tractable tree, and we write $Q\in \Ty$ if
\begin{equation}\label{defty}
\sum_{P\in\DB:P\sim\TT(e'(Q))} \EE_\infty(9P)\leq \Lambda^{\frac{-1}{3n}}\,\sigma(\HD_1(e(Q))).
\end{equation}

\vv

\begin{lemma}\label{lemsuper**}
We have
$$\sum_{Q\in\DB} \EE_\infty(9Q) \lesssim \Lambda^{\frac{-1}{2n}}(\log\Lambda)^2 
\sum_{R\in\sL(\GDF)}\,
\sum_{k\geq0} B^{-k/2} \sum_{Q\in\Trc_k(R)\cap\Ty}\sigma(\HD_1(e(Q))).$$
\end{lemma}

\begin{proof}
By Lemmas \ref{lemdbnodb} and \ref{lemred*}, we have
$$\sum_{Q\in \DB} \EE_\infty(9Q) \lesssim \Lambda^{-\frac1{2n}} \sum_{R\in \GDF} \sigma(\HD_1(R)).$$
Also, by \rf{eqover5} and Lemma~\ref{eqtec74},
\begin{align*}
\sum_{R\in \GDF}\sigma(\HD_1(R)) & \leq m_0\,B^{1/4}\sum_{R\in \sL(\GDF)}\sigma(\HD_1(e(R)))\\
& \leq m_0\,B^{1/4}\sum_{R\in \sL(\GDF)} \sum_{k\geq0} B^{-k/2}\sum_{Q\in\Trc_k(R)}\sigma(\HD_1(e(Q))).
\end{align*}
Therefore,
\begin{align}\label{eqqp45}
\sum_{Q\in \DB} \EE_\infty(9Q) & \lesssim m_0\,B^{1/4}\,\Lambda^{-\frac1{2n}} 
\sum_{R\in \sL(\GDF)} \sum_{k\geq0} B^{-k/2}\sum_{Q\in\Trc_k(R)}\sigma(\HD_1(e(Q)))\\
& = m_0\,B^{1/4}\,\Lambda^{-\frac1{2n}} 
\sum_{R\in \sL(\GDF)} \sum_{k\geq0} B^{-k/2}\sum_{Q\in\Ty}\sigma(\HD_1(e(Q)))\nonumber
\\
&\quad + m_0\,B^{1/4}\,\Lambda^{-\frac1{2n}} 
\sum_{R\in \sL(\GDF)} \sum_{k\geq0} B^{-k/2}\sum_{Q\in\Trc_k(R)\setminus \Ty}\sigma(\HD_1(e(Q)))\nonumber\\
&=: S_1 + S_2.\nonumber
\end{align}
By definition, for $Q\in \Trc_k(R)\setminus \Ty$, we have
$$\sigma(\HD_1(e(Q)))\leq \Lambda^{\frac{1}{3n}}\sum_{P\in\DB:P\sim\TT(e'(R))} \EE_\infty(9P).$$
Hence, the term $S_2$ in \rf{eqqp45} does not exceed
\begin{multline*}
m_0\,B^{1/4} \,\Lambda^{-\frac1{2n}} \Lambda^{\frac{1}{3n}}
\sum_{R\in \sL(\GDF)} \sum_{k\geq0} B^{-k/2}\!\!\!\sum_{Q\in\Trc_k(R)\setminus \Ty}\,
\sum_{P\in\DB:P\sim\TT(e'(Q))} \EE_\infty(9P)\\
 \lesssim
m_0\,B^{1/4} \,\Lambda^{-\frac1{6n}} 
\sum_{P\in\DB} \EE_\infty(9P)
 \sum_{k\geq0} B^{-k/2}\,
\# A(P,k),
\end{multline*}
where
$$A(P,k)= 
\big\{R\in \sL(\GDF):P\sim\TT(e'(Q))\text{ for some } Q\in\Trc_k(R)\big\}.$$
From the definition  \rf{defsim0} and Lemma \ref{lemimp9}, it follows that
\begin{align}\label{eqremaa95}
\#A(P,k) &\leq \sum_{\substack{
P'\in\DD_\mu: 20P'\cap20P\neq\varnothing\\A_0^{-2}\ell(P)\leq \ell(P')\leq A_0^2\ell(P)
}} \!\!\#
\big\{R\in \sL(\GDF):\exists \,Q\in\Trc_k(R) \text{ such that }P'\in\TT(e'(Q))\big\}\\
&\lesssim \sum_{\substack{
P'\in\DD_\mu: 20P'\cap20P\neq\varnothing\\A_0^{-2}\ell(P)\leq \ell(P')\leq A_0^2\ell(P)
}}\!\!\!(\log\Lambda)^2\lesssim (\log\Lambda)^2.\nonumber
\end{align}
Therefore, the term $S_2$ in \rf{eqqp45} satisfies
\begin{align*}
S_2& \lesssim m_0\,B^{1/4} \,\Lambda^{-\frac1{6n}} \,(\log \Lambda)^2
\sum_{P\in\DB} \EE_\infty(9P) \sum_{k\geq0} B^{-k/2} \lesssim m_0\,B^{1/4} \,\Lambda^{-\frac1{6n}} \,(\log \Lambda)^2
\sum_{P\in\DB} \EE_\infty(9P).
\end{align*}
Assuming that $B\leq \Lambda^{\frac1{2n}}$, say, we deduce that\footnote{Here we are assuming that
$\sum_{P\in\DB} \EE_\infty(9P) <\infty$. To ensure that this holds, if necessary we may replace the measure $\mu$ by
another approximating measure of the form $\mu_\ell =\vphi_\ell * \mu$, where $\vphi_\ell$ is a $C^\infty$ bump function supported on $B(0,\ell)$, with $\|\vphi_\ell\|_1=1$. Then we prove Proposition \ref{propomain} for $\mu_\ell$ with bounds independent of $\ell$, and finally we let $\ell\to\infty$.
}
$$S_2\leq \frac12 \sum_{P\in\DB} \EE_\infty(9P)$$
if $\Lambda$ is big enough.
So we get
$$\sum_{Q\in\DB} \EE_\infty(9Q)\leq 2S_1,$$
which proves the lemma.
\end{proof}

\vv


\section{Lower estimates for the Riesz transform of the approximating measure $\eta$ in a typical tractable tree}\label{sec8}

In this section, for a given $R\in\MDW$ such that $\TT(e'(R))$ is tractable, we will define a suitable measure $\eta$ that approximates $\mu$ at the level of the cubes from
$\TT(e'(R))$ and we will get a lower bound for $\|\RR\eta\|_{L^p(\eta)}$.
In the next section we will transfer these
estimates to $\RR\mu$.

\vv



Let $R$ be a fixed cube from $\MDW\cap\Trc$.
Recall that the family $\TT(e'(R))$ was constructed by stopping time conditions involving the families $\LD(\,\cdot\,)$, $\HD(\,\cdot\,)$, and $\NDB(\,\cdot\,)$.
Next we need to define some regularized family of stopping cubes. First we consider the function 
$$d_R(x) = \inf_{Q\in\TT(e'(R))}\big(\dist(x,Q) + \ell(Q)\big).$$
Given $0<\ell_0\ll\ell(R)$, we denote
\begin{equation}\label{eql00*23}
d_{R,\ell_0}(x) = \max\big(\ell_0,d_R(x)\big).
\end{equation}
Notice that both $d_R$ and $d_{R,\ell_0}$ are $1$-Lipschitz. 

For each $x\in e'(R)$ we take the largest cube $Q_x\in\DD_\mu$ 
such that $x\in Q_x$ and
\begin{equation}\label{eqdefqx}
\ell(Q_x) \leq \frac1{60}\,\inf_{y\in Q_x} d_{R,\ell_0}(y).
\end{equation}
We consider the collection of the different cubes $Q_x$, $x\in e'(R)$, and we denote it by $\Reg(e'(R))$ (this stands for ``regularized cubes''). Observe that the cubes from $\Reg$ are disjoint by construction, and they cover $e'(R)$.

The constant $\ell_0$ is just an auxiliary parameter that prevents $\ell(Q_x)$ from vanishing.  Eventually $\ell_0$ will be taken extremely small. In particular, we assume $\ell_0$ small enough
so that 
\begin{equation}\label{eql00}
\mu\bigg(\bigcup_{Q\in\HD_1(e(R)):\ell(Q)\geq \ell_0} Q \bigg)\geq \frac12\,\mu\bigg(\bigcup_{Q\in\HD_1(e(R))} Q\bigg).
\end{equation}

We let $\TT_\Reg(e'(R))$ be the family of cubes made up of $R$ and all the cubes of the next generations which are contained in $e'(R)$ but are not 
strictly contained in any cube from $\Reg(e'(R))$.

\vv

\begin{lemma}\label{lem74}
The cubes from $\Reg(e'(R))$ are pairwise disjoint and satisfy the following properties:
\begin{itemize}
\item[(a)] If $P\in\Reg(e'(R))$ and $x\in B(x_{P},50\ell(P))$, then $10\,\ell(P)\leq d_{R,\ell_0}(x) \leq c\,\ell(P)$,
where $c$ is some constant depending only on $n$. 

\item[(b)] There exists some absolute constant $c>0$ such that if $P,\,P'\in\Reg(e'(R))$ satisfy $B(x_{P},50\ell(P))\cap B(x_{P'},50\ell(P'))
\neq\varnothing$, then
$$c^{-1}\ell(P)\leq \ell(P')\leq c\,\ell(P).$$
\item[(c)] For each $P\in \Reg(e'(R))$, there are at most $N$ cubes $P'\in\Reg(e'(R))$ such that
$$B(x_{P},50\ell(P))\cap B(x_{P'},50\ell(P'))
\neq\varnothing,$$
 where $N$ is some absolute constant.
 
\end{itemize}
\end{lemma}

The proof of this lemma is standard. See for example \cite[Lemma 6.6]{Tolsa-memo}.

\vvv

Next we introduce a measure $\eta$ which approximates $\mu|_{e'(R)}$ at the level of the cubes
from $\Reg(e'(R))$.
We let
$$\eta = \sum_{Q\in\Reg(e'(R))} \mu(Q) \frac{\LL^{n+1}|_{\tfrac12B(Q)}}{\LL^{n+1}(\tfrac12B(Q))}.
$$
Then we have:

\begin{lemma}\label{lemrieszeta}
Let $R\in\MDW\cap\Trc$ and let $\eta$ be as above.
Assume $\Lambda>0$ big enough and let $c_2\in(0,1)$ be small enough (depending at most on the parameters of the David-Mattila lattice).
Then there is a subset $V_R\subset \supp\eta$ satisfying
\begin{equation}\label{eqtec733}
\dist(Q,\supp\mu\setminus e'(R)) \gtrsim \ell(R)\quad\mbox{ for all $Q\in\Reg(e'(R))$ such that
$B_Q\cap V_R\neq\varnothing$}
\end{equation}
such that
$$\int_{V_R} \big|(|\RR\eta(x)| - \frac{c_2}2\,\Theta(\HD_1))_+\big|^p\,d\eta(x) \gtrsim \Lambda^{-p'\ve_n}\sigma_p(\HD_1(e(R))),$$
for any $p\in (1,2]$, where $p'=p/(p-1)$ and the implicit constant depends on $p$.
\end{lemma}

Recall that, for $\cI\subset\DD_\mu$,
$$\sigma_p(\cI) = \sum_{P\in \cI}\Theta(P)^p\,\mu(P).$$
In the lemma $\RR\eta(x)$ is well defined by an absolutely convergent integral for all $x\in\R^{n+1}$ because $\eta$ is absolutely continuous with respect to Lebesgue measure in $\R^{n+1}$ and it has a bounded density function, since the cubes from
the family $\Reg$ have side length bounded away from $0$, thanks to the parameter $\ell_0$.

The proof Lemma \ref {lemrieszeta} is the same as the one in \cite[Lemma 6.14]{DT} and so we omit it here. Remark that the construction of the tractable tree $\TT(e'(R))$ in that work does not involve the stopping condition $\NDB(\,\cdot\,)$. However, the reader can check that
this type of cubes do not play any role in the proof of \cite[Lemma 6.14]{DT}, and exactly the same verbatim arguments are valid here.

The proof of the previous lemma is one main key points in the \cite{DT}. The arguments to prove this are partially inspired by related results from 
\cite{ENV}, \cite{Reguera-Tolsa}, and \cite{JNRT}.

\vv


\section{Lower estimates for the Haar coefficients of $\RR\mu$ for cubes near a typical tractable tree}\label{sec9}

In all this section we assume that $R\in\MDW$ is such that $R\in\Trc\cap\Ty$, i.e., $\TT(e'(R))$ is tractable and typical, and we consider the measure $\eta$ constructed in Section \ref{sec8}.
Roughly speaking, our objective is to transfer the lower estimate we obtained for $\RR\eta$ in Lemma \ref{lemrieszeta} to 
the Haar coefficients of
$\RR\mu$ for cubes close to $\TT(e'(R))$. 
Recall that $\RR\mu(x)$ exists $\mu$-a.e.\ as a principal value under the assumptions
of Proposition \ref{propomain}.

\vv

\subsection{The operators $\RR_{\TT_\Reg}$, $\RR_{\wt \TT}$,  and $\Delta_{\wt \TT}\RR$}

To simplify notation, in this section we will write
$$\End=\End(e'(R)), \quad \;\Reg = \Reg(e'(R)),\; \quad \;\TT = \TT(e'(R)),\quad\,\text{and}\quad\; \TT_\Reg = \TT_\Reg(e'(R)).$$
We need to consider an enlarged version of the generalized tree $\TT$, due to some technical difficulties that arise because the cubes from $\Neg:= \Neg(e'(R))\cap\End$ are not $\PP$-doubling.
To this end, denote by $\Reg_\Neg$ the family of the cubes from $\Reg$ which are contained in some cube from $\Neg$.
Let $\sD_\Neg$ be the subfamily of the cubes $P\in\Reg_\Neg$ for which there exists some
$\PP$-doubling cube $S\in\TT_\Reg$ that contains $P$. By the definition of $\Neg$, such cube $S$ should be contained in the cube from $Q\in\Neg$ such that $P\subset Q$. 
We also denote by $\sM_\Neg$ the family of maximal $\PP$-doubling cubes which belong to $\TT_\Reg$ and are contained in some
cube from $\Neg$, so that, in particular, any cube from $\sD_\Neg$ is contained in another from $\sM_\Neg$.

We define
$$\wt\End = (\End \setminus \Neg) \cup \sM_\Neg,$$
and we let $\wt \TT=\wt \TT(e'(R))$ be the family of cubes that belong to $\TT_\Reg$ but are not strictly contained in any cube from $\wt\End$.
Further, we write
\begin{equation}\label{eqdef*f}
Z = Z(e'(R)) = e'(R)\setminus \bigcup_{Q\in\End} Q\quad \text{ and }\quad\wt Z = \wt Z(e'(R)) = e'(R)\setminus \bigcup_{Q\in\wt\End} Q
\end{equation}
Then we denote
$$\RR_{\TT_\Reg}\mu(x) = \sum_{Q\in\Reg} \chi_Q(x)\,\RR(\chi_{2R\setminus 2Q}\mu)(x),$$
$$\RR_{\wt \TT}\mu(x) = \sum_{Q\in\wt\End} \chi_Q(x)\,\RR(\chi_{2R\setminus 2Q}\mu)(x),$$
and
$$\Delta_{\wt\TT}\RR\mu(x) = \sum_{Q\in\wt\End} \chi_Q(x)\,\big(m_{\mu,Q}(\RR\mu) - m_{\mu,2R}(\RR\mu)\big)
+ \chi_{Z}(x) \big(\RR\mu(x) -  m_{\mu,2R}(\RR\mu)\big)
.$$

Remark that the cubes from $\Reg$ have the advantage over the cubes from $\wt \End$ that their size changes smoothly so that, for example, neighboring cubes have comparable side lengths. However, they need not be $\PP$-doubling or doubling, unlike the cubes from $\wt \End$. In particular, it is not clear
how to estimate their energy $\EE_\infty$ in terms of the condition $\DB$ (recall that the cubes from $\DB$ are asked to be $\PP$-doubling).

For $Q\in\DD_\mu$, we define
$$\QQ_\Reg(Q) = \sum_{P\in\Reg} \frac{\ell(P)}{D(P,Q)^{n+1}}\,\mu(P),$$
where 
$$D(P,Q) = \ell(P) + \dist(P,Q) + \ell(Q).$$
The coefficient $\QQ_\Reg(Q)$ will be used to bound some ``error terms" in our transference 
arguments. We will see later how they can be estimated in terms of the coefficients $\PP(Q)$. Notice that, unlike $\PP(Q)$, the coefficients $\QQ_\Reg(Q)$ depend on the family
$\Reg$.

The next three lemmas are proven in the same way as Lemmas 7.1, 7.2, and 7.3 from \cite{DT}.

\begin{lemma}\label{lemaprox1}
For any $Q\in\Reg$ such that $(Q\cup \frac12 B(Q))\cap V_R\neq\varnothing$ and $x\in Q$, $y\in \frac12B(Q)$,
$$\big|\RR_{\TT_\Reg}\mu(x) - \RR\eta(y)\big| \lesssim \Theta(R) + \PP(Q) + \QQ_\Reg(Q).$$
\end{lemma}

\vv
\begin{lemma}\label{lemaprox2}
For any $Q\in\wt\End$ and $x,y\in Q$,
$$\big|\RR_{\wt\TT}\mu(x) - \Delta_{\wt\TT}\RR\mu(y)\big| \lesssim \PP(R) + \left(\frac{\EE(4R)}{\mu(R)}\right)^{1/2} + \PP(Q) +  \left(\frac{\EE(2Q)}{\mu(Q)}\right)^{1/2}.$$
\end{lemma}

\vv

\begin{lemma}\label{lemaprox3}
Let $1<p\leq2$. For any $Q\in\wt\End$ such that $\ell_0\leq \ell(Q)$,,
$$\int_Q \big|\RR_{\wt\TT}\mu - \RR_{\TT_\Reg}\mu\big|^p\,d\mu \lesssim \EE(2Q)^{p/2} \,\mu(Q)^{1-\frac p2}.$$
\end{lemma}

\vv


\subsection{Estimates for the $\PP$ and $\QQ_\Reg$ coefficients of some cubes from $\wt \End$ and $\Reg$}

We will transfer the lower estimate obtained for the $L^p(\eta)$ norm of $\RR\eta$ in Lemma
\ref{lemrieszeta} to $\RR_{\TT}\mu$, $\RR_{\TT_\Reg}\mu$, and $\Delta_{\TT}\RR\mu$ by means of
Lemmas \ref{lemaprox1}, \ref{lemaprox2}, and \ref{lemaprox3}. To this end, we will need careful 
estimates for the $\PP$ and $\QQ_\Reg$ coefficients of cubes from $\wt\End$ and  $\Reg$. This is the
task we will perform in this section.
\vv


Given $R\in\MDW$, recall that
$\HD_1(e'(R))=\HD(R)\cap\sss(e'(R))$. 
To shorten notation, we will write $\HD_1=\HD_1(e'(R))$ in this section.
Notice also that, by \rf{eqstop2}, we have
\begin{equation}\label{eqsplit71}
\wt\End =  \LD_1 \cup \NDB_1 \cup \LD_2  \cup \NDB_2 \cup \HD_2\cup \sM_\Neg,
\end{equation}
where we introduced the following notations:
\begin{itemize}
\item $\LD_1$ is the subfamily of $\wt\End$ of those maximal $\PP$-doubling cubes which are contained both in $e'(R)$ and in some cube from $\LD(R)\cap\sss_1(e'(R))\setminus \Neg$.
\item $\NDB_1$ is the  subfamily of $\wt\End$ of those maximal $\PP$-doubling cubes which are contained both in $e'(R)$ and in some cube from $\NDB(R)\cap\sss_1(e'(R))\setminus \Neg$.
\item $\LD_2$ is the subfamily of $\wt\End$ of those maximal $\PP$-doubling cubes which are contained in some cube
$Q\in \LD(Q')\cap \sss(Q')\setminus \Neg$ for some $Q'\in \HD_1$.
\item $\NDB_2$ is the subfamily of $\wt\End$ of those maximal $\PP$-doubling cubes which are contained in some cube
$Q\in \NDB(Q')\cap\sss(Q')$ for some $Q'\in \HD_1$.
\item $\HD_2= \bigcup_{Q'\in\HD_1}(\HD(Q')\cap\sss(Q'))$.
\end{itemize}
Remark that the splitting in \rf{eqsplit71} is disjoint. Indeed, notice that, by the definition
of $\sM_\Neg$, the cubes from $\NDB_2\cup\HD_2$ do not belong to $\sM_\Neg$, since they are strictly contained in some cube from $\HD_1$, which is $\PP$-doubling, in particular.

For $i=1,2$, we also denote by $\Reg_{\LD_i}$ the subfamily of the cubes from $\Reg$ which are contained in some cube from $\LD_i$, and we define $\Reg_{\NDB_i}$, $\Reg_{\HD_2}$, $\Reg_\Neg$, and
$\Reg_{\sM_\Neg}$ analogously.\footnote{Notice that $\sD_\Neg=\Reg_{\sM_\Neg}$.}
We let $\Reg_\Ot$ be the ``other'' cubes from $\Reg$: the ones which are not contained in any cube from $\End$ (which, in particular, have side length comparable to $\ell_0$).
Also, we let $\Reg_{\DB}$ be the subfamily of the cubes from $\Reg$ which are contained in some cube from $\wt\End\cap\DB$. Notice that we have the splitting
\begin{equation}\label{eqsplitreg0}
\Reg = \Reg_{\LD_1} \cup \Reg_{\NDB_1}\cup \Reg_{\LD_2}\cup \Reg_{\NDB_2}\cup \Reg_{\HD_2}\cup \Reg_{\Neg}\cup \Reg_{\Ot}.
\end{equation}
The families above may intersect the family $\Reg_{\DB}$.

Given a family $\cI\in\DD_\mu$ and $1<p\leq2$, we denote
$$\Sigma_p^\PP(\cI) = \sum_{Q\in \cI} \PP(Q)^p\,\mu(Q),\qquad \Sigma_p^\QQ(\cI) = \sum_{Q\in \cI} \QQ_\Reg(Q)^p\,\mu(Q).$$
We also write $\Sigma^\PP(\cI)= \Sigma_2^\PP(\cI)$, $\Sigma^\QQ(\cI)= \Sigma_2^\QQ(\cI)$.

\vv

\begin{lemma}\label{lemenereg}
For any $Q\in\wt\End$, 
$$\Sigma^\PP(\Reg\cap\DD_\mu(Q)) \lesssim \PP(Q)^2\,\mu(Q) + \EE(2Q).$$
\end{lemma}

This is proven in the same way as Lemma 7.4 from \cite{DT}.



\vv
\begin{lemma}\label{lempoiss} 
We have:
\begin{itemize}
\item[(i)] If $Q\in\LD_1$, then $\PP(Q)\lesssim \delta_0^{\frac1{n+1}}\,\Lambda^{\frac n{n+1}}\,\Theta(R)$.
\item[(ii)] If $Q\in\LD_2$, then $\PP(Q)\lesssim \delta_0^{\frac1{n+1}}\,\Lambda^{1+\frac n{n+1}}\,\Theta(R)$.
\vspace{1mm}

\item[(iii)] If $Q\in\NDB_1$, then $\PP(Q)\lesssim \Lambda\,\Theta(R)$.
\vspace{1mm}

\item[(iv)] If $Q\in\NDB_2\cup \HD_2$, then $\PP(Q)\lesssim \Lambda^2\,\Theta(R)$.

\vspace{1mm}

\item[(v)] If $Q\in\Neg\cup\sM_\Neg$, then $\PP(Q)\lesssim \left(\frac{\ell(Q)}{\ell(R)}\right)^{1/3}\,\Theta(R)$.
\end{itemize}
\end{lemma}

\begin{proof}
The statement (i) is proven as in Lemma 3.2 (c) in \cite{Reguera-Tolsa}.
The statement (ii) follows in the same way as (i), replacing $\Theta(R)$ by $\Lambda\Theta(R)$.

The statement (iii) is due to the fact that, by the stopping conditions, the cubes  $Q'\in\sss_1(e'(R))$ satisfy $\PP(Q')\lesssim \Lambda\,\PP(R)\approx \Lambda\,\Theta(R)$. Then 
it just remains to notice that if
$Q$ is a maximal doubling cube contained in $Q'$, from Lemma \ref{lemdobpp} it follows that
$\PP(Q)\lesssim\PP(Q')$.

The property (iv) for the cubes in $\NDB_2$ follows by arguments analogous to the ones for (iii), taking into account that such cubes are maximal doubling cubes contained in cubes 
from $\sss(R')$, for some $R'\in\HD_1$ with $\PP(R')\approx \Theta(R')=\Lambda\,\Theta(R)$.\footnote{This arugment shows that, in fact, $\PP(Q)\lesssim\Lambda^2\Theta(R)$ for all $Q\in\TT(e'(R))$.}
On the other hand the cubes $Q\in\HD_2$ satisfy $\Theta(Q)=\Lambda^2\,\Theta(R)$ and are $\PP$-doubling by construction.

Finally, we turn our attention to (v). By the definitions of $\Neg$ and $\sM_\Neg$ and Lemma \ref{lemdobpp}, for all
$S\in\DD_\mu$ such that $Q\subset S\subset R$, we have
$$\Theta(S)\lesssim \left(\frac{\ell(S)}{\ell(R)}\right)^{1/2}\,\PP(R).$$
Thus,
\begin{align*}
\PP(Q) &\approx \frac{\ell(Q)}{\ell(R)}\,\PP(R) + \sum_{S:Q\subset S\subset R}\frac{\ell(Q)}{\ell(S)}\,\Theta(S)\\
& \lesssim \frac{\ell(Q)}{\ell(R)}\,\Theta(R) + \sum_{S:Q\subset S\subset R}\frac{\ell(Q)}{\ell(S)}
\,\left(\frac{\ell(S)}{\ell(R)}\right)^{1/2}\,\Theta(R)\\
& \lesssim \frac{\ell(Q)}{\ell(R)}\,\Theta(R) + \sum_{S:Q\subset S\subset R}\,\left(\frac{\ell(Q)}{\ell(R)}\right)^{1/2}\,\Theta(R)\\
& \lesssim \frac{\ell(Q)}{\ell(R)}\,\Theta(R) + \log\left(\frac{\ell(R)}{\ell(Q)}\right)\,\left(\frac{\ell(Q)}{\ell(R)}\right)^{1/2}\,\Theta(R) \lesssim \left(\frac{\ell(Q)}{\ell(R)}\right)^{1/3}\,\Theta(R).
\end{align*}

\end{proof}
\vv

\begin{rem}\label{rem9.7}
We will assume that 
$$\delta_0\leq \Lambda^{-(n+2)^2}.$$
With this choice, it follows easily that $\delta_0^{\frac1{n+1}}\,\Lambda^{1+\frac n{n+1}}
\leq \delta_0^{\frac1{n+2}}$, so that by the preceding lemma
 $$\PP(Q)\lesssim \delta_0^{\frac1{n+2}}\,\Theta(R)\quad \mbox{ for all $Q\in\LD_1\cup\LD_2$.}$$
\end{rem}

\vv
\begin{lemma}\label{lemregmolt}
Suppose that $R\in\Trc\cap\Ty$. Then

\begin{equation}\label{eqlem*1}
\Sigma^\PP(\wt\End)\approx \sigma(\wt\End) \lesssim B\,\sigma(\HD_1),
\end{equation}

\begin{equation}\label{eqlem*2}
\Sigma^\PP(\Reg_\DB) \lesssim \sum_{Q\in\wt \End\cap \DB}\EE_\infty(9Q)\leq
\sum_{Q\in\DB:Q\sim \TT}\EE_\infty(9Q)
\lesssim \Lambda^{\frac{-1}{3n}}\,\sigma(\HD_1),
\end{equation}

\begin{equation}\label{eqlem*3}
\Sigma^\PP(\wt\End\cap \DB)\approx\sigma(\wt\End\cap \DB)\leq M_0^{-2}\,\Lambda^{\frac{-1}{3n}}\,\sigma(\HD_1),
\end{equation}

\begin{equation}\label{eqlem*4}
\Sigma^\PP(\Reg_{\NDB_1}) + \Sigma^\PP(\Reg_{\NDB_2})\lesssim 
\sum_{Q\in\NDB_1\cup\NDB_2}\EE_\infty(9Q)
\lesssim \Lambda^{\frac{-1}{3n}}\,\sigma(\HD_1),
\end{equation}

\begin{equation}\label{eqlem*4.5}
\Sigma^\PP(\Reg_{\LD_1}\cup \Reg_{\LD_2}) \lesssim \sum_{Q\in
\LD_1\cup \LD_2} \EE_\infty(9Q)\lesssim 
\big(B\,M_0^2\,\delta_0^{\frac2{n+2}} + \Lambda^{\frac{-1}{3n}}\big)\,\sigma(\HD_1).
\end{equation}

\noi Also, for $1<p\leq2$,

\begin{equation}\label{eqlem*5}
\Sigma_p^\PP(\Reg_{\LD_1}\setminus\Reg_{\DB}) + \Sigma_p^\PP(\Reg_{\LD_2})\lesssim
\Big(B\,\Lambda^2\,\delta_0^{\frac p{n+2}} + \Lambda^{\frac{-p}{6n}}\Big)\,
\sigma_p(\HD_1),
\end{equation}

\begin{equation}\label{eqlem*6}
\Sigma_p^\PP(\Reg_{\HD_2})\lesssim \Sigma_p^\PP(\HD_2) \approx \sigma_p(\HD_2) \lesssim B\,\Lambda^{p-2}\,\sigma_p(\HD_1),
\end{equation}

\begin{equation}\label{eqlem*7}
\Sigma_p^\PP(\Reg_{\NDB_2})\lesssim 
\Lambda^{\frac{-p}{6n}}\,\sigma_p(\HD_1).
\end{equation}
\end{lemma}
\vv

\begin{proof}
To prove \rf{eqlem*1}, notice first that $\Sigma^\PP(\wt\End)\approx \sigma(\wt\End)$ because
the cubes from $\wt\End$ are $\PP$-doubling. Also, by \rf{eqsplit71} we have
$$\Sigma^\PP(\wt\End) =  \Sigma^\PP(\LD_1) + \Sigma^\PP(\LD_2) +\Sigma^\PP(\NDB_1)+ \Sigma^\PP(\NDB_2)
+ \Sigma^\PP(\HD_2) + \Sigma^\PP(\sM_\Neg).$$
By Lemma \ref{lempoiss} and the subsequent remark, we have
\begin{equation}\label{eqld12}
\Sigma^\PP(\LD_1)+ \Sigma^\PP(\LD_2) \lesssim \delta_0^{\frac2{n+2}}\,\Theta(R)^2\,\mu(R) = \delta_0^{\frac2{n+2}}\,\sigma(R).
\end{equation}
Again by Lemma \ref{lempoiss}, we have $\PP(Q)\lesssim\Theta(R)$ for all $Q\in\sM_\Neg$ and thus\footnote{We will obtain better estimates for $\Sigma^\PP(\sM_\Neg)$ below.}
$$\Sigma^\PP(\sM_\Neg)\lesssim \sigma(R).$$
Also, since $\TT$ is a tractable tree, we have
$$\sigma(R)\leq B\,\sigma(\HD_1)\quad \text{ and }\quad
\Sigma^\PP(\HD_2)\approx \sigma(\HD_2) \lesssim B\,\sigma(\HD_1).$$

Regarding the family $\NDB_1\cup\NDB_2$, denote 
$$\sss_\NDB = \sss_1(e'(R))\cap\NDB(R) \cup \bigcup_{S\in\HD_1} \big(\sss(S)\cap\NDB(S)\big),$$
so that the cubes from $P\in\NDB_1\cup\NDB_2$ are maximal $\PP$-doubling cubes contained in some cube from
$\sss_\NDB$. By Lemma \ref{lemdobpp}, if $P\in\NDB_1\cup\NDB_2$ is contained in $Q\in\sss_\NDB$,
then $\PP(P)\lesssim \PP(Q)$. Also, by the definition of $\NDB(\,\cdot\,)$, there exists
$Q'=Q'(Q)\in\DB$ such that $\ell(Q')=\ell(Q)$ and $Q'\subset 20Q$. 
In particular, this implies that $Q'\sim\TT(e'(R))$ and that $\PP(Q)\approx\PP(Q')\approx\sigma(Q')$, and by the definition of $\DB$,
\begin{equation}\label{eqDB99}
\EE_\infty(9Q')\gtrsim M_0^2\,\sigma(Q').
\end{equation}
Thus,
\begin{align}\label{eqndb945}
\Sigma^\PP(\NDB_1\cup\NDB_2) & = \sum_{Q\in\sss_\NDB} \sum_{P\in\End\cap\DD_\mu(Q)} \PP(P)^2\,\mu(P)\\
& \lesssim \sum_{Q\in\sss_\NDB}\PP(Q)^2\,\mu(Q) \approx \sum_{Q\in\sss_\NDB}\sigma(Q'(Q))\nonumber\\
& \lesssim \frac1{M_0^2}\sum_{Q\in\sss_\NDB}\EE_\infty(9Q'(Q))\lesssim \frac1{M_0^2}\sum_{Q'\in\DB:Q'\sim\TT} \EE_\infty(9Q'),\nonumber
\end{align}
where in the last estimate we took into account that for each $Q'$ there is a bounded number of cubes $Q\in\sss_\NDB$ such that $Q'=Q'(Q)$ (possibly depending on $n$ and $A_0$).
Since $\TT$ is a typical tree, by the definition in \rf{defty}, we have
\begin{equation}\label{eqDB09}
\sum_{Q'\in\DB:Q'\sim\TT} \EE_\infty(9Q')\leq \Lambda^{\frac{-1}{3n}}\,\sigma(\HD_1).
\end{equation}
So we have
\begin{equation}\label{eqalighs33}
\Sigma^\PP(\NDB_1\cup\NDB_2)\lesssim M_0^{-2}\Lambda^{\frac{-1}{3n}}\,\sigma(\HD_1).
\end{equation}
Gathering the estimates above, the estimate \rf{eqlem*1} follows.

\vv

To prove \rf{eqlem*2}, we apply Lemma \ref{lemenereg} and the fact that $\TT$ is a typical tree again:
\begin{align*}
\Sigma^\PP(\Reg_\DB) & =\sum_{Q\in\wt\End\cap\DB} \Sigma^\PP(\Reg_\DB\cap \DD_\mu(Q)) \\
& \lesssim
\sum_{Q\in\wt\End\cap\DB} \big(\PP(Q)^2\,\mu(Q) + \EE(9Q)\big)\lesssim \sum_{Q\in\wt\End\cap\DB} \EE_\infty(9Q)
\end{align*}
Notice that, by the construction of the family
$\Reg$, for each $Q\in\sM_\Neg$ there exists some cube $Q'\in\TT$ such that
\begin{equation}\label{eqnegsim*}
\ell(Q)\leq \ell(Q')\leq A_0^2\,\ell(S)\quad \text{ and }\quad 20Q'\cap20Q\neq\varnothing.
\end{equation}
Hence, $Q\sim\TT$ (see \rf{defsim0}), and so
\begin{align*}
\sum_{Q\in\wt\End\cap\DB}  \EE_\infty(9Q) & = \sum_{Q\in\End\cap\DB}  \EE(9Q)_\infty
+ \sum_{Q\in\sM_\Neg\cap\DB} \EE_\infty(9Q)\\
& \lesssim \sum_{S\in\DB:S\sim\TT} \EE_\infty(9S)
\leq \Lambda^{\frac{-1}{3n}}\,\sigma(\HD_1),
\end{align*}
taking into account that $\TT$ is a typical tree, which completes the proof of \rf{eqlem*2}.
Notice also that \rf{eqlem*3} follows from \rf{eqlem*2} and the fact that the cubes from $\DB$ satisfy
\rf{eqDB99}.

\vv
Next we turn our attention to \rf{eqlem*4}. By Lemma \ref{lemenereg} we have
\begin{align*}
\Sigma^\PP(\Reg_{\NDB_1}\cup\Reg_{\NDB_2})& \lesssim \sum_{Q\in \NDB_1\cup\NDB_2} \EE_\infty(9Q)\\
&\leq \sum_{Q\in \wt\End\cap\DB} \EE_\infty(9Q) + \sum_{Q\in (\NDB_1\cup\NDB_2)\setminus \DB} \EE_\infty(9Q).
\end{align*}
By \rf{eqlem*2}, the first term on the right hand side above does not exceed $C\Lambda^{\frac{-1}{3n}}\,\sigma(\HD_1)$. Concerning the second term, we use the fact for the cubes $Q\not\in\DB$,
we have $\EE_\infty(9Q)\lesssim M_0^2\,\sigma(Q)$ together with \rf{eqalighs33}. Then we get
$$\sum_{Q\in (\NDB_1\cup\NDB_2)\setminus \DB} \EE_\infty(9Q)
\lesssim M_0^2 \,\sigma(\NDB_1\cup\NDB_2)\lesssim \Lambda^{\frac{-1}{3n}}\,\sigma(\HD_1).$$

\vv

Regarding the estimate \rf{eqlem*4.5}, by \rf{eqld12} and Lemma
\ref{lemenereg}, we obtain
\begin{align*}
\Sigma^\PP((\Reg_{\LD_1}\cup \Reg_{\LD_2})\setminus\Reg_{\DB}) & \lesssim \sum_{Q\in
(\LD_1\cup \LD_2)\setminus\DB} \EE_\infty(9Q)\lesssim M_0^2\,\Sigma^\PP(\LD_1\cup \LD_2) \\
&\lesssim M_0^2\,\delta_0^{\frac2{n+2}}\sigma(R)\leq 
B\,M_0^2\,\delta_0^{\frac2{n+2}}\sigma(\HD_1).
\end{align*}
Together with \rf{eqlem*2}, this yields \rf{eqlem*4.5}.

To get \rf{eqlem*5}, we apply H\"older's inequality in the preceding estimate:
\begin{align*}
\Sigma_p^\PP((\Reg_{\LD_1}\cup \Reg_{\LD_2})\setminus\Reg_{\DB}) & \leq 
(\Sigma^\PP((\Reg_{\LD_1}\cup \Reg_{\LD_2})\setminus\Reg_{\DB}))^{\frac p2}\,\mu(e'(R))^{1-\frac p2}\\
& \lesssim (B\,M_0^2\,\delta_0^{\frac2{n+2}}\sigma(\HD_1))^{\frac p2}\,\mu(R)^{1-\frac p2}.
\end{align*}
Observe now that, writing $HD_i=\bigcup_{Q\in\HD_i}Q$,
\begin{equation}\label{eqmuhd1}
\mu(HD_1) = \frac1{\Lambda^2\,\Theta(R)^2}\,\sigma(\HD_1) \geq \frac1{B\,\Lambda^2\,\Theta(R)^2}\,\sigma(R) = \frac1{B\,\Lambda^2}\,\mu(R).
\end{equation}
Thus, using also that $M_0\leq \Lambda$,
\begin{align}\label{eqld7345}
\Sigma_p^\PP((\Reg_{\LD_1}\cup \Reg_{\LD_2})\setminus\Reg_{\DB}) &
\lesssim (B\,M_0^2\,\delta_0^{\frac2{n+2}}\sigma(\HD_1))^{\frac p2}\,(B\,\Lambda^2\,\mu(HD_1))^{1-\frac p2} \\
& \leq B\,\Lambda^2\,\delta_0^{\frac p{n+2}}\sigma_p(\HD_1).\nonumber
\end{align}
On the other hand, if we let $\Reg_{\DB_2}$ be the subfamily of the cubes from $\Reg_{\DB}$ which
are contained in some cube from $\HD_1$, by H\"older's inequality, Lemma
\ref{lemenereg}, and \rf{eqDB09}, we get
\begin{align*}
\Sigma_p^\PP(\Reg_{\DB_2}) &\lesssim \Sigma^\PP(\Reg_{\DB_2})^{\frac p2}\,\mu(HD_1)^{1-\frac p2}\\
&\lesssim \bigg(\sum_{Q'\in\DB:Q'\sim\TT} \EE_\infty(9Q')\bigg)^{\frac p2}\,\mu(HD_1)^{1-\frac p2}\\
& \lesssim \big(\Lambda^{\frac{-1}{3n}}\,\sigma(\HD_1)\big)^{\frac p2}\,\mu(HD_1)^{1-\frac p2}= \Lambda^{\frac{-p}{6n}}\,\sigma_p(\HD_1).
\end{align*}
Gathering \rf{eqld7345} and the last estimate, we get \rf{eqlem*5}.

To prove \rf{eqlem*6}, recall that $\Theta(Q) \leq\Lambda^2\,\Theta(R)$ for all $Q\in\TT$.
This implies that also $\PP(Q) \lesssim\Lambda^2\,\Theta(R)$ for all $Q\in\Reg$, by Lemma \ref{lemdobpp}. Consequently,
$$\Sigma_p^\PP(\Reg_{\HD_2})\lesssim \Lambda^{2p}\,\Theta(R)^p\,\mu(HD_2) = \sigma_p(\HD_2) \approx
\Sigma_p^\PP(\HD_2).$$
On the other hand, since $R\in\Trc$,
\begin{align*}
\sigma_p(\HD_2) & = \Theta(\HD_2)^{p-2}\,\sigma(\HD_2)\leq B\,\Theta(\HD_2)^{p-2}\,\sigma(\HD_1)
\\ &= B\,\frac{\Theta(\HD_2)^{p-2}}{\Theta(\HD_1)^{p-2}}\,\sigma_p(\HD_1) = B\,\Lambda^{p-2}\,\sigma_p(\HD_1),
 \end{align*}
 which completes the proof of \rf{eqlem*6}.
 
Finally, observe that \rf{eqlem*7} follows \rf{eqlem*4} using H\"older's inequality and the fact that
the cubes from $\Reg_{\NDB_2}$ are contained in cubes from $\HD_1$:
$$\Sigma_p^\PP(\Reg_{\NDB_2}) \leq \Sigma^\PP(\Reg_{\NDB_2})^{\frac p2}\mu(HD_1)^{1-\frac p2} \lesssim \big(\Lambda^{\frac{-1}{3n}}\sigma(\HD_1)\big)^{\frac p2}\mu(HD_1)^{1-\frac p2}
=  \Lambda^{\frac{-p}{6n}}\,\sigma_p(\HD_1).$$
\end{proof}

\vv
For the record, notice that \rf{eqndb945} also shows that the family $\sss_\NDB$ defined just above
\rf{eqDB99} satisfies
$$\Sigma^\PP(\sss_\NDB)= \sum_{Q\in\sss_\NDB}\PP(Q)^2\,\mu(Q) \lesssim \frac1{M_0^2}\sum_{Q'\in\DB:Q'\sim\TT} \EE_\infty(9Q'),$$
which together with \rf{eqDB09} yields
\begin{equation}\label{eqsssndb6}
\Sigma^\PP(\sss_\NDB)\lesssim M_0^{-2}\Lambda^{\frac{-1}{3n}}\,\sigma(\HD_1).
\end{equation}

\vv

\begin{lemma}\label{lemneg3}
Suppose that $R\in\Trc\cap\Ty$. Then
\begin{equation}\label{eqlemneg01}
\Sigma^\PP(\sM_\Neg)\lesssim\Sigma^\PP(\Neg) \lesssim \big(\delta_0\,B\,\Lambda^5 + B\,\Lambda^{-3}
+M_0^{-2}\Lambda^{\frac{-1}{3n}}\big)\,\sigma(\HD_1),
\end{equation}
and
\begin{equation}\label{eqlemneg02}
\Sigma^\PP(\Reg_\Neg) \lesssim \Sigma^\PP(\Neg) + \sum_{Q\in\sM_\Neg} \EE_\infty(9Q)\lesssim \big(\delta_0\,B\,\Lambda^7 + B\,\Lambda^{-1}
+\Lambda^{\frac{-1}{3n}}\big)\,\sigma(\HD_1).
\end{equation}
\end{lemma}

\begin{proof}
By Lemma \ref{lemnegs}, the cubes from $\Neg$ belong either to $\LD(R)$ or to $\NDB(R)$. Thus, 
if we denote
$$\Neg_\LD = \Neg \cap\LD(R)\cap\sss(e'(R)),\quad\;\Neg_\NDB = \Neg \cap\NDB(R)\cap\sss(e'(R)),$$
it is clear that
$$\Neg = \Neg_\LD \cup \Neg_\NDB.$$
We can split $\Reg_\Neg$ in an analogous way:
$$\Reg_\Neg = \Reg_{\Neg_\LD} \cup \Reg_{\Neg_\NDB}.$$
Recall that the cubes from $\Neg$ are not $\PP$-doubling and that $\PP(Q)\lesssim \left(\frac{\ell(Q)}{\ell(R)}\right)^{1/3}\,\Theta(R)$ for all $Q\in\Neg$.

To estimate $\Sigma^\PP(\Neg_\LD)$, we split $\Neg_\LD$ into two subfamilies $\cI$ and $\cJ$ so that the 
cubes from $\cI$ have side length at least $\Lambda^{-5}\ell(R)$, opposite to the ones from $\cJ$. We have
\begin{align*}
\Sigma^\PP(\cI) &= \sum_{Q\in \cI}\PP(Q)^2\,\mu(Q) \lesssim \Theta(R)^2\sum_{Q\in \cI}\mu(Q)\\
& \lesssim \Theta(R)^2\sum_{Q\in \cI}\Theta(Q)\,\ell(Q)^n \lesssim \delta_0\Theta(R)^3\sum_{Q\in \cI}\ell(Q)^n.
\end{align*}
Using now that the balls $\frac12B(Q)$, with $Q\in \cI$, are disjoint and that $\ell(Q)\geq \Lambda^{-5}\ell(R)$, we get
$$\sum_{Q\in \cI}\ell(Q)^n \leq \sum_{Q\in \cI}\frac{\Lambda^5\,\ell(Q)^{n+1} }{\ell(R)}\lesssim
\frac{\Lambda^5\,\ell(R)^{n+1} }{\ell(R)}=\Lambda^5\,\ell(R)^n.$$
Therefore,
$$\Sigma^\PP(\cI)\lesssim \delta_0\,\Lambda^5\,\Theta(R)^3\,\ell(R)^n\approx  \delta_0\,\Lambda^5\,\sigma(R)\lesssim \delta_0\,B\,\Lambda^5\,\sigma(\HD_1).$$
In connection with the family $\cJ$, we have
\begin{align*}
\Sigma^\PP(\cJ) &= \sum_{Q\in \cJ}\PP(Q)^2\,\mu(Q)\lesssim \Theta(R)^2\sum_{Q\in \cJ}\left(\frac{\ell(Q)}{\ell(R)}\right)^{2/3}\,\mu(Q)\\
&\leq \Lambda^{-10/3}\,\Theta(R)^2\sum_{Q\in \cJ}\mu(Q) \leq \Lambda^{-3}\,\sigma(R)\leq B\,\Lambda^{-3}\,\sigma(\HD_1).
\end{align*}
Hence,
$$\Sigma^\PP(\Neg_\LD)\lesssim \big(\delta_0\,B\,\Lambda^5 + B\,\Lambda^{-3}\big)\,\sigma(\HD_1).$$

To estimate $\Sigma^\PP(\Neg_\NDB)$ we just take into account that $\Neg_\NDB\subset \sss_\NDB$,
 and then by \rf{eqsssndb6} we have
$$\Sigma^\PP(\Neg_\NDB)\leq
\Sigma^\PP(\sss_\NDB)\lesssim M_0^{-2}\Lambda^{\frac{-1}{3n}}\,\sigma(\HD_1).
$$
Gathering the estimates obtained for $\Sigma^\PP(\Neg_\LD)$ and $\Sigma^\PP(\Neg_\NDB)$ we get
$$\Sigma^\PP(\Neg) \lesssim \big(\delta_0\,B\,\Lambda^5 + B\,\Lambda^{-3}
+M_0^{-2}\Lambda^{\frac{-1}{3n}}\big)\,\sigma(\HD_1),$$
as wished.
The fact that $\Sigma^\PP(\sM_\Neg)\lesssim \Sigma^\PP(\Neg)$ follows from the fact that if $P\in\sM_\Neg$ satisfies $P\subset Q\in \Neg$, then $\PP(P)\lesssim \PP(Q)$ because $P$ is a maximal
$\PP$-doubling cube contained in $Q$.

\vv
Next we deal with \rf{eqlemneg02}.
 We split
$$\Sigma^\PP(\Reg_\Neg) = \Sigma^\PP(\Reg_{\sM_\Neg}) + \Sigma^\PP(\Reg_\Neg\setminus \Reg_{\sM_\Neg}).$$
Notice that, for all $P\in\Reg_\Neg\setminus \Reg_{\sM_\Neg}$ such that $P\subset Q\in\Neg$, by Lemma \ref{lemdobpp}, 
$\PP(P)\lesssim\PP(Q)$. Then it follows that
$$\Sigma^\PP(\Reg_\Neg\setminus \Reg_{\sM_\Neg})\lesssim \Sigma^\PP(\Neg).$$

To estimate $\Sigma^\PP(\Reg_{\sM_\Neg})$, we apply Lemma \ref{lemenereg}:
\begin{align}\label{eq83e}
\Sigma^\PP(\Reg_{\sM_\Neg})  & =\sum_{S\in\sM_\Neg} \Sigma^\PP(\Reg_{\sM_\Neg}\cap\DD_\mu(S))
 \lesssim \sum_{S\in\sM_\Neg} \EE_\infty(9S) \\
 & \leq  \sum_{S\in\sM_\Neg\setminus\DB} \EE_\infty(9S) + \sum_{S\in\DB\cap\wt \End} \EE_\infty(9S).\nonumber
\end{align}
By the definition of $\DB$ and the fact that $\Theta(S)\lesssim\PP(Q)$ whenever $S\subset Q\in\Neg$,
we derive
$$\sum_{S\in\sM_\Neg\setminus\DB} \EE_\infty(9S)\lesssim M_0^2\sum_{S\in\sM_\Neg\setminus\DB} \sigma(S)\lesssim
M_0^2\,\Sigma^\PP(\Neg).$$
On the other hand, by \rf{eqlem*2},
$$\sum_{S\in\DB\cap\wt \End} \EE_\infty(9S) \leq\Lambda^{\frac{-1}{3n}}\,\sigma(\HD_1).
$$
 Therefore,
$$\Sigma^\PP(\Reg_{\sM_\Neg})\lesssim M_0^2\,\Sigma^\PP(\Neg) + \Lambda^{\frac{-1}{3n}}\,\sigma(\HD_1).$$

Gathering the estimates above, and taking into account that $M_0<\Lambda$ (by \rf{eqlammm}), we deduce
\begin{align*}
\Sigma^\PP(\Reg_\Neg)&\lesssim M_0^2\,\Sigma^\PP(\Neg) + \Lambda^{\frac{-1}{3n}}\,\sigma(\HD_1)\\
& \lesssim
\big(\delta_0\,B\,M_0^2\,\Lambda^5 + B\,\Lambda^{-3}\,M_0^2
+\Lambda^{\frac{-1}{3n}}\big)\,\sigma(\HD_1) + \Lambda^{\frac{-1}{3n}}\,\sigma(\HD_1)\\
&\leq \big(\delta_0\,B\,\Lambda^7 + B\,\Lambda^{-1}
+\Lambda^{\frac{-1}{3n}}\big)\,\sigma(\HD_1).
\end{align*}
\end{proof}
\vv

Next lemma deals with $\Sigma_p^\PP(\Reg_\Ot)$.

\begin{lemma}\label{lemregot}
For $1\leq p\leq2$, we have
\begin{equation}\label{eqgaga34}
\limsup_{\ell_0\to0} \Sigma_p^\PP(\Reg_\Ot(\ell_0)) \lesssim \Lambda^{2p}\,\Theta(R)^p\,\mu(Z).
\end{equation}
Consequently, if $\mu(Z)\leq \ve_Z\,\mu(R)$, then
\begin{equation}\label{eqgaga35}
\limsup_{\ell_0\to0} \Sigma_p^\PP(\Reg_\Ot(\ell_0)) \lesssim B\,\Lambda^4\,\ve_Z\,\sigma_p(\HD_1).
\end{equation}
\end{lemma}

Above we wrote $\Reg_\Ot(\ell_0)$ to recall the dependence of the family $\Reg_\Ot$ on the parameter
$\ell_0$ in \rf{eql00*23}.

\begin{proof}
If $x\in Q\in\End$ with $\ell(Q)\geq \ell_0$, then 
$$d_{R,\ell_0}(x)\leq \max(\ell_0,\ell(Q)) = \ell(Q)$$
(recall that $d_{R,\ell_0}$ is defined in \rf{eql00*23}), and thus $x$ is contained in some cube $Q'\in\Reg$ with $\ell(Q')\leq \ell(Q)$, by the definition of the family $\Reg$. So $Q'\subset Q$ and then $Q'\in\Reg\setminus \Reg_\Ot$. Therefore,
\begin{equation}\label{eqregotin}
\bigcup_{P\in\Reg_\Ot} P \subset e'(R) \setminus \bigcup_{Q\in\End:\ell(Q)>\ell_0} Q,
\end{equation}
and so
\begin{equation}\label{eqlimot62}
\limsup_{\ell_0\to0} \mu\bigg(\bigcup_{P\in\Reg_\Ot(\ell_0)} P\bigg) \leq \mu\bigg(e'(R) \setminus \bigcup_{Q\in\End} Q\bigg) = \mu(Z).
\end{equation}
To complete the proof of \rf{eqgaga34} it just remains to notice that, using the fact that $\PP(P)\lesssim \Lambda^2\,\Theta(R)$ for all $P\in\Reg$, we have
$$\Sigma_p^\PP(\Reg_\Ot(\ell_0)) \lesssim \Lambda^{2p}\,\Theta(R)^p \,\mu\bigg(\bigcup_{P\in\Reg_\Ot(\ell_0)} P\bigg).$$

Regarding the second statement of the lemma recall that, as in \rf{eqmuhd1},
$\mu(HD_1) \geq  \frac1{B\,\Lambda^2}\,\mu(R)$, which implies that
$$\mu(Z) \leq B\,\Lambda^2\,\ve_Z\,\mu(HD_1).$$
Plugging this estimate into \rf{eqgaga34}, we get 
$$\limsup_{\ell_0\to0} \Sigma_p^\PP(\Reg_\Ot(\ell_0)) \lesssim B\,\Lambda^{2+p}\,\ve_Z\,\Theta(\HD_1)^p\,\mu(HD_1) \leq B\,\Lambda^4\,\ve_Z\,\sigma_p(\HD_1).$$
\end{proof}

\vv

\begin{rem}
By Lemma \ref{lemregmolt}, for $p\in (1,2)$, if $\Lambda\gg B\gg1$ and
$\delta_0\ll\Lambda^{-1}$, we have
$$\Sigma_p^\PP(\Reg_{\LD_1}\setminus\Reg_{\DB}) + \Sigma_p^\PP(\Reg_{\LD_2}\setminus\Reg_{\DB}) +
\Sigma_p^\PP(\Reg_{\HD_2}) +
\Sigma_p^\PP(\Reg_{\NDB_2})\ll \sigma_p(\HD_1).$$
Analogously, by Lemma \ref{lemregot}, if $\ve_Z\ll (B\Lambda^4)^{-1}$ and $\ell_0$ is small enough,
$$\Sigma_p^\PP(\Reg_{\Ot})\ll \sigma_p(\HD_1).$$
To transfer the lower estimates for the $L^p(\eta)$ norm of $\RR\eta$ in Lemma
\ref{lemrieszeta} to $\RR_{\TT}\mu$, $\RR_{\TT_\Reg}\mu$, and $\Delta_{\TT}\RR\mu$, it would be useful to have estimates for $\Sigma_p^\PP(\Reg_{\NDB_1})$, $\Sigma_p^\PP(\Reg_{\DB})$, and $\Sigma_p^\PP(\Neg)$ analogous to the ones above. However, we have not been able to get them. 
This fact originates some technical difficulties, which will be solved in the next lemmas.
\end{rem}
\vv

\begin{rem}\label{rem9.12}
By Lemma \ref{lemregmolt}, for $1<p\leq 2$, and $\ell_0$ small enough we have
\begin{multline*}
\Sigma_p^\PP(\Reg_{\LD_1}\setminus\Reg_{\DB}) + \Sigma_p^\PP(\Reg_{\LD_2}) +
\Sigma_p^\PP(\Reg_{\NDB_2}) 
\lesssim 
\Big(B\,\Lambda^2\,\delta_0^{\frac 1{n+2}} + \Lambda^{\frac{-1}{6n}}\Big)\,
\sigma_p(\HD_1).
\end{multline*}
Also, by Lemma \ref{lemregot}, for $\ell_0$ small enough,  if
$\mu(Z)\leq \ve_Z\,\mu(R)$, 
$$\Sigma_p^\PP(\Reg_\Ot(\ell_0))\lesssim 
B\,\Lambda^4\,\ve_Z\,\sigma_p(\HD_1).$$
From now on we will assume that $\ell_0$ small enough so that this holds,. Remember also that we chose
$$B=\Lambda^{\frac1{100n}}.$$
We assume also $\delta_0$ and $\ve_Z$ small enough so that
\begin{equation}\label{eqassu78}
B\,\Lambda^2\,\delta_0^{\frac 1{n+2}} + B\,\Lambda^4\,\ve_Z + B\,\Lambda^7\,\delta_0\leq \Lambda^{\frac{-1}{3n}}.
\end{equation}
In this way, we have
\begin{align}\label{eqtot999}
\Sigma_p^\PP(\Reg_{\LD_1}\setminus\Reg_{\DB}) + \Sigma_p^\PP(\Reg_{\LD_2}) +
\Sigma_p^\PP(\Reg_{\NDB_2})  + \Sigma_p^\PP(\Reg_\Ot)\lesssim 
 \Lambda^{\frac{-1}{6n}} \,
\sigma_p(\HD_1),
\end{align}
under the assumption that $\mu(Z)\leq \ve_Z\,\mu(R)$. Also, from \rf{eqlemneg02}, it follows that
\begin{equation}\label{eqassu78699}
\Sigma^\PP(\Reg_\Neg) \lesssim \big(\delta_0\,B\,\Lambda^7 + B\,\Lambda^{-1}
+\Lambda^{\frac{-1}{3n}}\big)\,\sigma(\HD_1)\lesssim \Lambda^{\frac{-1}{3n}} \,
\sigma(\HD_1).
\end{equation}
\end{rem}

\vv

\begin{lemma}\label{lemreg73}
Suppose that $R\in\Trc\cap\Ty$ and that $\mu(Z)\leq \ve_Z\,\mu(R)$.
We have
$$\Sigma^\PP(\Reg)\lesssim B\,\sigma(\HD_1).$$
For $p_0 = 2- \frac1{18n}$, we have
$$\Sigma_{p_0}^\PP(\Reg)\lesssim \Lambda^{\frac{-1}{25n}}\,\sigma_{p_0}(\HD_1).$$
\end{lemma}

\begin{proof}
By \rf{eqtot999} and \rf{eqlem*6}, for $1<p\leq 2$,
\begin{multline*}
\Sigma_p^\PP(\Reg_{\LD_1}\setminus\Reg_{\DB}) + \Sigma_p^\PP(\Reg_{\LD_2}) +
\Sigma_p^\PP(\Reg_{\NDB_2}) + \Sigma_p^\PP(\Reg_{\HD_2})+ \Sigma_p^\PP(\Reg_\Ot)\\
\lesssim 
\Big(\Lambda^{\frac{-1}{6n}} + B\,\Lambda^{p-2}\Big)\,
\sigma_p(\HD_1).
\end{multline*}
Also, by \rf{eqlem*2}, \rf{eqlem*4}, and \rf{eqassu78699},
$$\Sigma^\PP(\Reg_\DB) + \Sigma^\PP(\Reg_{\NDB_1}) + \Sigma^\PP(\Reg_{\Neg}) \lesssim 
\big(\delta_0\,B\,\Lambda^7 + B\,\Lambda^{-1}
+\Lambda^{\frac{-1}{3n}}\big)\,\sigma(\HD_1)\lesssim
\Lambda^{\frac{-1}{3n}}\,\sigma(\HD_1),
$$
by the assumptions on $\delta_0$, $\Lambda$, and $B$ in Remarks \ref{rem9.7} and \ref{rem9.12}.
Choosing $p=2$ above and adding the preceding estimates, we get the first statement in the lemma.

To get the second inequality in the lemma, we apply by H\"older's inequality and \rf{eqmuhd1} in the last estimate, and we get
\begin{align}\label{eqsam438}
\Sigma_p^\PP(\Reg_\DB) + \Sigma_p^\PP(\Reg_{\NDB_1}) + \Sigma_p^\PP(\Reg_{\Neg}) &\lesssim \big(\Lambda^{\frac{-1}{3n}}\,\sigma(\HD_1)\big)^{\frac p2}\,\mu(R)^{1-\frac p2}\\
& \lesssim
\big(\Lambda^{\frac{-1}{3n}}\,\sigma(\HD_1)\big)^{\frac p2}\,\big( B\,\Lambda^2\mu(HD_1)\big)^{1-\frac p2}\nonumber\\
& \leq \Lambda^{\frac{-1}{6n}}\,\big(\Lambda^3\big)^{1-\frac p2}\,\sigma_p(\HD_1).\nonumber
\end{align}
For $p=2-\frac1{18n}$, we have 
$$\big(\Lambda^3\big)^{1-\frac p2} = \Lambda^{\frac1{12n}},$$
and thus
\begin{equation}\label{eqsam439}
\Sigma_p^\PP(\Reg_\DB) + \Sigma_p^\PP(\Reg_{\NDB_1})+\Sigma_p^\PP(\Reg_{\Neg}) \lesssim \Lambda^{\frac{-1}{12n}}\,\sigma_p(\HD_1).
\end{equation}
Also,
$$B\,\Lambda^{p-2} = \Lambda^{\frac1{100n}}\,\Lambda^{\frac{-1}{18n}} = \Lambda^{\frac{-41 }{900n}}
< \Lambda^{\frac{-1 }{25n}},$$
and thus
\begin{multline*}
\Sigma_p^\PP(\Reg_{\LD_1}\setminus\Reg_{\DB}) + \Sigma_p^\PP(\Reg_{\LD_2}) +
\Sigma_p^\PP(\Reg_{\NDB_2}) + \Sigma_p^\PP(\Reg_{\HD_2})+ \Sigma_p^\PP(\Reg_\Ot)\\
\lesssim \Lambda^{\frac{-1 }{25n}}\,
\sigma_p(\HD_1).
\end{multline*}
Gathering the estimates above, the lemma follows.
\end{proof}

\vv
Remark that in order to transfer the lower estimates for $\RR\eta$ to $\RR\mu$ we will need
to take $p$ very close to $1$ below, in particular with $p<p_0$.

Next lemma shows how one can estimate the $\QQ_\Reg$ coefficients in terms of the $\PP$ coefficients.
This is proven in \cite[Lemma 7.12]{DT}.
\vv

\begin{lemma}\label{lemregpq}
For all $p\in(1,\infty)$,
$$\Sigma_p^\QQ(\Reg)\lesssim \Sigma_p^\PP(\Reg).$$
\end{lemma}

\vv


\subsection{Transference of the lower estimates for $\RR\eta$ to $\Delta_{\wt \TT}\RR\mu$}
Denote
$$F = e'(R) \setminus \bigcup_{Q\in\Reg_{\HD_2}} Q.$$
By the splitting \rf{eqsplitreg0}, it is clear that $F$ coincides with the union of the cubes in the family
$$\Reg_F:=
 \Reg_{\LD_1} \cup \Reg_{\LD_2}\cup \Reg_{\NDB_1}\cup \Reg_{\NDB_2}\cup \Reg_{\Neg}\cup \Reg_{\Ot}.
$$
We also write
$$F_\eta= \bigcup_{Q\in\Reg_F} \frac12 B(Q),$$
so that the measure $\mu|_F$ is well approximated by $\eta|_{F_\eta}$, in a sense. 

Recall that in Lemma \ref{lemrieszeta} we showed that  
$$\int_{V_R} \big|(|\RR\eta(x)| - \frac{c_2}2\,\Theta(\HD_1))_+\big|^p\,d\eta(x)
 \gtrsim \Lambda^{-p'\ve_n}\sigma_p(\HD_1(e(R))),$$
for any $p\in (1,2]$, with $c_2$ be as in Lemma \ref{lemrieszeta}. 
We will show below that  a similar lower estimate holds if we restrict the integral on the 
left side to $HD_2$. The first step is the next lemma.

We let $\wt V_R$ be the union of the balls $\frac12B(Q)$, with $Q\in\Reg$, that intersect $V_R$.

\vv
\begin{lemma}\label{leminteta*}
Suppose that $R\in\Trc\cap\Ty$ and also that $\mu(Z)\leq \ve_Z\,\mu(R)$ and
$$\|\Delta_\wt\TT \RR\mu\|_{L^2(\mu)}^2\leq \Lambda^{-1}\,\sigma(\HD_1).$$
Then, for $p_0 = 2- \frac1{18n}$, assuming $\ve_Z\leq \Lambda^{-72n}$ and $\ell_0$ small enough,
$$\int_{F_\eta\cap\wt V_R}\big|(|\RR\eta| - \frac{c_2}4 \Theta(\HD_1))_+\big|^{p_0}
\,d\eta\lesssim \Lambda^{\frac{-1}{25n}}\,\sigma_{p_0}(\HD_1).$$
\end{lemma}

Remark that the proof of this lemma takes advantage of the good estimates we have
obtained for $\Sigma^\PP_{p_0}(\Reg)$ in  Lemma \ref{lemreg73}. Later on we will show that an analogous estimate holds for all $p\in (1,p_0)$ (see Lemma \ref{lemNZ} below).

\begin{proof}
Denote
$$F_a= \bigcup_{Q\in\wt \End\setminus \HD_2} Q,\qquad F_b= \bigcup_{Q\in\Reg_\Neg\setminus \Reg_{\sM_\Neg}} Q,\qquad F_\Ot= \bigcup_{Q\in\Reg_\Ot} Q,$$
so that 
$F=F_a\cup F_b\cup F_\Ot$.
Write also 
$$\End_a = \LD_1 \cup \LD_2 \cup \NDB_1  \cup \NDB_2  \cup\sM_\Neg$$
and
$$\Reg_a = \Reg_{\LD_1}\cup\Reg_{\LD_2}\cup\Reg_{\NDB_1}\cup\Reg_{\NDB_2}\cup\Reg_{\sM_\Neg},$$
so that
$$F_a= \bigcup_{Q\in \End_a} Q = \bigcup_{Q\in \Reg_a} Q.$$
We also consider
$$F_{a,\eta} = \bigcup_{Q\in \Reg_a} \frac12B(Q),\qquad F_{b,\eta}= \bigcup_{Q\in\Reg_\Neg\setminus \Reg_{\sM_\Neg}} \frac12B(Q),  \qquad F_{\Ot,\eta}= \bigcup_{Q\in\Reg_\Ot} \frac12B(Q),$$
so that these sets approximate $F_a,F_b,F_\Ot$ at the level of the family $\Reg$, in a sense. Moreover, we have $F_\eta=F_{a,\eta}\cup F_{b,\eta}\cup F_{\Ot,\eta}$.

We split
$$\int_{F_\eta\cap\wt V_R}\big|(|\RR\eta| - \frac{c_2}4 \Theta(\HD_1))_+\big|^{p_0}
\,d\eta = \int_{F_{a,\eta}\cap\wt V_R} \ldots +  \int_{F_{b,\eta}\cap\wt V_R} \ldots + \int_{F_{\Ot,\eta}\cap\wt V_R}\ldots =:I_a+ I_b + I_\Ot,$$
where ``$\ldots$'' stands for 
$\big|(|\RR\eta| - \frac{c_2}4 \Theta(\HD_1))_+\big|^{p_0}
\,d\eta$. 

\vv
\noi {\bf Estimates for $I_b$.}
We claim that $I_b=0$. Indeed, given $Q\in\Reg_\Neg\setminus \Reg_{\sM_\Neg}$, notice that all the cubes $S$ such that $Q\subset S\subset R$ satisfy
$$\Theta(S)\lesssim \left(\frac{\ell(S)}{\ell(R)}\right)^{1/2}\,\Theta(R),$$
and so, for any $x\in Q$,
\begin{align*}
|\RR\eta(x)| &\lesssim \sum_{S:Q\subset S\subset R} \theta_\eta(4B_S) \lesssim 
\sum_{S:Q\subset S\subset R} \theta_\mu(CB_S) \\
&\lesssim \sum_{S:Q\subset S\subset R}\left(\frac{\ell(S)}{\ell(R)}\right)^{1/2}\,\Theta(R) + \Theta(R)\lesssim \Theta(R).
\end{align*}
Hence, for $\Lambda$ big enough, $(|\RR\eta(x)| - \frac{c_2}4 \Theta(\HD_1))_+=0$, which proves our claim.

\vv
\noi {\bf Estimates for $I_a$.} 
By Lemma \ref{lemaprox1}, for all $Q\in\Reg_a$ such that $\frac12B(Q)\subset\wt V_R$, all $x\in \frac12B(Q)$, and all $y\in Q$,
$$|\RR\eta(x)| \leq |\RR_{\TT_\Reg}\mu(y)| + C\Theta(R) + C\PP(Q) + C\QQ_\Reg(Q).$$
Thus, for $\Lambda$ big enough, since $C\Theta(R)=\Lambda^{-1}\Theta(\HD_1)<\frac{c_2}8 \Theta(\HD_1)$,
$$(|\RR\eta(x)| - \frac{c_2}4 \Theta(\HD_1))_+\leq
(|\RR_{\TT_\Reg}\mu(y)| - \frac{c_2}8 \Theta(\HD_1))_+  + C\PP(Q) + C\QQ_\Reg(Q).$$
Consequently,
\begin{align}\label{eqalgjx2}
I_a & \lesssim \sum_{Q\in\Reg_a} \int_Q \big|(|\RR_{\TT_\Reg}\mu(y)| - 
\frac{c_2}8 \Theta(\HD_1))_+\big|^{p_0}\,d\mu(y) + \!\sum_{Q\in\Reg_a} (\PP(Q)^{p_0} + \QQ_\Reg(Q)^{p_0})\,\mu(Q)\\
& \lesssim \sum_{S\in\End_a}\int_S \big|(|\RR_{\TT_\Reg}\mu| - 
\frac{c_2}8 \Theta(\HD_1))_+\big|^{p_0}\,d\mu + \Sigma_{p_0}^\PP(\Reg),\nonumber
\end{align}
where, in the last inequality, we used the fact that all the cubes from $\Reg_a$
are contained in some cube $S\in\End_a$ and we applied Lemma \ref{lemregpq}.

For each $S\in\End_a$, by the triangle inequality and the fact that $(\;\cdot\;)_+$ is a $1$-Lipschitz function, we get 
\begin{align*}
\int_S \big|(|\RR_{\TT_\Reg}\mu| - 
\frac{c_2}8 \Theta(\HD_1))_+\big|^{p_0}\,d\mu &\lesssim 
\int_S \big|(|\RR_{\wt \TT}\mu| - 
\frac{c_2}8 \Theta(\HD_1))_+\big|^{p_0}\,d\mu \\
&\quad+ \int_S\big|\RR_{\wt\TT}\mu - \RR_{\TT_\Reg}\mu\big|^{p_0}\,d\mu
\end{align*}
By Lemma \ref{lemaprox3}, the last integral does not exceed 
$C\EE(2S)^{\frac{p_0}2} \,\mu(S)^{1-\frac{p_0}2}$, and thus we deduce that
\begin{equation}\label{eqIa1}
I_a\lesssim
\sum_{S\in\End_a}\int_S \big|(|\RR_{\wt\TT}\mu| - 
\frac{c_2}8 \Theta(\HD_1))_+\big|^{p_0}\,d\mu + \sum_{S\in\End_a}\EE(2S)^{\frac{p_0}2} \,\mu(S)^{1-\frac {p_0}2} + \Sigma_{p_0}^\PP(\Reg).
\end{equation}

Next we apply Lemma \ref{lemaprox2}, which ensures that
for any $S\in\End_a\subset\wt\End$ and all $x\in S$,
\begin{equation}\label{eqIa2}
|\RR_{\wt\TT}\mu(x)| \leq |\Delta_{\wt\TT}\RR\mu(x)| + C\left(\frac{\EE(4R)}{\mu(R)}\right)^{1/2} +  C\left(\frac{\EE(2S)}{\mu(S)}\right)^{1/2}.
\end{equation}
In case that $R\not\in\DB$, recalling that
$\Lambda  = M_0^{\frac{8n-1}{8n-2}}\gg M_0$ by \rf{eqlammm}, for  $\Lambda$ big enough we obtain
\begin{equation}\label{eqEr4}
C\left(\frac{\EE(4R)}{\mu(R)}\right)^{1/2} \leq C\,M_0\,\Theta(R)\leq \frac{c_2}8 \Theta(\HD_1).
\end{equation}
In case that $R\in\DB$, since $R\in\Ty$, we have 
$$
\EE(4R)\lesssim
\sum_{Q\in\DB:Q\sim\TT} \EE_\infty(9Q)\leq \Lambda^{\frac{-1}{3n}}\,\sigma(\HD_1),
$$
and so we also get
\begin{equation}\label{eqEr5}
C\left(\frac{\EE(4R)}{\mu(R)}\right)^{1/2} \leq C \left(\frac{\Lambda^{\frac{-1}{3n}}\,\sigma(\HD_1)}{\mu(R)}\right)^{1/2}\leq C\, \Lambda^{\frac{-1}{6n}} \,\Theta(\HD_1)
\leq \frac{c_2}8 \Theta(\HD_1).
\end{equation}
Thus, in any case,
$$(|\RR_{\wt\TT}\mu(x)| - \frac{c_2}8 \Theta(\HD_1))_+ \leq 
|\Delta_{\wt\TT}\RR\mu(x)| + C\left(\frac{\EE(2S)}{\mu(S)}\right)^{1/2}.$$
Plugging this estimate into \rf{eqIa1}, we get
\begin{equation}\label{eqIa99}
I_a\lesssim
\sum_{S\in\End_a}\int_S |\Delta_{\wt\TT}\RR\mu|^{p_0}\,d\mu + \sum_{S\in\End_a}\EE(2S)^{\frac {p_0}2} \,\mu(S)^{1-\frac {p_0}2} + \Sigma_{p_0}^\PP(\Reg).
\end{equation}

We deal with each term on the right hand side of the preceding inequality separately. First, by H\"older's inequality and the assumptions in the lemma, we have
$$\int |\Delta_{\wt\TT}\RR\mu|^{p_0}\,d\mu \lesssim \|\Delta_{\wt\TT}\RR\mu\|_{L^2(\mu)}^{p_0}\,
\mu(R)^{1-\frac{p_0}2}\leq \Lambda^{-\frac{p_0}2}\,\sigma(\HD_1)^{\frac{p_0}2}\,\mu(R)^{1-\frac{p_0}2}.$$
Regarding the second term in \rf{eqIa99}, by H\"older's inequality again,
\begin{align*}
\sum_{S\in\End_a}\EE(2S)^{\frac{p_0}2} \,\mu(S)^{1-\frac {p_0}2} 
& \leq \bigg(\sum_{S\in\End_a}\EE(2S)\bigg)^{\frac{p_0}2} \,\bigg(\sum_{S\in\End_a} \mu(S)\bigg)^{1-\frac {p_0}2} \\& \lesssim  \bigg(\sum_{S\in\End_a}\EE(2S)\bigg)^{\frac{p_0}2} \,\mu(R)^{1-\frac {p_0}2}.
\end{align*}
We estimate the first factor on the right hand side using \rf{eqlem*4}, \rf{eqlem*4.5}, and \rf{eqlemneg02}:
\begin{align*}
\bigg(\sum_{S\in\End_a}\EE(2S)\bigg)^{\frac{p_0}2} & \leq \bigg(\sum_{S\in\NDB_1\cup\NDB_2}\EE_\infty(9S) +
\sum_{S\in
\LD_1\cup \LD_2} \EE_\infty(9S) + \sum_{S\in\sM_\Neg} \EE_\infty(9S)\bigg)^{\frac{p_0}2} \\
& \lesssim \big(B\Lambda^{-1}\!+\Lambda^{\frac{-1}{3n}} + B\,M_0^2\,\delta_0^{\frac2{n+2}} + B\,\Lambda^7\delta_0\big)^{\frac{p_0}2}\,\sigma(\HD_1)^{\frac{p_0}2}
\leq \Lambda^{\frac{-1}{6n}}\,\sigma(\HD_1)^{\frac{p_0}2},
\end{align*}
by the assumption \rf{eqassu78} on $\delta_0$.
In connection with the last summand on the right hand side of \rf{eqIa99}, by Lemma \ref{lemreg73}
we have
$$\Sigma_{p_0}^\PP(\Reg)\lesssim \Lambda^{\frac{-1}{25n}}\,\sigma_{p_0}(\HD_1).$$
Therefore,
$$I_a\lesssim \Lambda^{\frac{-1}{6n}}\,\sigma(\HD_1)^{\frac{p_0}2}\,\mu(R)^{1-\frac{p_0}2} + \Lambda^{\frac{-1}{25n}}\,\sigma_{p_0}(\HD_1).$$
From \rf{eqmuhd1} we derive that 
$\mu(HD_1) \geq\Lambda^{-3}\,\mu(R)$, and then
\begin{equation}\label{eqhd1*p}
 \Lambda^{\frac{-1}{6n}}\,\sigma(\HD_1)^{\frac{p_0}2}\,\mu(R)^{1-\frac{p_0}2} 
\leq \Lambda^{\frac{-1}{6n}}\,\big(\Lambda^3\big)^{1-\frac {p_0}2}\,\sigma_{p_0}(\HD_1)=
\Lambda^{\frac{-1}{12n}}\,\sigma_{p_0}(\HD_1)\leq \Lambda^{\frac{-1}{25n}}\,\sigma_{p_0}(\HD_1).
\end{equation}
So we get
\begin{equation}\label{eqIa93}
I_a\lesssim \Lambda^{\frac{-1}{25n}}\,\sigma_{p_0}(\HD_1).
\end{equation}
\vv

\noi {\bf Estimate of $I_\Ot$.} 
By the same arguments as in \rf{eqalgjx2}, just replacing $\Reg_a$ by $\Reg_\Ot$, we obtain
$$I_\Ot  \lesssim \sum_{Q\in\Reg_\Ot} \int_Q \big|(|\RR_{\TT_\Reg}\mu(y)| - 
\frac{c_2}8 \Theta(\HD_1))_+\big|^{p_0}\,d\mu(y) + \!\sum_{Q\in\Reg_\Ot} (\PP(Q)^{p_0} + \QQ_\Reg(Q)^{p_0})\,\mu(Q).$$
Thus, using again Lemmas \ref{lemregpq} and \ref{lemreg73}, we get
\begin{align}\label{eqiot78}
I_\Ot  &\lesssim \sum_{Q\in\Reg_\Ot} \int_Q \big|(|\RR_{\TT_\Reg}\mu(y)| - 
\frac{c_2}8 \Theta(\HD_1))_+\big|^{p_0}\,d\mu(y) +\Lambda^{\frac{-1}{25n}}\,\sigma_{p_0}(\HD_1)\\
& =: \wt I_\Ot +
\Lambda^{\frac{-1}{25n}}\,\sigma_{p_0}(\HD_1).\nonumber
\end{align}
To estimate the integral $\wt I_\Ot$ on the right hand side, we split
\begin{align*}
\wt I_\Ot & =\sum_{Q\in\Reg_{\Ot}\setminus \Neg(e'(R))}
\int_Q \big|(|\RR_{\TT_\Reg}\mu(x)| - 
\frac{c_2}8 \Theta(\HD_1))_+\big|^{p_0}\,d\mu(x)\\
&\quad+ \sum_{Q\in\Reg_{\Ot}\cap \Neg(e'(R))}
\int_Q \big|(|\RR_{\TT_\Reg}\mu(x)| - 
\frac{c_2}8 \Theta(\HD_1))_+\big|^{p_0}\,d\mu(x)\\
& = \wt I_{\Ot,1} + \wt I_{\Ot,2}.
\end{align*}
Notice that, by definition we have $\Reg_\Ot\subset\TT$. 
Recall that $\Neg = \Neg(e'(R))\cap \End$ and thus we may have $\Reg_{\Ot}\cap \Neg(e'(R))\neq\varnothing$.
In this case, we have $\Reg_{\Ot}\cap \Neg(e'(R))\subset \TT_\sss(e'(R))$ (because $\Neg(e'(R))\subset \TT_\sss(e'(R))$ by construction).

The same argument used to show that $I_b=0$ shows
that 
$$\wt I_{\Ot,2} = 0.$$
To estimate $\wt I_{\Ot,1}$,
denote by $\sM_\Ot$ the family of maximal $\PP$-doubling cubes which are contained in some cube from $\Reg_{\Ot}\setminus \Neg(e'(R))$
and let
$$N_\Ot = \bigcup_{Q\in \Reg_{\Ot}\setminus \Neg(e'(R))} Q\setminus \bigcup_{P\in \sM_{\Ot}} P.$$
We claim that 
\begin{equation}\label{eqinclu827}
\sM_\Ot\subset \TT\quad \text{ and }\quad N_\Ot\subset Z.
\end{equation}
To check this, for a given $P\in\sM_\Ot$ with $P\subset Q\in\Reg_{\Ot}\setminus \Neg(e'(R))$, suppose there exists $S\in\End$ such that $S\supset P$. As $P$ is a contained in some $Q\in\Reg_{\Ot}\setminus \Neg(e'(R))$, we  have
$S\not\in \Neg$. Further, $S\subsetneq Q$ because $Q\in\Reg_{\Ot}$ implies that $Q\not\subset S$. Since $S$ is $\PP$-doubling, we deduce that $P=S$, by the maximality of $P$ as $\PP$-doubling cube contained in $Q$. An analogous argument shows that $N_\Ot\subset Z$.

By H\"older's inequality and \rf{eqlimot62}, for $\ell_0$ small enough we have
\begin{align*}
\wt I_\Ot &\leq \bigg(\sum_{Q\in\Reg_{\Ot}} \int_Q \big|(|\RR_{\TT_\Reg}\mu(x)| - 
\frac{c_2}8 \Theta(\HD_1))_+\big|^{2}\,d\mu(x)\big)^{\frac {p_0}2} \bigg(\sum_{Q\in\Reg_{\Ot}}\mu(Q)\bigg)^{1-\frac {p_0}2}\\
& \leq
\bigg(\sum_{P\in\sM_{\Ot}} \int_P \big|\RR_{\TT_\Reg}\mu\big|^{2}\,d\mu + \int_{N_\Ot} \big|\RR_{\TT_\Reg}\mu\big|^{2}\,d\mu
\bigg)^{\frac {p_0}2} \bigg(\mu(Z) + o(\ell_0)\bigg)^{1-\frac {p_0}2},
\end{align*}
with $o(\ell_0)\to 0$ as $\ell_0\to0$. 

Denote
$$\RR_{\sM_\Ot}\mu(x) = \sum_{P\in\sM_\Ot} \chi_P(x)\,\RR(\chi_{2R\setminus 2P}\mu)(x)$$
and
$$\Delta_{\sM_\Ot}\RR\mu(x) =
\sum_{P\in\sM_\Ot} \chi_P(x)\,\big(m_{\mu,P}(\RR\mu) - m_{\mu,2R}(\RR\mu)\big)
+ \chi_{Z}(x) \big(\RR\mu(x) -  m_{\mu,2R}(\RR\mu)\big)
.$$
Notice that, for $x\in P\in\sM_\Ot$ and $Q\in \Reg_\Ot\setminus \Neg(e'(R))$ such that $Q\supset P$, since there are no $\PP$-doubling cubes $P'$ such that $P\subsetneq P'\subset Q$,
$$\big|\RR_{\sM_\Ot}\mu(x) - \RR_{\TT_\Reg}\mu(x)\big| = |
\RR(\chi_{2 Q\setminus 2P}\mu)(x)|\lesssim \sum_{P: P\subset P'\subset Q} \Theta(P') \lesssim \PP(Q)\lesssim
\Lambda\,\Theta(\HD_1).$$
Almost the same argument shows also that, for $x\in N_\Ot$,
$$\big|\RR(\chi_{2R}\mu)(x) - \RR_{\TT_\Reg}\mu(x)\big| \lesssim
\Lambda\,\Theta(\HD_1).$$
Remark also that
$$\int_{N_\Ot} \big|\RR(\chi_{2R}\mu)\big|^{2}\,d\mu \leq 
\int_Z \big|\RR(\chi_{2R}\mu)\big|^{2}\,d\mu.$$
So we deduce that, for $\ell_0$ small enough,
$$\wt I_\Ot \lesssim \bigg(\sum_{P\in\sM_{\Ot}} \int_P \big|\RR_{\sM_\Ot}\mu\big|^{2}\,d\mu + \int_Z \big|\RR(\chi_{2R}\mu)\big|^{2}\,d\mu + \Lambda^2\,\Theta(\HD_1)^2\,\mu(R)
\bigg)^{\frac {p_0}2} \big(\ve_Z\,\mu(R)\big)^{1-\frac {p_0}2}.$$

Almost the same arguments as in Lemma \ref{lemaprox2} show that for $x\in P\in\sM_\Ot$,
$$\big|\RR_{\sM_\Ot}\mu(x) - \Delta_{\sM_\Ot}\RR\mu(x)\big| \lesssim \PP(R) + \left(\frac{\EE(4R)}{\mu(R)}\right)^{1/2} + \PP(P) +  \left(\frac{\EE(2P)}{\mu(P)}\right)^{1/2}$$
and that, for $x\in Z$,
$$\big|\RR(\chi_{2R}\mu)(x) - \Delta_{\sM_\Ot}\RR\mu(x)\big| \lesssim \PP(R) + \left(\frac{\EE(4R)}{\mu(R)}\right)^{1/2}.$$
Therefore, by \rf{eqEr4} and \rf{eqEr5} and taking into account that $\PP(P)\lesssim\Lambda\,\Theta(\HD_1)$ for 
$P\in\sM_\Ot$, we deduce
$$\wt I_\Ot \lesssim \bigg(\int |\Delta_{\sM_\Ot}\RR\mu|^2\,d\mu +
\sum_{P\in\sM_{\Ot}} \EE(2P) + \Lambda^2\,\Theta(\HD_1)^2\,\mu(R)
\bigg)^{\frac {p_0}2} \big(\ve_Z\,\mu(R)\big)^{1-\frac {p_0}2}.$$ 

By the orthogonality of the functions $\Delta_Q\RR\mu$, $Q\in\DD_\mu$, the assumptions in the lemma, 
and \rf{eqinclu827},
it is clear that
$$\int |\Delta_{\sM_\Ot}\RR\mu|^2\,d\mu\leq \|\Delta_{\wt\TT}\RR\mu\|_{L^2(\mu)}^2\leq \Lambda^{-1}\,\sigma(\HD_1).$$
On the other hand, since the tree $\TT$ is typical,
\begin{align*}
\sum_{P\in\sM_{\Ot}} \EE(2P) & \leq \sum_{P\in\sM_{\Ot}\setminus\DB} \EE(2P) + \sum_{P\in \TT\cap\DB} \EE(2P)
\leq M_0^2\,\sigma(\sM_{\Ot}) + \Lambda^{\frac{-1}{3n}}\,\sigma(\HD_1)\\
&
\leq \Lambda^2\,\Theta(\HD_1)^2\,\mu(R)+ \Lambda^{\frac{-1}{3n}}\,\sigma(\HD_1)\lesssim \Lambda^2\,\Theta(\HD_1)^2\,\mu(R).
\end{align*}
Thus, using also \rf{eqmuhd1},
\begin{align*}
\wt I_\Ot & \lesssim \Lambda^{\frac{-p_0}2}\,\sigma(\HD_1)^{\frac{p_0}2}\, \mu(R)^{1-\frac {p_0}2}+ \ve_Z^{1-\frac {p_0}2}\,\Lambda^{p_0}\,\Theta(\HD_1)^{p_0}\,\mu(R)\\
&
\lesssim 
 (B\Lambda^2)^{^{1-\frac {p_0}2}}\Lambda^{\frac{-p_0}2}\,\sigma_{p_0}(\HD_1)
+ \ve_Z^{1-\frac {p_0}2}\,B\,\Lambda^{2+{p_0}}\,\sigma_{p_0}(\HD_1).
\end{align*}
By the choice of $p_0$, $\Lambda$, and $B$ and the assumption $\ve_Z\leq \Lambda^{-72n},$ we have
$$\wt I_\Ot\lesssim \Lambda^{\frac{-1}2}\,\sigma_{p_0}(\HD_1) + \Lambda^{\frac{-1}{25n}}\,\sigma_{p_0}(\HD_1)
\lesssim \Lambda^{\frac{-1}{25n}}\,\sigma_{p_0}(\HD_1)
.$$
Together with \rf{eqiot78}, this yields
$$I_\Ot  \lesssim \Lambda^{\frac{-1}{25n}}\,\sigma_{p_0}(\HD_1).
$$
Gathering the estimates obtained for $I_a$ and $I_\Ot$, the lemma follows.
\end{proof}
\vv


\vv
\begin{lemma}\label{lemNZ}
Suppose that $R\in\Trc\cap\Ty$ and also that $\mu(Z)\leq \ve_Z\,\mu(R)$ and
$$\|\Delta_{\wt\TT} \RR\mu\|_{L^2(\mu)}^2\leq \Lambda^{-1}\,\sigma(\HD_1).$$
 Assume also that
$\ve_Z\leq \Lambda^{-72n}$ and $\ell_0$ is small enough.
Then
$$\eta\big(\big\{x\in \wt V_R: |\RR\eta(x)|> \frac{c_2}2 \Theta(\HD_1)\big\}\big)\lesssim \Lambda^{\frac{-1}{25n}}\,\mu(HD_1).$$
Also, for any $p\in (1,p_0)$,
$$\int_{F_\eta\cap\wt V_R}\big|(|\RR\eta| - \frac{c_2}2 \Theta(\HD_1))_+\big|^{p}
\,d\eta\lesssim \Lambda^{\frac{-1}{25n}}\,\sigma_{p}(\HD_1).$$
\end{lemma}

\begin{proof}
Denote 
$$A = \big\{x\in \wt V_R: |\RR\eta(x)|> \frac{c_2}2 \Theta(\HD_1)\big\}.$$
By the definition of $F_\eta$ we can split
$$\eta(A)\leq \eta(A\cap F_\eta) + \eta\bigg(\bigcup_{Q\in\HD_2} \frac12B(Q)\bigg).$$
Notice first that 
\begin{equation}\label{eqhd62}
\eta\bigg(\bigcup_{Q\in\HD_2} \frac12B(Q)\bigg) = \frac1{\Theta(\HD_2)^2}\,\sigma(\HD_2)
\leq \frac{B}{\Lambda^2\,\Theta(\HD_1)^2}\,\sigma(\HD_1) \leq \Lambda^{-1}\,\mu(HD_1).
\end{equation}
On the other hand, for $x\in A$ we have 
$$(|\RR\eta(x)| - \frac{c_2}4 \Theta(\HD_1))_+ \geq \frac{c_2}4 \Theta(\HD_1).$$
So, by Chebyshev and Lemma \ref{leminteta*},
\begin{align*}
\eta(A\cap F_\eta) &\lesssim \frac1{\Theta(\HD_1)^{p_0}} 
\int_{F_\eta\cap\wt V_R}\big|(|\RR\eta| - \frac{c_2}4 \Theta(\HD_1))_+\big|^{p_0}
\,d\eta\\
& \lesssim \frac{\Lambda^{\frac{-1}{25n}}\,\sigma_{p_0}(\HD_1)}{\Theta(\HD_1)^{p_0}}  = 
\Lambda^{\frac{-1}{25n}}\,\mu(HD_1),
\end{align*}
which, together with \rf{eqhd62}, proves the first assertion of the lemma.

For the second statement in the lemma we use H\"older and Lemma \ref{leminteta*} again:
\begin{align*}
\int_{F_\eta\cap\wt V_R}\big|(|\RR\eta| - \frac{c_2}2 \Theta(\HD_1))_+\big|^{p}
\,d\eta& = \int_{A\cap F_\eta}\big|(|\RR\eta| - \frac{c_2}2 \Theta(\HD_1))_+\big|^{p}
\,d\eta\\
& \leq \left(\int_{A\cap F_\eta}\!\big|(|\RR\eta| - \frac{c_2}4 \Theta(\HD_1))_+\big|^{p_0}
d\eta\right)^{\frac p{p_0}}\!\eta(A\cap F_\eta)^{1-\frac p{p_0}}\\
& \lesssim \big(\Lambda^{\frac{-1}{25n}}\,\sigma_{p_0}(\HD_1)\big)^{\frac p{p_0}}\,
\big(\Lambda^{\frac{-1}{25n}}\,\mu(HD_1)\big)^{1-\frac p{p_0}}\\
&= \Lambda^{\frac{-1}{25n}}\,\sigma_{p}(\HD_1).
\end{align*}
\end{proof}

\vv

Observe that, given $R\in\MDW\cap\Trc\cap\Ty$, from Lemmas \ref{lemrieszeta} and \ref{lemNZ}, under
the assumptions in those lemmas, we derive that 
\begin{align}\label{eqh2893} 
\int_{V_R\cap HD_2} \big|(|\RR\eta(x)| - &\frac{c_2}2\,\Theta(\HD_1))_+\big|^p\,d\eta(x)\\
 &\geq c\Lambda^{-p'\ve_n}\,\sigma_p(\HD_1) - C\Lambda^{\frac{-1}{25n}}\,\sigma_p(\HD_1)\approx
 \Lambda^{-p'\ve_n}\,\sigma_p(\HD_1),\nonumber
\end{align}
for $p\in (1,p_0]$, assuming that $\ve_n$ and $p$ are chosen so that
$p'\ve_n \ll \frac{1}{25n}.$
This is the main ingredient for the proof of the next lemma, which is the main result
of this section.
\vv

\begin{lemma}\label{lemalter*}
Let $R\in\MDW\cap\Trc\cap\Ty$. Let $\Lambda>0$ be big enough and suppose that
$\ve_Z\leq \Lambda^{-72n}$. 
Then one of the following alternatives holds:
\begin{itemize}
\item[(a)] $\mu(Z) > \ve_Z\,\mu(R)$, or\vv

\item[(b)] $\|\Delta_{\wt \TT} \RR\mu\|_{L^2(\mu)}^2> \Lambda^{-1}\,\sigma(\HD_1).$
\end{itemize}
\end{lemma}

\begin{proof}
Suppose that none of the alternatives holds. Then, by Lemmas  \ref{lemrieszeta} and \ref{lemNZ},
as in \rf{eqh2893}, we have
\begin{equation} \label{eqtrans736}
I_{\HD_2}:=\int_{V_R\cap HD_2} \big|(|\RR\eta(x)| - \frac{c_2}2\,\Theta(\HD_1))_+\big|^p\,d\eta(x)
\gtrsim
 \Lambda^{-p'\ve_n}\,\sigma_p(\HD_1),
\end{equation}
for all $p\in (1,p_0]$,
assuming that $\ell_0$ is chosen small enough and that 
\begin{equation}\label{eqass934}
{p'\ve_n} \leq \frac{1}{50n}.
\end{equation}
The appropriate values of $\ve_n$ and $p$ will be chosen at the end of the proof.

By arguments analogous to the ones we used to estimate the integral $I_a$ in the proof of Lemma
\ref{leminteta*} we will ``transfer'' the estimate for $\RR\eta$ in \rf{eqtrans736} to $\Delta_{\wt\TT}\RR\mu$, so that we will obtain a lower estimate for $\|\Delta_{\wt \TT} \RR\mu\|_{L^2(\mu)}$ which will contradict the assumption that (b) does not hold. 
Although some of the estimates below are very similar to the ones to obtain the inequality 
\rf{eqIa99} in the proof of Lemma
\ref{leminteta*}, we will include the full details here for the reader's convenience. 
On the other hand, an important differences between the proof of that lemma and the current proof is that in Lemma 
\ref{leminteta*} we took advantage of the fact that $p_0$ is close to $2$, and the estimates there would not work for the family $\Reg_{\HD_2}$, while in the arguments below it is  essential the fact that we are taking $p$ close to $1$ and the estimates work fine for the family $\Reg_{\HD_2}$, 
while they would fail for the family $\Reg_a$.

By Lemma \ref{lemaprox1}, for all $Q\in\Reg_{\HD_2}$ such that $\frac12B(Q)\subset\wt V_R$, all $x\in \frac12B(Q)$, and all $y\in Q$,
$$|\RR\eta(x)| \leq |\RR_{\TT_\Reg}\mu(y)| + C\Theta(R) + C\PP(Q) + C\QQ_\Reg(Q).$$
Thus, for $\Lambda$ big enough, since $\Theta(R)=\Lambda^{-1}\Theta(\HD_1)<\frac{c_2}4 \Theta(\HD_1)$,
$$(|\RR\eta(x)| - \frac{c_2}2 \,\Theta(\HD_1))_+\leq
(|\RR_{\TT_\Reg}\mu(y)| - \frac{c_2}4\, \Theta(\HD_1))_+  + C\PP(Q) + C\QQ_\Reg(Q).$$
Therefore,
$$
I_{\HD_2}  \lesssim \sum_{Q\in\Reg_{\HD_2}} \!\int_Q \big|(|\RR_{\TT_\Reg}\mu(y)| - 
\frac{c_2}4 \Theta(\HD_1))_+\big|^{p}\,d\mu(y) + \!\!\sum_{Q\in\Reg_{\HD_2}} \!\!\!\!(\PP(Q)^{p} + \QQ_\Reg(Q)^{p})\,\mu(Q).$$
By \rf{eqlem*6} and Lemma \ref{lemregpq}, we have
\begin{align*}
\sum_{Q\in\Reg_{\HD_2}} (\PP(Q)^{p} + \QQ_\Reg(Q)^{p})\,\mu(Q) & = 
\Sigma_p^\PP(\Reg_{\HD_2}) +\Sigma_p^\QQ(\Reg_{\HD_2})\\ 
&\lesssim
\sigma_p(\HD_2) + \Sigma^\QQ(\Reg)^{\frac p2}\,\mu(HD_2)^{1-\frac p2} \\ 
& \lesssim B\,\Lambda^{p-2}\,\sigma_p(\HD_1) + \Sigma^\PP(\Reg)^{\frac p2}\,\mu(HD_2)^{1-\frac p2}.
\end{align*}
Also, recalling that $\sigma(\HD_2)\leq B\,\sigma(\HD_1)$, we get
$\mu(HD_2)\leq B\,\Lambda^{-2}\,\mu(HD_1)$. Then, by Lemma \ref{lemreg73}, we obtain
$$\Sigma^\PP(\Reg)^{\frac p2}\,\mu(HD_2)^{1-\frac p2} \leq \big(B\,\sigma(\HD_1)\big)^{\frac p2}\,
\big(B\,\Lambda^{-2} \,\mu(HD_1)\big)^{1-\frac p2} = B\,\Lambda^{p-2}\,\sigma_p(\HD_1).
$$
So, since any cube from $\Reg_{\HD_2}$
is contained in some cube $S\in\HD_2$:
$$I_{\HD_2}  \lesssim \sum_{S\in\HD_2} \!\int_S \big|(|\RR_{\TT_\Reg}\mu(y)| - 
\frac{c_2}4 \Theta(\HD_1))_+\big|^{p}\,d\mu(y) + B\,\Lambda^{p-2}\,\sigma_p(\HD_1).$$

For each $S\in\HD_2$, by the triangle inequality and the fact that $(\;\cdot\;)_+$ is a $1$-Lipschitz function, we obtain
\begin{align*}
\int_S \big|(|\RR_{\TT_\Reg}\mu| - 
\frac{c_2}4 \Theta(\HD_1))_+\big|^{p}\,d\mu &\lesssim 
\int_S \big|(|\RR_{\wt \TT}\mu| - 
\frac{c_2}4 \Theta(\HD_1))_+\big|^{p}\,d\mu \\
&\quad+ \int_S\big|\RR_{\wt\TT}\mu - \RR_{\TT_\Reg}\mu\big|^{p}\,d\mu
\end{align*}
By Lemma \ref{lemaprox3}, the last integral does not exceed 
$C\EE(2S)^{\frac{p}2} \,\mu(S)^{1-\frac{p}2}$, and thus
\begin{equation}\label{eqIa1*}
I_{\HD_2}\lesssim\!
\sum_{S\in\HD_2}\int_S \big|(|\RR_{\wt\TT}\mu| - 
\frac{c_2}4 \Theta(\HD_1))_+\big|^{p}\,d\mu +\!\! \sum_{S\in\HD_2}\!\!\EE(2S)^{\frac{p}2} \mu(S)^{1-\frac {p}2} + B\,\Lambda^{p-2}\,\sigma_p(\HD_1).
\end{equation}

Next we apply Lemma \ref{lemaprox2}, which implies that
for any $S\in\HD_2$ and all $x\in S$,
\begin{equation}\label{eqIa2*}
|\RR_{\wt\TT}\mu(x)| \leq |\Delta_{\wt\TT}\RR\mu(x)| + C\left(\frac{\EE(4R)}{\mu(R)}\right)^{1/2} +  C\left(\frac{\EE(2S)}{\mu(S)}\right)^{1/2}.
\end{equation}
In case that $R\not\in\DB$, recalling that
$\Lambda  = M_0^{\frac{8n-1}{8n-2}}\gg M_0$ by \rf{eqlammm}, for  $\Lambda$ big enough we obtain
$$C\left(\frac{\EE(4R)}{\mu(R)}\right)^{1/2} \leq C\,M_0\,\Theta(R)\leq \frac{c_2}4 \Theta(\HD_1),$$
If $R\in\DB$, then we use the fact that $R\in\Ty$, which ensures that
$$\EE(4R)\lesssim \sum_{P\sim \TT:P\in\DB}\EE_\infty(9P) \leq \Lambda^{\frac{-1}{3n}}\,\sigma(\HD_1),$$
and so we also get
$$C\left(\frac{\EE(4R)}{\mu(R)}\right)^{1/2} \leq C \left(\frac{\Lambda^{\frac{-1}{3n}}\,\sigma(\HD_1)}{\mu(R)}\right)^{1/2}\leq C\, \Lambda^{\frac{-1}{6n}} \,\Theta(\HD_1)
\leq \frac{c_2}4 \Theta(\HD_1).$$
Hence, in any case we have
$$(|\RR_{\wt\TT}\mu(x)| - \frac{c_2}4 \Theta(\HD_1))_+ \leq 
|\Delta_{\wt\TT}\RR\mu(x)| + C\left(\frac{\EE(2S)}{\mu(S)}\right)^{1/2}.$$
Plugging this estimate into \rf{eqIa1*}, we get
\begin{equation}\label{eqIa99*}
I_{\HD_2}\lesssim
\int_{HD_2} |\Delta_{\wt\TT}\RR\mu|^{p}\,d\mu + \sum_{S\in\HD_2}\EE(2S)^{\frac{p}2} \,\mu(S)^{1-\frac {p}2} 
+  B\,\Lambda^{p-2}\,\sigma_p(\HD_1).
\end{equation}

Next we will estimate each term on the right hand side. First, by H\"older's inequality, we have
\begin{align*}
\int_{HD_2} |\Delta_{\wt\TT}\RR\mu|^{p}\,d\mu & \lesssim \|\Delta_{\wt\TT}\RR\mu\|_{L^2(\mu)}^{p}\,
\mu(HD_2)^{1-\frac{p}2}\\
&\leq \|\Delta_{\wt\TT}\RR\mu\|_{L^2(\mu)}^{p}\,(B\Lambda^{-2}\mu(HD_1))^{1-\frac{p}2}
\leq \|\Delta_{\wt\TT}\RR\mu\|_{L^2(\mu)}^{p}(\Lambda^{-1}\mu(HD_1))^{1-\frac{p}2}.
\end{align*}
Regarding the second term in \rf{eqIa99*}, by H\"older's inequality again,
\begin{align}\label{eqplug731}
\sum_{S\in\HD_2}\EE(2S)^{\frac{p}2} \,\mu(S)^{1-\frac {p}2} 
& \leq  \bigg(\sum_{S\in\HD_2}\EE(2S)\bigg)^{\frac{p}2} \,\mu(HD_2)^{1-\frac {p}2}.
\end{align}
We estimate now the first factor on the right hand side:
\begin{align*}
\sum_{S\in\HD_2}\EE(2S) & \leq \sum_{S\in\HD_2\setminus\DB}\EE(2S) + \sum_{S\sim\TT:S\in\DB}\EE(2S)\\
&\lesssim M_0^2\,\sum_{S\in\HD_2\setminus\DB}\sigma(S) + \Lambda^{\frac{-1}{3n}}\,\sigma(\HD_1)\\
&\leq M_0^2\,\sigma(\HD_2) + \Lambda^{\frac{-1}{3n}}\,\sigma(\HD_1)\\
& \leq B\,M_0^2\,\sigma(\HD_1) + \Lambda^{\frac{-1}{3n}}\,\sigma(\HD_1) \lesssim B\,M_0^2\,\sigma(\HD_1).
\end{align*}
Hence, plugging this estimate into \rf{eqplug731} and using that $\mu(HD_2)\leq B\,\Lambda^{-2}\mu(HD_1)$,
$$\sum_{S\in\HD_2}\EE(2S)^{\frac{p}2} \,\mu(S)^{1-\frac {p}2}\lesssim 
\big(B\,M_0^2\,\sigma(\HD_1)\big)^{\frac p2} \,\big(B\,\Lambda^{-2}\mu(HD_1)\big)^{1-\frac p2}
=M_0^p\, B\,\Lambda^{p-2}\,\sigma_p(\HD_1).$$

Altogether, we deduce that
$$I_{\HD_2}\lesssim \|\Delta_{\wt\TT}\RR\mu\|_{L^2(\mu)}^{p}\,(\Lambda^{-1}\mu(HD_1))^{1-\frac{p}2} + M_0^p\, B\,\Lambda^{p-2}\,\sigma_p(\HD_1).$$
Recalling the lower estimate for $I_{\HD_2}$ in \rf{eqtrans736}, we obtain
\begin{equation}\label{eqdelt634}
\|\Delta_{\wt\TT}\RR\mu\|_{L^2(\mu)}^{p}\,(\Lambda^{-1}\mu(HD_1))^{1-\frac{p}2} \geq
c\, \Lambda^{-p'\ve_n}\,\sigma_p(\HD_1) - C\,M_0^p\, B\,\Lambda^{p-2}\,\sigma_p(\HD_1).
\end{equation}
Recall now that 
$$M_0= \Lambda^{1-\frac1{8n-1}}\ll\Lambda.$$ 
Notice that for $p$ close enough to $1$, we have $M_0^p\, B\,\Lambda^{p-2}\ll1$, so that the last term on the right hand side of \rf{eqdelt634} is much smaller than the first one, assuming $\ve_n$ close enough to $0$. To be more precise, let us take 
$$p=1+ \frac1{4(8n-1)}.$$
A straightforward calculation gives $M_0^p\, \Lambda^{p-2} = \Lambda^{-\frac1{2(8n-1)} -\frac1{4(8n-1)^2}}$, so that
$$B\,M_0^p\, \Lambda^{p-2} \leq \Lambda^{\frac1{100n}}\,\Lambda^{-\frac1{2(8n-1)}} \leq 
\Lambda^{-\frac1{4(8n-1)}}.$$
Then we choose $\ve_n$ so that, besides \rf{eqass934}, it satisfies
$$\ve_n \leq \frac1{8(8n-1)}\,(p-1) = \frac1{32(8n-1)^2},$$
and we derive
$$\|\Delta_{\wt\TT}\RR\mu\|_{L^2(\mu)}^{p}\,(\Lambda^{-1}\mu(HD_1))^{1-\frac{p}2} \gtrsim
\Lambda^{-p'\ve_n}\,\sigma_p(\HD_1),$$
which is equivalent to
\begin{align*}
\|\Delta_{\wt\TT}\RR\mu\|_{L^2(\mu)}^{2} &\gtrsim
\big( \Lambda^{-p'\ve_n}\,
\,\Lambda^{1-\frac p2} 
\,\sigma_p(\HD_1)\,\mu(HD_1)^{\frac{p}2-1}\big)^{\frac p2}\\
& = \Lambda^{-\frac{2\ve_n}{p-1} +\frac p2-1}\,\sigma(\HD_1)\gg \Lambda^{-1}\,\sigma(\HD_1).
\end{align*}
This contradicts the assumption that the alternative (b) in the lemma does not hold.
\end{proof}

\vv


\section{The proof of Proposition \ref{propomain}}

We have to show that
$$\sum_{Q\in\DB} \EE_\infty(9Q) \leq C\big( \|\RR\mu\|_{L^2(\mu)}^2 + 
\theta_0^2\,\|\mu\|\big),$$
with $C$ possibly depending on $\Lambda$ and other parameters.
Recall that by Lemma \ref{lemsuper**}, we have
$$\sum_{Q\in\DB} \EE_\infty(9Q) \lesssim \Lambda^{\frac{-1}{2n}}(\log\Lambda)^2 
\sum_{R\in\sL(\GDF)}\,
\sum_{k\geq0} B^{-k/2} \sum_{Q\in\Trc_k(R)\cap\Ty}\sigma(\HD_1(e(Q))).$$
Also, by Lemma \ref{lemimp9}, it turns out that, for all $P\in\DD_\mu$ and all $k\geq0$,
$$\#\big\{R\in\sL(\GDF):\exists \,Q\in\Trc_k(R) \mbox{ such that } P\in\TT(e'(Q))\big\}\leq C_2\,(\log\Lambda)^2.$$
Observe now that, by Lemma \ref{lemalter*}, for each $Q\in\Trc_k(R)\cap\Ty$,
either
$$\sigma(\HD_1(e(Q)))\lesssim \theta_0^2\,\mu(Q)\leq \theta_0^2\,\ve_Z^{-1}\,\mu(Z(Q)),$$
where $Z(Q)$ is the set $Z$ appearing in Lemma \ref{lemsuper**} (replacing $R$ by $Q$ there), or
$$\sigma(\HD_1(e(Q))) \leq \Lambda\,\|\Delta_{\wt \TT(e'(Q))} \RR\mu\|_{L^2(\mu)}^2.$$
Therefore,
\begin{align}\label{eqsumtot93}
\sum_{Q\in\DB} \EE_\infty(9Q) & \lesssim_\Lambda 
\sum_{R\in\sL(\GDF)}\,
\sum_{k\geq0} B^{-k/2} \sum_{Q\in\Trc_k(R)\cap\Ty}
\|\Delta_{\wt \TT(e'(Q))} \RR\mu\|_{L^2(\mu)}^2\\
&\quad+\!\!\sum_{R\in\sL(\GDF)}\,
\sum_{k\geq0} B^{-k/2} \!\!\!\!\sum_{Q\in\Trc_k(R)\cap\Ty}\!
\theta_0^2\,\ve_Z^{-1}\,\mu(Z(Q))\nonumber\\
&=: T_1 + T_2.\nonumber
\end{align}

\vv

\subsection{Estimate of $T_1$}
Recall that 
\begin{align*}
\Delta_{\wt\TT(e'(Q))}\RR\mu(x) & = \sum_{P\in\wt\End(e'(Q))} \chi_P(x)\,\big(m_{\mu,P}(\RR\mu) - m_{\mu,2Q}(\RR\mu)\big)\\
&\quad +  \chi_{Z(Q)}(x) \big(\RR\mu(x) -  m_{\mu,2R}(\RR\mu)\big).
\end{align*}
For $Q\in\MDW$, we write $S\prec Q$ if $S\in\DD_\mu$ is a maximal cube contained in $e'(Q)$. Then we denote
$$\wh\Delta_Q\RR\mu = \sum_{S\prec Q} (m_{\mu,S}(\RR_\mu) - m_{\mu,2Q}(\RR\mu)\big)\,\chi_S.$$
Then, it is easy to check that
$$\Delta_{\wt\TT(e'(Q))}\RR\mu = \bigg(\sum_{P\in\wt\TT(e'(Q))\setminus \wt\End(e'(Q))} \Delta_P \RR\mu + \wh\Delta_Q\RR\mu\bigg)\,\chi_{\wt G(Q)},$$
where
$$\wt G(Q) = \bigcup_{P\in\wt\End(e'(Q))} P \,\cup \,Z(Q)$$
(see {\cite[Section 8.1]{DT} for the full details).
It is also immediate to check that, for a fixed $Q$, all the functions appearing on the right hand side are mutually orthogonal in $L^2(\mu)$.  Arguing as in \rf{eqnegsim*}, one sees that the cubes $P\in\wt\TT(e'(Q))$ satisfy 
$P\sim\TT(e'(Q))$. So we get
\begin{align*}
\|\Delta_{\wt \TT(e'(Q))} \RR\mu\|_{L^2(\mu)}^2 & =
\sum_{P\in\wt\TT(e'(Q))\setminus \wt\End(e'(Q))}\| \Delta_P \RR\mu\|_{L^2(\mu)}^2  + \|\wh\Delta_Q\RR\mu\|_{L^2(\mu)}^2 \\
& \leq \sum_{P\sim\TT(e'(Q))} \|\Delta_P \RR\mu\|_{L^2(\mu)}^2  + \|\wh\Delta_Q\RR\mu\|_{L^2(\mu)}^2.
\end{align*}
Therefore,
\begin{align*}
T_1 & \leq \sum_{R\in\sL(\GDF)}\,
\sum_{k\geq0} B^{-k/2} \sum_{Q\in\Trc_k(R)\cap\Ty}
\sum_{P\sim\TT(e'(Q))} \|\Delta_P \RR\mu\|_{L^2(\mu)}^2\\
&\quad + \sum_{R\in\sL(\GDF)}\,\sum_{k\geq0} B^{-k/2} \sum_{Q\in\Trc_k(R)\cap\Ty}
\|\wh\Delta_Q\RR\mu\|_{L^2(\mu)}^2\\
& =: T_{1,1} + T_{1,2}.
\end{align*}
By Fubini, regarding the term $T_{1,1}$, we have
\begin{align*}
T_{1,1}
 &\leq
\sum_{P\in \DD_\mu} \|\Delta_P \RR\mu\|_{L^2(\mu)}^2
 \sum_{k\geq0} B^{-k/2}\,
\# A(P,k),
\end{align*}
where
$$A(P,k)= 
\big\{R\in \sL(\GDF):\exists \,Q\in\Trc_k(R) \text{ such that }P\sim\TT(e'(Q))\big\}.$$
As shown in \rf{eqremaa95}, it holds
$$\#A(P,k)\lesssim 
(\log\Lambda)^2.$$
Hence,
$$T_{1,1}
 \lesssim_\Lambda
\sum_{P\in \DD_\mu} \|\Delta_P \RR\mu\|_{L^2(\mu)}^2\lesssim_\Lambda \|\RR\mu\|_{L^2(\mu)}^2.$$
Concerning $T_{1,2}$, we argue analogously:
\begin{align*}
T_{1,2}
 &\le 
\sum_{Q\in \DD_\mu^\PP} \|\wh\Delta_Q\RR\mu\|_{L^2(\mu)}^2
 \sum_{k\geq0} B^{-k/2}\,
\# \wt A(Q,k),
\end{align*}
where
$$\wt A(Q,k)= \big\{R\in \sL(\GDF):
Q\in\Trc_k(R)\big\}.$$
Since 
$$\#\wt A(Q,k) \leq\#A(Q,k)\lesssim(\log\Lambda)^2,$$
we deduce that
$$T_{1,2}
 \lesssim_\Lambda
\sum_{Q\in \MDW} \|\wh \Delta_Q \RR\mu\|_{L^2(\mu)}^2.$$
By Lemma 8.1 from \cite{DT},\footnote{In fact, the family $\MDW$ in \cite{DT} is not the same as the one in the current paper because in \cite{DT} $\MDW$ is a subfamily of roots of a corona decomposition of $\DD_\mu$. However, the reader can check that the proof of Lemma 8.1 from \cite{DT} works verbatim in our situation too.}
 the right hand side above is also bounded by $C\|\RR\mu\|_{L^2(\mu)}^2$.
So we have,
$$T_1
 \lesssim_\Lambda \|\RR\mu\|_{L^2(\mu)}^2.$$

\vv

\subsection{Estimate of $T_2$}

We have
\begin{align*}
T_2 & = \sum_{R\in\sL(\GDF)}\,
\sum_{k\geq0} B^{-k/2} \!\!\!\!\sum_{Q\in\Trc_k(R)\cap\Ty}\!
\theta_0^2\,\ve_Z^{-1}\,\mu(Z(Q))\\
& =
\sum_{R\in\sL(\GDF)}\,
\sum_{k\geq0} B^{-k/2} \!\!\!\!\sum_{Q\in\Trc_k(R)\cap\Ty}\!
\int_{Z(Q)}\theta_0^2\,\ve_Z^{-1}\,d\mu\\
&= \int\theta_0^2\,\ve_Z^{-1}\, \bigg(\sum_{R\in\sL(\GDF)}\,\sum_{k\geq0} B^{-k/2}\!\!\!
\sum_{Q\in\Trc_k(R)\cap\Ty}\!\chi_{Z(Q)}\bigg)\,d\mu.
\end{align*}
By Fubini, we have
$$\sum_{R\in\sL(\GDF)}\,\sum_{k\geq0} B^{-k/2}\!\!\!
\sum_{Q\in\Trc_k(R)\cap\Ty}\!\chi_{Z(Q)} \leq \sum_{k\geq0} B^{-k/2} \,\# D(x,k),$$
where
$$D(x,k) = 
\big\{R\in \sL(\GDF):\exists \,Q\in\Trc_k(R) \text{ such that }x\in Z(Q)\big\}.
$$
Observe now that, given $j\ge1$, if we let
\begin{align*}D_j(x,k) = 
\big\{R\in \sL(\GDF): \exists \,Q\in&\Trc_k(R) \text{ such that $\TT(e'(Q))$ contains} \\
&\text{
every $P\in\DD_\mu$ such that $x\in P$ and $\ell(P)\leq A_0^{-j}$}\big\},
\end{align*}
then we have
$$D(x,k) = \bigcup_{j\geq 1} D_j(x,k),$$
and moreover $D_j(x,k)\subset D_{j+1}(x,k)$ for all $j$.
From Lemma \ref{lemimp9} we deduce that
$$\#D_j(x,k)\leq C\,(\log\Lambda)^2\quad \mbox{ for all $j\geq 1$.}$$
Thus, $\#D(x,k)\leq C\,(\log\Lambda)^2$ too. Consequently,
$$\sum_{R\in\sL(\GDF)}\,\sum_{k\geq0} B^{-k/2}\!\!\!
\sum_{Q\in\Trc_k(R)\cap\Ty}\!\chi_{Z(Q)} \lesssim_\Lambda 1,$$
and so
$$T_2\lesssim_\Lambda \ve_Z^{-1}
\theta_0^2\,\|\mu\|.$$
Together with the estimate we obtained for $T_1$, this yields
$$\sum_{Q\in\DB} \EE_\infty(9Q)\lesssim_\Lambda \|\RR\mu\|_{L^2(\mu)}^2 + \ve_Z^{-1}
\theta_0^2\,\|\mu\|,$$
which concludes the proof of Proposition \ref{propomain}.
\vv


\vvv


\begin{thebibliography}{MMNT}

\bibitem[AHM+]{AHM3TV}
J.~Azzam, S.~Hofmann, J.M. Martell, S.~Mayboroda, M.~Mourgoglou, X.~Tolsa, and
  A.~Volberg. \emph{Rectifiability of harmonic measure}. Geom. Funct. Anal. (GAFA), 26(3) (2016), 703--728. 
  

\bibitem[AMT]{AMT} J. Azzam, M. Mourgoglou and X. Tolsa. {\em Mutual absolute continuity of
interior and exterior harmonic measure implies rectifiability.}  Comm. Pure Appl. Math. Vol. LXX (2017), 2121--2163. 




\bibitem[AMTV]{AMTV} J. Azzam, M. Mourgoglou, X. Tolsa and A. Volberg. {\em On a two-phase problem for harmonic measure in general domains.} 
Amer. J. Math. 141(5) (2019), 1259--1279.

\bibitem[AT]{Azzam-Tolsa} J. Azzam and X. Tolsa. {\em Characterization of n-rectifiability in terms of Jones' square function: Part II.} Geom. Funct. Anal. 25 (2015), no. 5, 1371--1412.

\bibitem[DT]{DT} D. D\k{a}browski and X. Tolsa. {\em 
The measures with $L^2$-bounded Riesz transform satisfying a subcritical Wolff-type energy condition}. Preprint (2021).

\bibitem[DM]{David-Mattila} G. David and P. Mattila. {\em Removable sets for Lipschitz
harmonic functions in the plane.} Rev. Mat. Iberoamericana 16(1) (2000),
137--215.

\bibitem[DS1]{DS1} G. David and S. Semmes. {\em Singular integrals and
rectifiable sets in $\R^n$: Beyond Lipschitz graphs,} Ast\'{e}risque
No. 193 (1991).

\bibitem[DS2]{DS2} G. David and S. Semmes. \emph{Analysis of and on uniformly
rectifiable sets}, Mathematical Surveys and Monographs, 38. American
Mathematical Society, Providence, RI, (1993).

\bibitem[ENV]{ENV} V. Eiderman, F. Nazarov and A. Volberg. {\em The $s$-Riesz transform of an 
$s$-dimensional measure in $\R^2$ is unbounded for $1<s<2$.} 
J. Anal. Math. 122 (2014), 1--23. 

\bibitem[Gi]{Girela} D. Girela-Sarri\'on. {\em Geometric conditions for the $L^2$-boundedness of singular integral operators with odd kernels with respect to measures with polynomial growth in $R^d$.}  J. Anal. Math. 137 (2019), no. 1, 339--372. 

\bibitem[JN1]{JN1} B. Jaye and F. Nazarov. {\em Reflectionless measures for Calder\'on-Zygmund operators I: general theory.} J. Anal. Math. 135 (2018), no. 2, 599--638. 

\bibitem[JN2]{JN2} B. Jaye and F. Nazarov. {\em Reflectionless measures for Calder\'on-Zygmund operators II: Wolff potentials and rectifiability.} J. Eur. Math. Soc. (JEMS) 21 (2019), no. 2, 549–583.

\bibitem[JNRT]{JNRT} B. Jaye, F. Nazarov, and M.C. Reguera, and X. Tolsa. {\em The Riesz transform of codimension smaller than one and the Wolff energy}.  Mem.\ Amer.\ Math.\ Soc.  266 (2020), no.\ 1293.

\bibitem[JNT]{JNT} B. Jaye, F. Nazarov, and X. Tolsa. {\em  The measures with an associated square function operator bounded in $L^2$.}  Adv. Math. 339 (2018), 60--112.

\bibitem[Jo]{Jones} P.W.\ Jones. {\em Rectifiable sets and the travelling
salesman problem.} Invent. Math.\! 102 (1990), 1--15.

\bibitem[L\'e]{Leger} J.C. L\'eger. {\em Menger curvature and rectifiability.} Ann. of Math. 149 (1999), 831--869.

\bibitem[Ma]{Mattila-llibre} P. Mattila. {\em Geometry of sets and measures in
Euclidean spaces,} Cambridge Stud. Adv. Math. 44, Cambridge Univ.
Press, Cambridge, 1995.

\bibitem[Me]{Melnikov} M.S. Melnikov. {\em Analytic capacity: discrete
approach and curvature of a measure.} Sbornik: Mathematics
{186}(6) (1995), 827--846.

\bibitem[MV]{MV} M.S. Melnikov and J. Verdera. {\em A geometric proof of the
$L^2$ boundedness of the Cauchy integral on Lipschitz graphs.}
Internat. Math. Res. Notices (1995), 325--331.

\bibitem[MP]{Mattila-Paramonov} P. Mattila and P. V. Paramonov. {\em On geometric properties of harmonic Lip$_1$ capacity.} Pacific J. Math. 171(2) (1995), 469--491.

\bibitem[NToV1]{NToV1} F. Nazarov, X. Tolsa and A. Volberg. \emph{On the uniform
  rectifiability of {AD}-regular measures with bounded {R}iesz transform
  operator: the case of codimension 1}. Acta Math. {213} (2014), no.~2,
  237--321. 

\bibitem[NToV2]{NToV2} F. Nazarov, X. Tolsa and A. Volberg. {\em 
The Riesz transform,
rectifiability, and removability for
  Lipschitz harmonic functions.}
Publ. Mat. 58 (2014), 517--532.

\bibitem[NTrV1]{NTrV1} F. Nazarov, S. Treil, and A. Volberg. {\em Cauchy integral and Calder\'on–Zygmund operators on nonhomogeneous spaces.} Internat. Math. Res. Notices 15 (1997), 703--726.

\bibitem[NTrV2]{NTrV2} F. Nazarov, S. Treil, and A. Volberg. {\em The $Tb$-theorem on non-homogeneous spaces.} Acta Math. 190 (2) (2003).

\bibitem[Pa]{Paramonov} P.V. Paramonov. {\em Harmonic approximations in the {$C^1$}-norm}. Mat. Sb. 181 (1990), no. 10, 1341--1365.

\bibitem[RT]{Reguera-Tolsa} M.C. Reguera and X. Tolsa. {\em Riesz transforms of non-integer homogeneity on uniformly disconnected sets}. Trans. Amer. Math. Soc. 368 (2016), no. 10, 7045--7095. 

\bibitem[To1]{Tolsa-duke} X. Tolsa. {\em $L^2$-boundedness of the Cauchy
integral operator for continuous measures.} Duke Math. J. { 98}(2) (1999),
269-304.

\bibitem[To2]{Tolsa-sem} X. Tolsa. {\em Painlev\'{e}'s problem and the semiadditivity of analytic capacity}, Acta Math. 190:1 (2003), 105--149.

\bibitem[To3]{Tolsa-bilip} X. Tolsa. {\em Bilipschitz maps, analytic
capacity, and the Cauchy integral}. Ann. of Math. 162:3 (2005), 1241--1302.
 
\bibitem[To4]{Tolsa-memo} X. Tolsa. {\em Rectifiable measures, square functions involving densities, and the Cauchy transform.}
Mem. Amer. Math. Soc. {245} (2016), no.~1158, 1--130.

\bibitem[To5]{Tolsa-llibre} X. Tolsa. {\em Analytic capacity, the Cauchy transform, and non-homogeneous Calder\'on-Zygmund theory.} Progress in Mathematics, vol. 307, Birkh\"auser/Springer, Cham, 2014.

\bibitem[Vo]{Volberg} A.\ Volberg, {\em Calder\'on--Zygmund capacities and operators on nonhomogeneous spaces.}
CBMS Regional Conf. Ser. in Math. 100, Amer. Math. Soc., Providence, 2003.
\end{thebibliography}
\end{document}